\documentclass[11pt,a4paper]{article}
\usepackage{fancyhdr}
\usepackage{amsmath}
\usepackage{amsthm}
\usepackage{amssymb}
\usepackage{wasysym}
\usepackage{paralist}
\usepackage{makeidx}
\usepackage[all]{xy}
\usepackage{mathdots}
\usepackage{yhmath}
\usepackage{mathtools}

\usepackage{young}
\usepackage[vcentermath]{youngtab}

\usepackage{graphicx}
\usepackage{epstopdf}

\input xy

\newcommand{\nc}{\newcommand}
 
 \nc{\cl}{\centerline}
 
 \nc{\SL}{{\rm SL}}
 \nc{\hatQ}{{\hat Q}}
 \nc{\hatb}{{\hat b}}
 \nc{\sgn}{{\rm sgn}}
 \nc{\alg}{{\rm alg}}
 
 \nc{\barV}{{\overline V}}
 
 \nc{\hatY}{{\hat Y}}
 
 \nc{\bareta}{{\bar\eta}}
 
  \nc{\barzeta}{{\bar\zeta}}

   \nc{\barQ}{{\overline Q}}
 
 \nc{\hyperalgebra}{{\rm hyperalgebra }}

\newcommand{\JI}{{\rm JI}}

\nc{\barv}{{\bar v}}

\newcommand{\can}{{\rm can}}

\newcommand{\haty}{{\hat y}}

\newcommand{\bary}{{\bar y}}
 
 \nc{\ev}{{\rm ev}}
 
  \nc{\pt}{{\rm pt}}

 \nc{\seee}{\mathbb C}

  \newcommand{\st}{{\rm st}}
  
    \newcommand{\add}{{\rm add}}

 \nc{\hatlambda}{{\hat\lambda}}
 
 \nc{\daggerlambda}{{\lambda^\dagger}}

\newcommand{\bara}{{\bar a}}
\newcommand{\barb}{{\bar b}}

     \newcommand{\barL}{{\bar L}}

    \newcommand{\len}{{\rm len}}
    
\newcommand{\val}{{\rm val}}
     
  \newcommand{\height}{{\rm ht}}

\nc\diag{{\rm diag}}
\renewcommand{\vert}{{\,|\,}}
\nc{\hatL}{{\hat L}}

\nc{\barE}{{\bar   E}}
\nc{\D}{{\mathcal D}}
\nc{\E}{{\mathcal E}}
\nc{\F}{{\mathcal F}}
\nc{\FF}{{\mathcal F}}
\nc{\I}{{\mathcal I}}
\nc{\even}{{\rm e}}
\nc{\ep}{\epsilon}
\nc{\odd}{{\rm o}}
\nc{\Coker}{{\rm Coker}}
\nc{\olE}{{\overline E}}
\nc{\indBG}{{\rm ind}_B^G\,}
\nc{\indHG}{{\rm ind}_H^G\,}

\nc{\que}{{\mathbb Q}}
\nc{\barlambda}{{\bar\lambda}}
\nc{\barmu}{{\bar\mu}}
\nc{\barnu}{{\bar\nu}}
\nc{\bartau}{{\bar\tau}}
\nc{\barm}{{\bar m}}
\nc{\divind}{{\rm div.ind}}
\nc{\tl}{{\tilde{\lambda}}}
\nc{\dar}{\downarrow}

\nc{\Sym}{{\rm Sym}}
\nc{\Symm}{{\rm Sym}}

\newcommand{\q}{\quad}

\newcommand{\real}{{\mathbb R}}

\renewcommand{\mod}{{\rm mod}}

\newcommand{\St}{{\rm St}}
\newcommand{\Sp}{{\rm Sp}}
\newcommand{\Lie}{{\rm Lie}}
\newcommand{\bs}{\bigskip}

\renewcommand{\vert}{\,|\,}

\renewcommand{\sgn}{{\rm sgn}}

\xyoption{all}

\newcommand{\ind}{{\rm ind}}
\renewcommand{\vert}{\,|\,}
 
\newcommand{\zed}{{\mathbb Z}}
\newcommand{\Ext}{{\rm Ext}}
\newcommand{\End}{{\rm End}}
\newcommand{\Hom}{{\rm Hom}}

\renewcommand{\mod}{{\rm mod}}

\newcommand{\GL}{{\rm GL}}

\renewcommand{\mod}{{\rm{mod}}}
\renewcommand{\char}{{\rm char}}

\nc{\geom}{{\rm geom}}
\nc{\rep}{{\rm rep}}

\newcommand{\hy}{{\rm hy}}

  \newcommand{\chat}{{\hat{c}}}

 \renewcommand{\v}{v} 
 \newcommand{\w}{w}

     \newcommand{\x}{{x_i^{\rm st}}} 
\newcommand{\y}{{y_j^{\rm st}}} 
  \newcommand{\z}{{z_k^{\rm st}}} 
    
      \newcommand{\px}{{x_i^{\rm pt}}}

\newtheorem{definition}{Definition}[section]
\newtheorem{proposition}[definition]{Proposition}
\newtheorem{theorem}[definition]{Theorem}
\newtheorem{lemma}[definition]{Lemma}
\newtheorem{corollary}[definition]{Corollary}
\newtheorem{example}[definition]{Example}

\newtheorem{remark}[definition]{Remark}

\begin{document}


\centerline{\bf First degree cohomology of  Specht modules  and}

\centerline{\bf   extensions of symmetric powers.}

\bigskip

\centerline{Stephen Donkin  and Haralampos Geranios}

\bigskip

{\it Department of Mathematics, University of York, York YO10 5DD}\\

\medskip

{\tt stephen.donkin@york.ac.uk,  haralampos.geranios@york.ac.uk}

\bs

\bs

\centerline{24 March   2017}
\bs\bs\bs

\section*{Abstract}

\q  Let $\Sigma_d$ denote the symmetric group of degree  $d$ and let  $K$   be a field of positive characteristic  $p$. For $p>2$ we give an explicit description of the first cohomology group    $H^1(\Sigma_d,\Sp(\lambda))$, of the Specht module $\Sp(\lambda)$ over $K$,  labelled by a partition $\lambda$ of $d$. 

\q We also give a sufficient condition for the cohomology to be non-zero for $p=2$ and we find a lower bound for the dimension.

\q The cohomology of Specht modules has been considered in many papers including  \cite{Hem}, \cite{HN2}, \cite{KN} and  \cite{Web}.

\q Our method is to proceed by comparison with the cohomology for the   general linear group $G(n)$ over $K$ and then to reduce to the calculation of\\
  $\Ext^1_{B(n)}(S^d E,K_\lambda)$, where $B(n)$ is a  Borel subgroup of $G(n)$, where $S^dE$ denotes the $d$th symmetric power of the natural module $E$ for $G(n)$ and  $K_\lambda$ denotes the one dimensional $B(n)$-module with weight $\lambda$.

\q  The main new input is the description of module extensions by: extensions sequences, coherent triples of extension sequences and coherent  multi-sequences of extension sequences,  and the  detailed calculation of the possibilities for such sequences. These sequences  arise from the action of    divided powers elements  in the negative part of the hyperalgebra of $G(n)$. 

\q Our methods are valid also in the quantised context and we aim to  treat this in a separate paper.

\section*{Introduction}

\q   Let $K$ be an algebraically closed field and let $p$ be the characteristic of $K$. Let $n$ be a positive integer.   We write $G(n)$ for the general linear group over $K$ of degree $n$, write $B(n)$ for the Borel subgroup consisting of lower triangular matrices and $T(n)$ for the maximal torus consisting of diagonal matrices.
The weights of a $G(n)$-module or a $B(n)$-module (or $T(n)$-module)  will be computed with respect to $T(n)$. 

\q We set $X(n)={\Bbb Z}^n$ and let $X^+(n)$ be the set of dominant weights, i.e., the set of $\lambda=(\lambda_1,\ldots,\lambda_n)\in X(n)$ such that $\lambda_1\geq \lambda_2\geq \cdots\geq \lambda_n$.  We write $\Lambda(n)$ for the set of polynomial weights, i.e., the set of $\lambda=(\lambda_1,\ldots,\lambda_n)\in X(n)$ such that $\lambda_1,\ldots,\lambda_n\geq 0$ and we write $\Lambda^+(n)$ for the set of dominant polynomial weights, i.e., $X^+(n)\bigcap \Lambda(n)$.  The degree of a polynomial weight $\lambda=(\lambda_1,\ldots,\lambda_n)$ is $\lambda_1+\cdots+\lambda_n$. For $d\geq 0$ we  write $\Lambda(n,d)$ for the set of polynomial weights of degree $d$ and $\Lambda^+(n,d)$ for $\Lambda(n,d)\bigcap X^+(n)$, the set of polynomial dominant weights of degree $d$.
For $1\leq r\leq n$ we define $\ep_r=(0,\ldots,0,1,0,\ldots,0)\in X(n)$  (with $1$ appearing in the $r$th position).

\q We write $E$ for the natural module and $S^dE$ for its $d$th symmetric power, for $d\geq 0$.  For $\lambda\in X(n)$ we write $K_\lambda$ for the one dimensional $B(n)$-module with weight $\lambda$.  For each $\lambda\in \Lambda^+(n)$ we have the induced module $G(n)$-module  $\nabla(\lambda)=\ind_{B(n)}^{G(n)}K_\lambda$.  In particular we have $\nabla(d,0,\ldots,0)=S^d E$.    Moreover, by   a result of    Cline, Parshall, Scott, van der Kallen,  see e.g., \cite{RAG}   II, 4.7 Corollary a)  we have the following.

\begin{theorem}  For  $i\geq 0$ and $\lambda,\mu\in \Lambda^+(n)$  we have  
$$\Ext^i_{G(n)}(\nabla(\mu),\nabla(\lambda))=\Ext^i_{B(n)}(\nabla(\mu),K_\lambda)$$
and   in particular  
$$\Ext^i_{G(n)}(S^dE,,\nabla(\lambda))=\Ext^i_{B(n)}(S^dE,K_\lambda)$$
if $\lambda\in \Lambda^+(n,d)$.

\end{theorem}

\q We denote by $\Sigma_d$ the symmetric group of degree $d$.   For each partition $\lambda$ of $d$ we have the Specht module $\Sp(\lambda)$ over $K$, for $\Sigma_d$.

\q Recall that, for $n\geq d$ and $\lambda,\mu\in \Lambda^+(n,d)$ we have \\
$\Ext^i_{G(n)}(\nabla(\mu),\nabla(\lambda))=\Ext^i_{\Sigma_d}(\Sp(\mu),\Sp(\lambda))$,  for $i=0$ and  $p\neq 2$ and for  $i=1$  and $p\neq 2,3$ , by \cite{HTT}, Proposition 10.5 (i) and Proposition 10.6 (i), or  \cite{HN}. Our interest in this is paper is in degrees $0$ and $1$ so the relevant result for us is the following.

\begin{proposition} For $\lambda\in \Lambda^+(n,d)$ we have:

(i) $\Hom_{G(n)}(S^d E,\nabla(\lambda))=H^0(\Sigma_d,\Sp(\lambda))$; and 

(ii) $\dim H^1(\Sigma_d,\Sp(\lambda))\geq \dim \Ext^1_{G(n)}(S^d E,\nabla(\lambda))$, with equality if $p\neq 2$.

\end{proposition}

\q  This follows in most cases by the above results or  the paper  \cite{KN} of Kleshchev and Nakano.  We also give a proof, via the arguments of \cite{HTT}, Section 10, in Appendix I.

\q The detailed description of $\Ext^1_{B(n)}(S^dE, K_\lambda)$ is quite complicated in general. However there are a couple of striking consequences that we now give. Throughout the analysis the notions of what we call James pairs and James partitions are of paramount importance.  We suppose that $K$ has prime characteristic $p$. Let $(a,b)$ be a two part partition. Let $v=\val_p(a+1)$ be the $p$-adic valuation of $a+1$. Let $\beta$ denote the $p$-adic length of $b$, i.e., $\beta$ is the non-negative integer such that $b=\sum_{i=0}^\beta b_ip^i$, with $0\leq b_i<p$ and $b_\beta\neq 0$. We call $(a,b)$ a James pair if $\beta<v$ (equivalently $b<p^v$, equivalently ${a+i\choose i}=0$, modulo $p$, for $1\leq i\leq b$). We call a partition $\lambda=(\lambda_1,\ldots,\lambda_n)$ of length $n\geq 2$ a James partition if $(\lambda_r,\lambda_{r+1})$ is James for all $1\leq r<n$.

\q Let $\lambda=(\lambda_1,\ldots,\lambda_n)$ be a James partition and let $l_i=\len_p(\lambda_i)$ and $v_i=\val_p(\lambda_i+1)$, for $1\leq i\leq n$.   We define the  {\em segments} of $\lambda$ to be  the equivalence classes of $\{1,\ldots,n\}$ for the relation $r=s$ if and only if $l_r=l_s$. For $1\leq r,s\leq n$ we define $r$ and $s$ to be adjacent if they belong  to the same segment or if $1<r<n$, $s=r+1$, and $r+1$ is the only element in its segment  and $\lambda_r=p^{v_{r-1}}-1$. We define the  $p$-segments of $\lambda$ to be  the equivalent classes of $\{1,\ldots,n\}$ for the equivalence relation generated by adjacency.

\q The following, obtained in Section 7,  is a complete description of the dimensions of the extension spaces for James partitions. 

\begin{theorem} Let $\lambda$ be a James partition or degree $d$  of length $n\geq 2$ and let $N\geq n$. 

(i) If $l_1=l_2$ then $\dim \Ext^1_{B(N)}(S^d E, K_\lambda)$  is   the number of $p$-segments of $\{1,\ldots,n\}$.

(ii) If $l_1>l_2$ then $\dim \Ext^1_{B(N)}(S^d E, K_\lambda)$  is one less than  the number of $p$-segments of $\{1,\ldots,n\}$.

\end{theorem}

\q In the non-James case the results are not so easy to describe but the following general result, Proposition 9.4(i) and (ii), is significant.

\begin{proposition} Let $\lambda=(\lambda_1,\ldots,\lambda_n)$ be a non-James partition of length $n$ and let $N\geq n$.  Let  $r$ be minimal such that $(\lambda_r,\lambda_{r+1})$ is not James. 

(i) If $n\geq r+2$ then there is an injective linear map from $\Ext^1_{B(N)}(S^dE, K_\lambda)$ to $\Ext^1_{B(N)}(S^{d_0}E,K_\mu)$, where $\mu$ is the $3$-part partition $(\lambda_r,\lambda_{r+1},\lambda_{r+2})$ and $d_0$ is the degree of $\mu$;

(ii) if $r=n-1$ then there is an injective linear map from $\Ext^1_{B(N)}(S^dE, K_\lambda)$ to $\Ext^1_{B(N)}(S^{d_0}E,K_\mu)$, where $\mu$ is the $2$-part partition $(\lambda_{n-1},\lambda_n)$ and $d_0$ is the degree of $\mu$.
\end{proposition}

\q This leads to the following, see Theorem 12.1. 

\begin{theorem} Let $\lambda$ be a non-James partition or degree $d$  of length $n\geq 2$ and let $N\geq n$.  Then we have 

$$\dim \Ext^1_{B(N)}(S^d E, K_\lambda)\leq 1.$$

\end{theorem}

\q

\q In Theorems 12.30 and 12.31 we describe all non-James, non-split  partitions, i.e. partitions $\lambda$ with $\dim \Ext^1_{B(N)}(S^d E, K_\lambda)=1$ (where $d$ is the degree of $\lambda$ and $N$ is at least the number of parts of $\lambda$).

\q Via Proposition 0.2,  these results also give the dimensions of the first  cohomology  spaces of Specht modules for  $p\neq 2$.

\q We mention some of the other highlights of the paper.   

\q It is important for us to establish a close relationship between modules for the hyperalgebra of $B(n)$ and rational $B(n)$-modules. What we do is in fact to give some version of Verma's Conjecture for $B(n)$-modules. Since the arguments also give a proof of Verma's Conjecture (for all semisimple, simply connected algebraic groups) we include, in Section 1, a new proof (for earlier proofs see the papers by Sullivan, \cite{sullivansimply} and Cline, Parshall and Scott, \cite{CPS}).

\q We give, in Section 3, a simple proof of James's Theorem on the fixed points of Specht modules, \cite{James}, 24.4 Theorem, via our formalism in terms of divided powers operators. 

\q Our investigation into the nature of coherent multi-sequence starts in Section 5, where we determine all extension sequences for two part partitions.  As a corollary we recover the result of Erdmann, \cite{Erdmann}, (3.6) Theorem, describing the first extension groups between Weyl modules for ${\rm SL}_2(K)$.

\q As a by-product of  our approach to extension spaces via coherent multi-sequences, we give, in Section 9, a short proof of the main result of Weber, \cite{Web}, Theorem 1.2, which says that $H^1(\Sigma_d,\Sp(\lambda))=0$ if the partition $\lambda$ contains two distant non-James pairs.

\q In Section 13, using our main result on extensions, we give a counterexample to a conjecture of Hemmer on the cohomology of Specht modules.

\q We now describe the layout of the paper. In Section 1 we describe the link between modules for the hyperalgebra and rational modules, which is the starting point of our investigation.  In Secton 2 this is taken further to describe relations between the divided powers of operators and to give a description of the hyperalgebra by generators and relations.

\q In Section 3 we give, as a warm-up exercise in the use of the divided powers operations, a proof of James's Theorem on invariants of Specht modules.

\q In Section 4 we define extension multi-sequences  make explicit the link between certain extensions of $B(n)$-modules and coherent multi-sequences. 

\q In Section 5 we determine the basic extension sequences, i.e, those sequence arising via two part partitions.   The result is  that  a two part partition falls into one of three types: it  is either split,  or James,  or what we call pointed.  This trichotomy is the basis of the analysis for $3$-part  partitions which we give in later sections.

\q In Section 6 we give the definition of a coherent triple in terms of certain relations between basic extension sequences. These relations are used many times in the sections which follow, particularly to determine extension groups for $3$-part partitions.

\q In Section 7 we first determine all coherent triples for a James $3$-part partition. We then use this to determine the dimension of $\Ext^1_{B(n)}(S^dE,K_\lambda)$ (and hence $H^1(\Sigma_d,\Sp(\lambda)))$, for any  James partition $\lambda$. We find in particular that if $\lambda$ is a   James partitions of length $n$ then the extension space has dimension at most $n-1$, and that this bound is achieved.

\q In Section 8 we calculate the extension sequence for $3$ part partitions $(a,b,c)$ in which $(a,b)$ is James and $(b,c)$ is not. 

\q In Section 9, we use the results given so far to describe various general results including Proposition 0.4 above, that the cohomology group, for a non-James partition $\lambda$  is embedded in the cohomology for a certain $3$ part or $2$ part partition  made up of  consecutive parts of $\lambda$.

\q In  Section 10 we describe all coherent triples for $3$ part partitions $(a,b,c)$ in which $(a,b)$ is split. 

\q In  Section 11 we describe all coherent triples for $3$ part partitions $(a,b,c)$ in which $(a,b)$ is pointed. 

\q In Section 12 we obtain our main result for extension groups. The final result turns out to be much simpler than one might have anticipated, in view of the complexity coming from the analysis of  coherent triples.  

\q In Section 13 we comment on some  questions raised by Hemmer, \cite{Hem}. 

\q In Appendix I we   give a proof of Proposition 0.2, based on \cite{HTT}, Section 10.

\q Finally, in Appendix II, we list together, for the convenience of the reader, the complete set of relations for coherent  multi-sequences that are used throughout.

\bs\bs\bs\bs


\section{Divided powers operators, hyperaglebras and Verma's Conjecture.}

\q In calculating extensions of $B(n)$-modules it is technically much simpler to work with the action of the hyperalgebra,  more precisely with\\
  $(\hy(U(n)),T(n))$-modules, where $\hy(U(n))$ is the hyperalgebra of the group $U(n)$ of lower unitriangular matrices in $G(n)$, than the group $B(n)$ itself. In this section we describe the hyperalgebra of $U(n)$. We shall also need a version of Verma's Conjecture relating rational modules and modules for the hyperalgebra of a semisimple group. In fact the approach we take here leads to a new proof, which we include here, of the usual version of Verma's Conjecture.

\q Let $G$ be a linear  algebraic group over $K$. We write $K[G]$ for the coordinate algebra of $G$. Then $K[G]$ is naturally a Hopf algebra, in particular a coalgebra,  and the full linear dual $K[G]^*$ is thus naturally an associative $K$-algebra.  If $V$ is a rational $G$-module then $V$ is naturally a $K[G]^*$-module. The action is given as follows. If $V$ has basis $\{v_i \vert i\in I\}$ and the elements $f_{ij}\in K[G]$,   are defined by the formulas  $gv_i=\sum_{j\in I} f_{ji}(g)v_j$, for $g\in G$, $i,j\in I$, then, for $\gamma\in K[G]^*$, we have  $\gamma(v_i)=\sum_{j\in J} \gamma(f_{ji})v_j$.

\q Let $M$ be the augmentation ideal of $K[G]$, i.e.,   \\ 
$M=\{f\in K[G]\vert f(1)=0\}$.   Then the  hyperalgebra (or algebra of distributions) $\hy(G)$ is defined by  
$$\hy(G)=\{\gamma\in K[G]^* \vert \gamma(M^r)=0 \hbox{ for some } r\geq 0\}.$$
Then $\hy(G)$ is a  Hopf subalgebra of the Hopf dual of $K[G]$.  For a left $\hy(G)$-module $V$ we have the space of \lq\lq fixed points" 
$$H^0(\hy(G),V)=\{v\in V \vert Xv=\ep(X)v  \hbox{ for all } X\in \hy(G)\}$$
where $\ep:\hy(G)\to K$ is the augmentation map of the Hopf algebra $\hy(G)$. i.e., the map given by $\ep(X)=X(1_{K[G]})$, for $X\in \hy(G)$.

\q Let  $e_{ij}$, $1\leq i,j\leq n$, be the matrix units  of the algebra $M(n)$  of $n\times n$ matrices with entries in $K$ and let $I_n$ denote the identity matrix. For $1\leq r,s\leq n$ with $r\neq s$,  we define $u(r,s)(\xi)=I_n+ \xi e_{rs}\in G(n)$, for $\xi\in K$,  and define $U(r,s)$ to the the one dimensional unipotent subgroup $\{u(r,s)(\xi) \vert \xi\in K\}$ of $G(n)$.

\q The coordinate algebra $K[U(r,s)]$ is the free polynomial algebra in one variable $x(r,s)$, where $x(r,s)(u(r,s)(\xi))=\xi$, for $\xi\in K$. The hyperalgebra $\hy(U(r,s))$ has $K$-basis $X(r,s)_i$, $i\geq 0$, where $X(r,s)_i(x(r,s)^j)=\delta_{ij}$ (the Kronecker delta), for $i,j\geq 0$. We shall call these basis elements  the {\em divided powers operators}.

\q If $V$ is a rational $U(r,s)$-module and $v\in V$ and $v_0,v_1,\ldots \in V$ are such that  
$$u(r,s)(\xi)v=\sum_{i\geq 0} \xi^i v_i$$
for $\xi\in K$,  then $X(r,s)_i v= v_i$.    From the fact that $u(r,s)(\nu)u(r,s)(\xi)=u(r,s)(\nu+\xi)$, for $\nu,\xi\in K$, we deduce that 
$$X(r,s)_i X(r,s)_j v= {i+j \choose  i} X(r,s)_{i+j}v$$
for $v\in V$, $i,j\geq 0$.  Taking $V=K[G(n)]$, the left regular module and  evaluating at $1$, we deduce that 
$$X(r,s)_i X(r,s)_j = {i+j \choose  i} X(r,s)_{i+j}$$
for all $i,j\geq 0$.

\q One easily deduces that a linear map between $U(r,s)$-modules is a module homomorphism if and only if it commutes with the action of all $X(r,s)_i$.

\q We write $U(n)$ for the unipotent radical of $B(n)$, i.e., the group of all lower unitriangular matrices.   

\q Suppose that $1\leq s<r\leq n$. Restriction of functions $K[U(n)]$ to $K[U(r,s)]$ is a surjection of Hopf algebras which induces a monomorphism of Hopf algebras $\hy(U(r,s))\to \hy(U(n))$, by which we identity $\hy(U(r,s))$ with a Hopf subalgebra of $\hy(U(n))$.

\q Furthermore multiplication 
$$U(2,1)\times U(3,1)\times U(3,2)\times \cdots \times U(n,n-1)\to U(n)$$ is an isomorphism of varieties and induces $K$-space isomorphism 
$$\hy(U(2,1))\otimes \hy(U(3,1))\otimes \hy(U(3,2))\otimes \cdots \otimes \hy(U(n,n-1))\to \hy((U(n)).$$

\q Hence the elements 
$$X(2,1)_{i(2,1)}X(3,1)_{i(3,1)}X(3,2)_{i(3,2)}\ldots X(n,n-1)_{i(n,n-1)}\eqno{(*)}$$
with $i(2,1),i(3,1),i(3,2),\ldots, i(n,n-1)\geq 0$, form a $K$-basis of $\hy(U(n))$ (a version of the Poincar\'e-Birkhoff-Witt Theorem).

\begin{lemma} (i) A linear map between $U(n)$-modules is a $U(n)$-module homomorphism if and only if it commutes with the action of all $X(r+1,r)_i$, with $1\leq r<n$, $i> 0$.

(ii)    A linear map between $B(n)$-modules is a $B(n)$-module homomorphism if and only if it is a $T(n)$-module homomorphism and commutes with the action of all $X(r+1,r)_i$, with $1\leq r < n$, $i> 0$.

\end{lemma}

\begin{proof}  
Let $V$ and $W$ be $B(n)$-modules and let $\phi:V\to W$ be a linear map. Certainly if $\phi$ is a $B(n)$-module map then it is a $k[B(n)]^*$-module map and hence commutes with the action of all $X(r+1,r)_i$, $1\leq r<n$, $i> 0$.  Suppose, on the other hand that $\phi$ commutes with the action of all $X(r+1,r)_i$.  Then, for $v\in V$ and $\xi\in K$, we have 
\begin{align*}\phi(u(r+1,r)(\xi)v)&=\phi(\sum_{i\geq 0} \xi^i X(r+1,r)_iv)=\sum_{i\geq 0} \xi^i X(r+1,r)_i \phi(v)\cr
&=u(r+1,r)(\xi) \phi(v). 
\end{align*}
Hence $\phi$ is a module homomorphism for the groups  $U(r+1,r)$, $1\leq r<n$, and hence for the group  they generate, namely $U(n)$. This proves  (i) and now (ii) is clear.

\end{proof}

By a similar, but simpler argument, one  also has the following.

\newpage

\begin{lemma} (i) A subspace of a $U(n)$-module is a $U(n)$-submodule if and only if it is stabilised by all divided power operators 
$X(r+1,r)_i$, for $1\leq r<n$, $i> 0$.

(ii) A $U(n)$-module is trivial if and only if all divided power operators, \\ 
$X(r+1,r)_i$, for $1\leq r<n$, $i>0$, act as zero. 

(iii)  A subspace of a $B(n)$-module is a $B(n)$-submodule if and only if it is a $T(n)$-submodule and stabilised by all divided power operators $X(r+1,r)_i$, for $1\leq r<n$, $i> 0$.

\end{lemma}

\q We shall consider the category of $(\hy(U(n)),T(n))$-modules. By a \\
 $(\hy(U(n)),T(n))$-module we mean a vector space $V$, say,  which is a $\hy(U(n))$-module and $T(n)$-module in such a way that $X(r,s)_i V^\alpha  \subseteq  V^{\alpha+i(\ep_r-\ep_s)}$,  for all $\alpha \in X(n)$, $1\leq r,s\leq n$ with $r>s$ and $i>0$.  If $V$ and $W$ are $(\hy(U(n)),T(n))$-modules then a map $\phi:V\to W$ is a morphism of $\hy(U(n))$-modules and $T(n)$-modules.  If $V$ is a $B(n)$-module then we write $V^\hy$, for $V$ regarded as a $T(n)$-module and $\hy(U(n))$-module as above. Thus we have an exact functor $V\mapsto V^\hy$, from finite dimensional  $B(n)$-modules to finite dimensional  $(\hy(U(n)),T(n))$-modules.  We shall need that this is an equivalence of categories.

 \q In fact the  argument we shall give is valid for an arbitrary reductive group over $K$ and we give  it in that context.  So now let $G$ be a reductive group over $K$. Let $T$ be a maximal torus and let $W=N_G(T)/T$ be the Weyl group (where $N_G(T)$ is the normaliser of $T$ in $G$). Let $X(T)$ be the character group of $T$. Then $W$ acts naturally on $X(T)$. We pick a positive definite symmetric $W$-invariant bilinear form on $\real\otimes_\zed X(T)$.  Now $T$ acts naturally on the Lie algebra $\Lie(G)$ of $G$ and the set of 
 $\Phi$ of non-zero weights is a root system in the $\real$-span, $E$ say, of $\Phi$.  We choose a set  of positive roots $\Phi^+=\{\alpha_1,\ldots,\alpha_N\}$ (where $N=|\Phi^+|$). We have the usual partial order on $X(T)$. Thus for $\lambda,\mu\in X(T)$ we write  $\lambda\geq \mu$ if $\lambda-\mu$ is a sum of positive roots.

\begin{remark}  \rm We shall use the following elementary fact.  If  $\mu$ is a weight of a $B$-module $V$ then there exists a one dimensional submodule isomorphic to $K_\lambda$
for some $\lambda\in X(T)$ with $\lambda\leq \mu$. (This follows for example from \cite{D4}, (1.4.4), or \cite{RAG}, II, 4.9 Lemma.)
\end{remark}

\q  The height of an element $\mu\in X(T)$ with $\mu\geq 0$  is defined by the condition that $\mu$ is a sum of $\height(\mu)$ simple roots. For $\mu\geq 0$ we define $P(\mu)$ to be the number of ways of writing $\mu$ as a sum of positive roots, i.e., the number of tuples $(m_\alpha)_{\alpha\in \Phi^+}$ such that $\mu=\sum_{\alpha\in \Phi^+} m_\alpha \alpha$.
 
 \q For $\alpha\in \Phi$ let $U_\alpha$ be the corresponding root subgroup.  Let $G_\add$ be the additive group of $K$, regarded as an algebraic group, in the natural way.  For $\alpha\in \Phi$, we choose an isomorphism of algebraic groups $\phi_\alpha:G_\add\to U_\alpha$ such that $t\phi_\alpha(\xi)t^{-1}=\phi_\alpha(\alpha(t) \xi)$, for $\xi\in K$. 
 
 \q There are uniquely determined elements $X_{\alpha,i}$, $i\geq 0$,  of the hyperalgebra of $U_\alpha$ such that if $V$ is a rational $U_\alpha$-module and $v\in V$ then $\phi_\alpha(\xi)v=\sum_{i\geq 0} \xi^i X_{\alpha,i} v$. We denote by $B$  the negative Borel subgroup of $G$ and denote by $U$ be the unipotent radical of $B$.   Identifying $\hy(U_\alpha)$ with a subalgebra of $\hy(U)$, for $\alpha\in \Phi^-$, we have, by the argument of (*) above,  that $\hy(U)$ has a $K$-basis consisting of the elements
 $$X_{-\alpha_1,i_1}X_{-\alpha_2,i_2}\ldots X_{-\alpha_N,i_N}$$
 with $i_1,\ldots,i_N$ running over the non-negative integers.
 
 \q For $\lambda\in X(T)$ we denote by $K_\lambda$ the one dimensional $B$-module on which $T$ acts with weight $\lambda$. We shall use the fact that if $\Ext^1_B(K_\lambda,K_\mu)\neq 0$, for $\lambda,\mu\in X(T)$, then $\lambda>\mu$ (see e.g. \cite{RAG}).

  \q We shall need a version of Lemma 1.2 in the more general situation. We leave the details to the reader.

 \begin{lemma} (i) A subspace of a $U$-module is a $U$-submodule if and only if it is stabilised by all divided power operators 
$X_{-\alpha,i}$, for $\alpha\in \Phi^+$ , $i>0$.

(ii) A $U$-module is trivial if and only if all divided power operators, \\ 
$X_{-\alpha,i}$, for $\alpha\in \Phi^+$, $i>0$, act as zero. 

(iii)  A subspace of a $B$-module is a $B$-submodule if and only if it is a $T$-submodule and stabilised by all divided power operators $X_{-\alpha,i}$, for $\alpha\in \Phi^+$, $i>0$.

\end{lemma}

 \q By a $(\hy(U),T)$-module we mean a vector space $V$, say, which is a $T$-module and $\hy(U)$-module in such a way that $X_{-\alpha,i}V^\lambda\subseteq V^{\lambda-i\alpha}$, for all $\lambda\in X(T)$, $\alpha\in \Phi^+$, $i\geq 0$.  If $V$ is a $B$-module we write $V^\hy$ for the corresponding $(\hy(U),T)$-module.  We write $\mod(B)$ for the category of finite dimensional rational $B$-module and $\mod(\hy(U),T)$ for the category of finite dimensional $(\hy(U),T)$-modules. We have, as above, an  exact functor 
  $F_B:\mod(B)\to \mod(\hy(U),T)$, sending $V\in \mod(B)$ to $V^\hy\in \mod(\hy(U),T)$.

  \begin{proposition} The functor $F_B:\mod(B)\to \mod(\hy(U),T)$ is an equivalence of categories. 
\end{proposition}

\begin{proof}  It is enough to prove that for any $V\in \mod(\hy(U),T)$ there exists $Z\in \mod(B)$ such that $Z=V^\hy$.

 \q We first produce a $B$-module that will be of great use to us.  Multiplication  $T\times U\to B$ is an isomorphism of varieties. We put 
 $$A=\{f\in K[B] \vert f(tu)=f(u)   \hbox { for all } t\in T, u\in U\}.$$
  Then $A$ is a left $B$-module summand of $K[B]$ and has trivial socle $K$. Thus $A$ is the injective hull of $K$, as a $B$-module.   The map  $\pi:A\to K[U]$, given by restriction of functions is an isomorphism of $K$-algebra.  The group $T$ acts on $U$ by conjugation and induces on $K[U]$ the structure of a left $T$-module. For $f\in A$, $t\in T$, $u\in U$, we have 
  $$\pi (tf)(u)=f(ut)=f(tt^{-1} u t)=(tf)(t^{-1} u t)=f(t^{-1} u t)=(t\pi(f))(u).$$
   Hence $\pi$ is an isomorphism of $T$-modules. Let $M=\{f\in K[U] \vert f(1)=0\}$. For $r\geq 0$, the  $r$th power  $M^r$ of the ideal $M$  is a $T$-submodule of $K[U]$.   By complete reducibility we have  that $K[U]$ and  $\bigoplus_{r\geq 0} M^r/M^{r+1}$ are isomorphic as $T$-modules.    Now  $K[U]$ is a polynomial algebra so that the natural map from the $r$th symmetric power    $S^r(M/M^2)$ to $M^r/M^{r+1}$ is an isomorphism, for $r\geq 0$.   Furthermore $M/M^2$ is $T$-module isomorphic to $\Lie(U)^*$, the dual of the Lie algebra of $U$. The weights of $\Lie(U)$ are $-\alpha_1,\ldots,-\alpha_N$, and hence the weights of $\Lie(U)^*$ are $\alpha_1,\ldots,\alpha_N$, each occurring with multiplicity one.  It follows that for $\mu\in X(T)$, we have  $\dim K[U]^\mu$ is $P(\mu)$, if $\mu\geq 0$ and is $0$ otherwise.
  
  \q We regard $K[U]$ as a $B$-module via the isomorphism $\pi:A\to K[U]$.   We fix $r\geq 0$.   Let  $Q_r$ be the largest $B$-submodule of $K[U]$ with weights of height at most $r$. We claim that  $\dim Q_r^\mu= P(\mu)$, for $\mu\in X(T)$ with $\height(\mu)\leq r$.   Suppose that $\nu\in X(T)$ and $K_\nu$ appears in the socle of $K[U]/Q_r$.  Then there is a submodule $H$ of $K[U]$ containing $Q_r$ such that $H/Q_r$ is isomorphic to $K_\nu$. Since $H$ is not contained in $Q_r$ we have $\height(\nu)>r$. Now if $\tau$ is any weight of $K[U]/Q_r$ then we must have $\tau\geq \nu$ for some $\nu\in X(T)$ such that $K_\nu$ appears in the socle of $K[U]/Q_r$, by Remark 1.3.  Hence we have $\height(\tau)\geq \height(\nu)>r$. Hence we have $\dim Q_r^\tau=\dim K[U]^\tau= P(\tau)$, for all $\tau\geq 0$ with $\height(\tau)\leq r$.

 \q We now show that any finite dimensional $(\hy(B),T)$-module $V$ embeds in a module of the form $Z^\hy$, for some $B$-module $Z$.  Suppose $V_1,V_2$ are non-zero $(\hy(B),T)$-submodules such that $V_1\bigcap V_2=0$. Thus $V$ embeds in $V/V_1\oplus V/V_2$ so that if 
 $V/V_1$ embeds in $Z_1^\hy$ and $V/V_2$ embeds in $Z_2^\hy$, for $B$-modules $Z_1,Z_2$ then $V$ embeds in $Z_1^\hy\oplus Z_2^\hy=(Z_1\oplus Z_2)^\hy$.  Hence we may suppose that $V$ has simple socle $L$, say.  Then $L$ is isomorphic to $K_\lambda^\hy$, for some $\lambda\in X(T)$. If $V\otimes L^*$ embeds in $Z^\hy$, for some $B$-module $Z$ then, since $V$ is isomorphic to $(V\otimes L^*)\otimes L$, it embeds in $Z^\hy\otimes K_\lambda^\hy=(Z\otimes  K_\lambda)^\hy$.  Hence, replacing $V$ by $V\otimes L^*$, we may assume that socle of $V$ is a one dimensional trivial module. Thus all weights of $V$ are $\geq 0$ and we choose $r$ such that $\height(\mu)\leq r$, for every  weight $\mu$ of $V$. 
 
\q The dual  module $H=V^*$ has simple head the one dimensional simple module. In particular $H$ is generated by a weight vector $h_0$, say, of weight $0$.  Hence $H$ is spanned by the elements $X_{-\alpha_1,i_1}\ldots X_{-\alpha_N,i_N}h_0$ such that $\height(i_1\alpha_1+\cdots+i_N\alpha_N)\leq r$.  We write $I_r$ for the $K$-span of the elements  $X_{-\alpha_1,i_1}\ldots X_{-\alpha_N}$ such that $\height(i_1\alpha_1+\cdots+i_N\alpha_N)> r$. We have a $\hy(U)$-module surjection $\phi:\hy(U)\to H$, $\phi(X)=Xh_0$, for $X\in \hy(U)$. Let $J_H$ be the kernel of $\phi$, i.e., the annihilator of $h_0$. Then we have $I_r\subseteq J_H$ and $\hy(U)/J_H \cong H$, in particular we have 
$$\dim H\leq \dim \hy(U)/ I_r=\sum_\mu P(\mu)$$
where the sum is over all $\mu\in X(T)$ with $\mu\geq 0$ and $\height(\mu)\leq r$.

\q This applies in particular to $Q_r^*$ and since $\dim Q_r^*= \dim \hy(U)/I_r$ we have that $I_r$ is the annihilator of an element of weight zero in $Q_r^*$, in particular $I_r$ is a left ideal and, regarding $\hy(U)/I_r$ as a $(\hy(U),T)$-module, in such a way that $X_{-\alpha_1,i_1}\ldots X_{-\alpha_N,i_N}+ I_r$ has weight $-(i_1\alpha_1+\cdots + i_N\alpha_N)$, we have $Q_r^*\cong \hy(U)/I_r$. Hence we have that $H=\hy(U)/J$ is an epimorphic image of $Q_r^*$ and so $V$ embeds in $Q_r^\hy$. 

\q Thus we have that a $(\hy(U),T)$-module $V$ may be identified with a submodule of a module of the form $Z^\hy$, for some $B$-module $Z$. But then, for $\alpha\in \Phi^+$, we have $\phi_{-\alpha}(\xi)v=\sum_{i\geq 0} \xi^i X_{-\alpha,i} v$, for $\xi\in K$, $v\in V$. Hence $V$ is a $B$-submodule of $Z$ and $V=V_0^\hy$, where $V_0$ is the space $V$, regarded as a $B$-submodule of $Z$.

\end{proof}

\section*{Verma's Conjecture}

\q We note that in fact the above Proposition gives a new proof of the so-called Verma's conjecture. For other proofs see \cite{sullivansimply} and \cite{CPS}.  We take $G$ to be semisimple and simply connected.  Then a $G$-module $V$ is naturally a $\hy(G)$-module giving an exact functor $F_G:\mod(G)\to \mod(\hy(G))$, where $\mod(\hy(G))$ denotes the category of finite dimensional $\hy(G)$-modules.  Less formally, for a $G$-module $V$, we will say also that $V$ is regarded as a $\hy(G)$-module. The  action of $X_{\alpha,i}$, for $\alpha\in \Phi$, $i\geq 0$, on $v\in V$, is given by the formula $x_{\alpha}(\xi)v=\sum_{i\geq 0} \xi^i x_{\alpha,i}v$. Similarly a rational $T$-module may be regarded as a $\hy(T)$-module.  The set $\Hom_{K-\alg}(\hy(T),K)$,  of 
$K$-algebra homomoprhisms from $\hy(T)$ to $K$, is naturally a group (since $\hy(T)$ is a Hopf algebra). 
For  $\lambda\in X(T)$, one has a $k$-algebra homomorphism $\lambda':\hy(T)\to K$, given by $hv=\lambda(h)v$, for $v\in T_\lambda$. Then map $X(T)\to \Hom_{K-\alg}(\hy(T),K)$, sending $\lambda$ to $\lambda'$ is an injective homomorphism.  We say that a $K$-algebra homomoprhism from  $\hy(T)$ to $K$ is integral if it has the form $\lambda'$ for some $\lambda\in X(T)$. We identify $\lambda\in X(T)$ with $\lambda\in \Hom_{K-\alg}(\hy(T),K)$ and in this way identify $X(T)$ with the group of integral $K$-algebra homomorphisms from $\hy(T)$ to $K$.  For a finite dimensional $\hy(G)$-module $V$ we have $V=\oplus_{\lambda\in X(T)} V^\lambda$, where $V^\lambda=\{v\in V\vert hv=\lambda(h) \hbox{ for all } h\in H\}$. Furthermore if $V$ is simple then, it follows from the classification of finite dimensional simple modules, that  $V$ is isomorphic to $F_G(M)$, for some finite dimensional $G$-module $M$.  For further details see e.g.,  \cite{Haboush}.

\q We shall say that a finite dimensional $\hy(G)$-module $V$ is rational if  is isomorphic to $F_G(M)$, for some rational $G$-module $M$. This is equivalent to the statement that the vector space $V$ may be given the structure of a rational $G$-module in such a way that $V^\lambda$ is the $\lambda$ weight space for $V$, regarded as a $T$-module and $x_\alpha(\xi)v=\sum_{i\geq 0} \xi^i X_{\alpha,i} v$, for $\alpha\in \Phi$, $\xi\in K$, $v\in V$.

\q We shall need the following.

\begin{lemma} (i) For a finite dimensional  $\hy(G)$-module $V$ we have  \\
$H^0(\hy(G),V)=H^0(\hy(U),V)^T$.

(ii) For finite dimensional  $\hy(G)$-modules $V_1,V_2$  we have $\Hom_{\hy(G)}(V_1,V_2)=\Hom_{\hy(U)}(V_1,V_2)^T$.

\end{lemma}

\begin{proof}
 (i)  Certainly $H^0(\hy(G),V)\leq H^0(\hy(U),V)^T$.  Let $U^+$ be the unipotent radical of the positive Borel subgroup $G$. Then we have that the multiplication map $\hy(U^+)\otimes \hy(T)\otimes  \hy(U)\to \hy(G)$ is a vector space isomorphism. Let  $v_0\in H^0(\hy(U),V)^T$.  Then we have that $\hy(G)v_0=\hy(U^+)\hy(T)v_0$ and $\hy(T)$ acts by scalar multiplication on $Kv_0$ so that $\hy(G)v_0=\hy(U^+)v_0$. Moreover $\hy(U^+)$ is a spanned by elements of the form $X_{\alpha_1,i_1}\cdots X_{\alpha_N,i_N}$ so that the module 
$\hy(G)v_0$ has only weights $\geq 0$. Hence the only composition factor of $\hy(G)v_0$ is the trivial module and so $\hy(G)v_0$ is trivial. Hence $v_0\in H^0(\hy(G),V)$ and $H^0(\hy(G),V)=H^0(\hy(U),V)^T$. 

(ii) This follows from (i) by identifying $V_1^*\otimes V_2$ with $\Hom_K(V_1,V_2)$ in the usual way.

\end{proof}

\begin{theorem} (Verma's Conjecture) The functor \\
$F_G:\mod(G)\to \mod(\hy(G))$ is an equivalence of categories.

\end{theorem}

\begin{proof} It is enough to prove that a finite dimensional $\hy(G)$-module is rational. Let $V$ be a finite dimensional $\hy(G)$-module. We regard $V$ as a $(\hy(U),T)$-module. By Proposition 1.5 the space $V$, regarded as a $(\hy(U),T)$-module has the structure of a $B$-module, i.e., the is a $B$-module action  on $V$ such that $x_\alpha(\xi)v=\sum_{i\geq 0} \xi^i X_{\alpha,i} v$, for $\alpha\in \Phi^+$, $\xi\in K$, $v\in V$.  We write $\barV$ for the $K$-space $V$ regarded as a $B$-module via this action.  

\q  The induction functor $\ind_B^G:\mod(B)\to \mod(G)$ is left exact and we have the right derived functors $R^i\ind_B^G:\mod(B)\to \mod(G)$, for $i\geq 0$. If $M$ is a $G$-module then we have that $R^1\ind_B^G M=0$  and, by \cite{tata}, Exercise 2.1.4 for example,  the evaluation map $\ev_M:\ind_B^G M\to M$ is a $B$-module isomorphism.

\q We claim that the evaluation map $\ev_{\barV}:\ind_B^G \barV\to \barV$ is an isomorphism.   If $V$ is irreducible then it has the form $F_G(M)$ for some simple $G$-module $M$ and then $\barV$ is isomorphic to the restriction of $M$ to $B$, and hence $\ev_{\barV}:\ind_B^G \barV\to \barV$ is an isomorphism. Now assume that $V$ is not irreducible, let $L$ be an irreducible $\hy(G)$-submodule and let $Q=V/L$. Then we have a commutative diagram

$$\begin{matrix}
0&\rightarrow&\ind_B^G  \barL  &\rightarrow &\ind_B^G \barV  &\rightarrow  &\ind_B^G \barQ &\rightarrow 0 \cr
&&\downarrow&&\downarrow&&\downarrow \cr
0&\rightarrow &\barL &\rightarrow &\barV &\rightarrow  &\barQ &\rightarrow 0 \cr
\end{matrix}
$$
of $B$-modules, with rows exact (since $R^1\ind_B^G \barL=0$), where the vertical maps are the evaluation maps. We may assume, by induction on dimension, that the first and third vertical maps are isomorphisms and hence, so too is the middle map $\ev_\barV:\ind_B^G \barV\to \barV$.

\q This proves the claim and therefore that the action of $B$ on $\barV$ extends to a $G$-module action. Let $Z$ be the vector space $V$ regarded as a $G$-module via this action. Then we have $Z^\lambda=V^\lambda$, for $\lambda\in X(T)$,  and $x_{-\alpha}(\xi)z=\sum_{i\geq 0} \xi^i X_{-\alpha,i}z$, for $\alpha\in \Phi^+$, $\xi\in K$, $z\in Z$. Hence, as a $(\hy(U),T)$-module, $Z$ is precisely $V$. The identity map $Z\to V$ is an isomorphism of $(\hy(U),T)$-modules and hence, by Lemma 1.6(ii), an isomorphism of $\hy(G)$-modules. Thus we have $V=F_G(Z)$ and we are done.

\end{proof}

\bs\bs\bs\bs


\section{Generators and Relations}

\bs To make further progress we give generators and relations for the hyperalgebra  $\hy(U(n))$ of $U(n)$.  First fix $1\leq r,s\leq n$, with $r\neq s$. Then, for $\eta,\zeta\in K$, we have $u(r,s)(\eta)u(r,s)(\zeta)=u(r,s)(\eta+\zeta)$. Hence, for a rational $G$-module $V$ and $v\in V$, we have 
$$u(r,s)(\eta)u(r,s)(\zeta)v=u(r,s)(\eta+\zeta)v$$ and hence
$$\sum_{i\geq 0,j\geq 0}\eta^i\zeta^j X(r,s)_i X(r,s)_jv=\sum_{k\geq 0} (\eta+\zeta)^k X(r,s)_kv.$$
Thus,  for $i,j\geq 1$,  we have 
$$X(r,s)_i X(r,s)_jv={i+j\choose i} X(r,s)_{i+j}v.$$
Applying this to the regular  module $K[G]$ and evaluating at $1$, we obtain
$$X(r,s)_i X(r,s)_j={i+j\choose i} X(r,s)_{i+j}.$$

\q   Since we are interested in the hyperalgebra of $U(n)$  we are interested in the operators $X(r,s)_i$ with $n\geq r>s\geq 1$, and $i>0$. For ease of calculation we will find it useful to set $Y(r,s)_i=X(s,r)_i$, for $1\leq r<s\leq n$, $i>0$. We thus have 

$$Y(r,s)_i Y(r,s)_j={i+j\choose i} Y(r,s)_{i+j}\eqno{(1)}$$

\q We now  fix $1\leq r<s< t\leq n$.  Then it is easy to check that we have
$$u(t,s)(\eta) u(s,r)(\zeta)=u(s,r)(\zeta)u(t,r)(\eta\zeta)u(t,s)(\eta)$$
for $\eta,\zeta\in K$. Hence, for any $G$-module $V$ and any $v\in V$ we have
$$u(t,s)(\eta) u(s,r)(\zeta)v=u(s,r)(\zeta)u(t,r)(\eta\zeta)u(t,s)(\eta)v$$
and therefore
$$\sum_{j,i\geq 0} \eta^j \zeta^i X(t,s)_j X(s,r)_iv=\sum_{h,i,j\geq 0}\zeta^{i+h}   \eta^{j+h} X(s,r)_i X(t,r)_h X(t,s)_jv.$$
Hence for all  $i,j\geq 0$ we have
$$X(t,s)_j X(s,r)_iv=\sum_{h=0}^{{\rm min}\{i,j\}}X(s,r)_{i-h}X(t,r)_h X(t,s)_{j-h}v$$
for all $v\in V$. Applying this to the regular module $K[G]$ and evaluating at the identity we have

$$X(t,s)_j X(s,r)_i=\sum_{h=0}^{{\rm min}\{i,j\}}X(s,r)_{i-h}X(t,r)_h X(t,s)_{j-h}.$$
Hence we have
$$Y(s,t)_j Y(r,s)_i=\sum_{h=0}^{{\rm min}\{i,j\}}Y(r,s)_{i-h}Y(r,t)_h Y(s,t)_{j-h} \eqno{(2)}$$

for $1\leq r<s< t\leq n$ and $i,j\geq 1$.

\q Moreover, we have $u(t,r)(\zeta)u(s,r)(\eta)=u(s,r)(\eta)u(t,r)(\zeta)$ and \\
$u(t,r)(\zeta)u(t,s)(\eta)=u(t,s)(\eta)u(t,r)(\zeta)$ from which we deduce that

$$Y(r,s)_iY(r,t)_j=Y(r,t)_jY(r,s)_i \eqno{(3)}$$
and
$$Y(s,t)_iY(r,t)_j=Y(r,t)_jY(s,t)_i \eqno{(4)}$$
for all $1\leq r<s<t\leq n$, $i,j>0$.

\q Now suppose that $1\leq q,r,s,t\leq n$ are distinct. Then we have \\
$u(q,r)(\eta)u(s,t)(\zeta)=u(s,t)(\zeta)u(q,r)(\eta)$ and by the argument above we obtain
$$Y(q,r)_j Y(s,t)_i=Y(s,t)_iY(q,r)_j \eqno{(5)}$$
for all $1\leq q,r,s,t\leq n$, with $q,r,s,t$  distinct and  $q<r$, $s<t$  and $i, j\geq 1$.

\q Now it is easy to see that relations (1)-(5)   may be used to express any monomial  in $Y(r,s)_i$ (with $1\leq r,s\leq n$, $r\neq s$, $i\geq 1$) as a linear combination of monomials of the form

$$Y(1,2)_{i(1,2)}Y(1,3)_{i(1,3)}Y(2,3)_{i(2,3)}\ldots Y(n-1,n)_{i(n-1,n)}$$
with $i(1,2),\ldots,i(n-1,n)$ non-negative integers. 

\q Hence we have:

\begin{proposition} For the  generators $Y(r,s)_i$, $1\leq r<s\leq n$, $i\geq 1$  the  relations  (1)-(5) above form a set of defining relations for $\hy(U(n))$.

\end{proposition}

\bs\bs\bs\bs


\section{James's Theorem on invariants of Specht modules.}

\q We are interested in the degree $1$ cohomology of  Specht modules.  However, we should mention that the degree $0$ situation, namely the invariants of the Specht modules,  was solved by James, see  \cite{James}, 24.4 Theorem.  As a  warm-up exercise we deduce this from a result of $B(n)$-module homomorphisms using the divided powers operators.

\q It is convenient to describe some notation and recall Lucas's formula at this point. In this section, for $a,b\geq 0$, the binomial coefficient ${a\choose b}$ will be understood to lie in the field $K$ (of positive characteristic $p$), i.e., we write simply ${a\choose b}$ for ${a\choose b}1_K$.

\bs

\bf Notation: \rm We write a non-negative  integer  $a$ in base $p$ as $a=\sum_{i\geq 0} p^i a_i$, with $0\leq a_i<p$.  If $a$ is positive its $p$ length $\len_p(a)$ is the largest non-negative integer $l$ such that $a_l\neq 0$ and its $p$-adic valuation $\val_p(a)$ is the maximum positive integer $v$ such that $p^v$ divides $a$.

\bs

\q  We recall that for a positive integers  $a$ and $b$ written in base $p$ as $a=\sum_{i\geq 0} a_i p^i$ and $b=\sum_{i\geq 0} b_ip^i$ we have 
$${a\choose b}={a_0\choose b_0}{a_1\choose b_1}\cdots {a_m\choose b_m} \eqno{(\hbox{Lucas's Formula)}}$$
where $m$ is the maximum of $\len_p(a)$ and $ \len_p(b)$.

\bs

\q For a non-negative integer $a$ we define $\bara$ by the formula $a=a_0+p\bara$.   By Lucas's formula, for non-negative integers $a=a_0+p\bara$, $b=b_0+p\barb$ we have ${a\choose b}={a_0\choose b_0}{\bara\choose \barb}$.

\begin{remark} \rm For non-negative integers $a,b$, written in base $p$  as above,  we have ${a+b\choose b}\neq 0$ if and only if $a_i+b_i\leq p-1$ for all $i\geq 0$. To see this we write $a=a_0+p\bara$, $b=b_0+p\barb$, as above.  If $a_0+b_0\geq p$ then $(a+b)_0<b_0$ and ${a+b\choose b}=0$ by Lucas's Formula. If $a_0+b_0\leq p-1$ then ${(a+b)_0\choose b_0}\neq 0$ and ${a\choose b}={(a+b)_0\choose b_0}{\bara+\barb\choose \barb}$ so  the result follows by induction.
\end{remark}

\q We first describe a condition that will be of crucial importance to us for the rest of the paper.

\begin{definition} A two part partition $(a,b)$, with $b>0$, will be called a {\em James partition}  if ${a+i\choose i}= 0$ for all $1\leq i \leq b$.  A partition $\lambda=(\lambda_1,\ldots,\lambda_n)$ of length $n$ will be called a James partition if $(\lambda_r,\lambda_{r+1})$ is James, for all $1\leq r<n$. We adopt the convention that a partition of length one is James.

\end{definition}

\q By a two part partition we mean a partition with precisely two non-zero parts.

\begin{remark} \rm A two part partition  $(a,b)$  is a James partition if and only if $b<p^{\val_p(a+1)}$.   The James condition is that for all  $1\leq m\leq b$ we have ${a+m\choose m}=0$, i.e., writing $m=\sum_{i\geq 0} m_ip^i$, we have $a_i+m_i>p-1$ for some $i$. This can only happen if $a_0=\cdots=a_{\len_p(b)}=p-1$, i.e., $\val_p(a+1)\geq \len_p(b)$, i.e., $b<p^{\val_p(a+1)}$.
\end{remark}

\q We establish some notation involving the symmetric powers of the natural module that will be in force for rest of the paper.   The module $E$ has the basis $e_1,\ldots,e_n$, where $e_i$ is the column vector of length $n$ with $1$ in the $i$th position and zeroes elsewhere.   For $\alpha=(\alpha_1,\ldots,\alpha_n)\in \Lambda(n,d)$ we set $e^\alpha=e_1^{\alpha_1} e_2^{\alpha_2}\ldots e_n^{\alpha_n}\in S^d E$.  

\q Let $1\leq s,t\leq n$ with $s\neq t$.  We calculate  the effect of $X(s,t)_i$ on $S^d E$. We have 
$$u(s,t)(\xi)e_i=\begin{cases} e_i +   \xi e_s,   &  \hbox{\  if\ }  i=t; \cr
e_i, & \hbox{\  if\ }  i\neq t.
\end{cases}.$$

Since $U(n)$ acts on $S(E)$ by algebra automorphisms we have 
$$u(s,t)(\xi)e^\alpha=e_1^{\alpha_1}\ldots e_{t-1}^{\alpha_{t-1}}(e_t+\xi e_s)^{\alpha_t}e_{t+1}^{\alpha_{t+1}}\ldots e_n^{\alpha_n}$$ 

which is 
$$\sum_{i=0}^{\alpha_t} {\alpha_t\choose i} \xi^i  e^{\alpha+i(\ep_s-\ep_t)}.
$$

Hence we get
$$X(s,t)_i e^\alpha=\begin{cases}  {\alpha_t\choose i}  e^{\alpha+i(\ep_s-\ep_t)}, 
 &\hbox{if \,}  i\leq \alpha_t; \cr
0, &\hbox{otherwise}.
\end{cases}
$$
in particular, for $1\leq r<n$, $i>0$, we have
$$Y(r,r+1)_i e^\alpha=\begin{cases}  {\alpha_r\choose i}  e^{\alpha-i(\ep_r-\ep_{r+1})}, 
 &\hbox{if \,}  i\leq \alpha_r; \cr
0, &\hbox{otherwise.}
\end{cases}
$$

\q Let $\lambda$ be a partition of degree $d$ with at most $n$ parts. We consider the space $\Hom_{B(n)}(S^d E,K_\lambda)$.   We have the  largest trivial $U(n)$-module quotient $H_0(U(n),S^d E)$.  Thus $H_0(U(n),S^d E)=S^d E/N$, where $N$ is the smallest $U(n)$ submodule of $S^d E$ with trivial quotient.  Thus, by Lemma 1.2 (iii), we have 
$$H_0(U(n),S^d E)= S^d E/N$$ 
where

$$N=\sum_{\substack{1\leq r<n, i >0, \\ \alpha\in \Lambda(n,d)}}  K Y(r,r+1)_i e^\alpha.$$

\q The weight spaces in $S^d E$, and hence in $N$ and $S^d E/N$ are at most one dimensional.  Thus $N$ is the $K$-span of the $e^\beta$ with $(\beta_1,\ldots,\beta_n)\in \Lambda(n,d)$ such that $e^\beta$ is a scalar multiple of $X(r,r+1)_i e^\alpha$, for some $\alpha=(\alpha_1,\ldots,\alpha_n)\in \Lambda(n,d)$ and $1\leq r<n$, $i>0$. By considering weights we see that we must have  $\alpha=\beta+i(\ep_r-\ep_{r+1})$ and $Y(r,r+1)_i e^\alpha={\beta_r+i\choose i}e^\beta$.  Hence $e^\beta\in N$ if and only if we have there exists some $1\leq r< n$ and $0<i \leq \beta_{r+1}$ such that ${\beta_r+i\choose i}\neq 0$,

\q Hence $H_0(U(n),S^d E)$ has $K$-basis $e^\mu+N$, where $\mu$ runs though the set of James partitions of $n$ of degree $r$.  Now we have  
$\Hom_B(S^d E,K_\lambda)=\Hom_{T(n)}(H_0(U(n),S^dE),K_\lambda)$, and so by Proposition 0.1  we obtain the following result.

\begin{proposition} For a partition $\lambda$ of degree $d$ we have
$$\Hom_B(S^d E,K_\lambda)=\begin{cases} K, & \hbox{\, if \,} \lambda \hbox{ if a James partition};\cr
0, & \hbox{otherwise}
\end{cases}
$$
or
$$\Hom_{G(n)}(S^d E,\nabla(\lambda))=\begin{cases} K, & \hbox{\, if \,} \lambda \hbox{ if a James partition};\cr
0, & \hbox{otherwise.}
\end{cases}
$$
\end{proposition}
By  Proposition 0.2 , we have James's result,  \cite{James}, 24.4 Theorem: 
\begin{proposition} For a partition $\lambda$ of degree $d$ we have
$$H^0(\Sigma_d,\Sp(\lambda))=\begin{cases} K, & \hbox{\, if \,} \lambda \hbox{  is a James partition};\cr
0, & \hbox{otherwise.}
\end{cases}
$$
\end{proposition}

\bs\bs\bs\bs


\section{Extension sequences and module extensions}

\q We now  wish to examine the possible $B(n)$-module extensions of $S^d E$ by a one dimensional module.   

\q We fix a partition $\lambda$  with degree $d$ and having at most $n$ parts.    Let $V$ be an extension of $S^d E$ by $K_\lambda$. Thus $V$ contains a one dimensional $B(n)$-submodule $V_\infty$ isomorphic to $K_\lambda$ and $V/V_\infty$ is isomorphic to $S^d E$. We fix $0\neq v_\infty\in V_\infty$ and an isomorphism $\theta: V/V_\infty\to S^d E$.  Note that $\theta$ is determined up to multiplication by a non-zero scalar since 
$$\End_{B(n)}(S^d E,S^d E)=\End_{G(n)}(S^d E, S^d E)=K.$$

\q For $\alpha\in \Lambda(n,r)$, $\alpha\neq \lambda$ there is a unique element  $v_\alpha$ of weight $\alpha$ such that $\theta(v_\alpha+V_\infty)=e^\alpha$. We also choose an element $v_\lambda\in V$  of weight $\lambda$ such that $\theta(v_\lambda+V_\infty)=e^\lambda$. 

\q Let $1\leq r<s\leq  n$. Suppose  $\alpha\in \Lambda(n,r)$, $0<i\leq \alpha_r$ Then we have 
$$Y(r,s)_i e^\alpha={\alpha_r\choose i}e^{\alpha-i(\ep_r-\ep_s)}$$
by Section 2 and hence,  applying $\theta$ to $Y(r,s)_iv_\lambda+V_\infty$ we see that 
 there exist scalars $y(r,s)_i\in K$ such that

$$Y(r,s)_i v_\alpha=\begin{cases} {\alpha_r\choose i}v_{\alpha-i(\ep_r-\ep_s)}+y(r,s)_i v_\infty, & \hbox{if \q }   \alpha=\lambda+i(\ep_r-\ep_s);\cr
{\alpha_r\choose i}v_{\alpha-i(\ep_r-\ep_s)},&  \hbox{if \q }   \alpha\neq\lambda+i(\ep_r-\ep_s).
\end{cases}$$

\q The scalars $y(r,s)_i$ are of paramount importance for our analysis of extensions of symmetric powers. 

\bs

\begin{definition} We call the sequence $(y(r,s)_i)_{1\leq i\leq \lambda_s}$ the $(r,s)$-extension sequence determined by the  $V,v_\infty$ and $\theta$.  We call the collection of sequences $(y(r,s)_i)_ {1\leq r<s \leq n, 1\leq i\leq  \lambda_s}$, the extension multi-sequence determined by $V,v_\infty$ and $\theta$.
\end{definition}

\q Of course we can take for $V$ the split extension $V=S^dE\oplus K_\lambda$ and $V_\infty=K_\lambda=Kv_\infty$.  We take $\theta:V/V_\infty\to S^d E$ to be the natural map (induced by the projection from  $V=S^d E\oplus V_\infty$ onto the first summand).  We take $v_\lambda=e^\lambda -v_\infty$.  Then, for $0<i\leq \lambda_s$, we have

\begin{align*}Y(r,s)_i v_{\lambda+i(\ep_r-\ep_s)}&=Y(r,s)_ie^{\lambda+i(\ep_r-\ep_s)}\cr
&={\lambda_r+i\choose i}e^\lambda={\lambda_r+i\choose i}v_\lambda+{\lambda_r+i\choose i}v_\infty
\end{align*}

\begin{definition} The extension multi-sequence produced in this way will be called the standard multi-sequence for $\lambda$ and denoted $(y(r,s)_i^\st)$. Specifically we have  $y(r,s)_i={\lambda_r+i\choose i}$, for $1\leq r<s\leq n$ and $1\leq i\leq \lambda_s$. A multiple of the standard multi-sequence will be called a standard multi-sequence.
\end{definition}

\q We continue with an analysis of extension multi-sequences arising from a general  module extension $V$.

\q Suppose $i,j\geq 1$ and $i+j\leq \lambda_s$. Then by Section 2, (1), we have 
$$Y(r,s)_iY(r,s)_jv_{\lambda+(i+j)(\ep_r-\ep_s)}={i+j\choose i}Y(r,s)_{i+j}v_{\lambda+(i+j)(\ep_r-\ep_s)}$$
i.e.,
$${\lambda_r+i+j\choose j}Y(s,t)_i v_{\lambda+i(\ep_s-\ep_t)}={i+j\choose i}Y(s,t)_{i+j}v_{\lambda+(i+j)(\ep_s-\ep_t)}$$
and taking the coefficients of $v_\infty$ we get: 
$${\rm (E)} \hskip 10pt  {\lambda_r+i+j\choose j}y(r,s)_j={i+j\choose i}y(r,s)_{i+j}, \hskip 55pt  i,j\geq 1, i+j\leq \lambda_s.$$

\q It will be important to know all sequences satisfying this relation (E) and we classify them in the next section. 

\q Now we consider a triple $1\leq r<s< t\leq n$ and consider the relations that exist between the extension sequences $(y(r,s)_i), (y(s,t)_j)$ and $(y(r,t)_k)$.  We take $1\leq i\leq \lambda_s$, $1\leq k\leq \lambda_t$. By Section 2, (3), we have
$$Y(r,s)_iY(r,t)_kv_{\lambda+(i+k)\ep_r-i\ep_s-k\ep_t}=Y(r,t)_kY(r,s)_iv_{\lambda+(i+k)\ep_r-i\ep_s-k\ep_t}$$
i.e.,
$${\lambda_r+i+k\choose k}Y(r,s)_iv_{\lambda+i(\ep_r-\ep_s)}={\lambda_r+i+k\choose i}Y(r,t)_kv_{\lambda+k(\ep_r-\ep_t)}$$
and taking the coefficients of $v_\infty$ we get: 
$${\rm (T1)} \hskip 10pt {\lambda_r+i+k\choose k}y(r,s)_i={\lambda_r+i+k\choose i}y(r,t)_k, \hskip 5pt 1\leq i\leq \lambda_s, 1\leq k\leq \lambda_t.$$

\q We take now $j,k\geq 1$ with $j+k\leq \lambda_t$.  By Section 2, (4), we have
$$Y(s,t)_jY(r,t)_kv_{\lambda+k\ep_r+j\ep_s-(j+k)\ep_t}=Y(r,t)_kY(s,t)_jv_{\lambda+k\ep_r+j\ep_s-(j+k)\ep_t}$$
i.e.,
$${\lambda_r+k\choose  k}Y(s,t)_jv_{\lambda_r+j(\ep_s-\ep_t)}={\lambda_s+j\choose j}Y(r,t)_kv_{\lambda+k(\ep_r-\ep_t)}$$
and taking the coefficients of $v_\infty$ we get: 
$${\rm (T2)}\hskip 10pt {\lambda_r+k\choose k}y(s,t)_j={\lambda_s+j\choose j}y(r,t)_k,  \hskip 40pt j,k\geq 1, j+k\leq \lambda_t.$$

\q We now take $1\leq i\leq j\leq \lambda_t$ and use relation Section 2, (2). We have 

\begin{align*}Y(s,t)_j &Y(r,s)_iv_{\lambda+i(\ep_r-\ep_s)+j(\ep_s-\ep_t)}\cr
=\sum_{h=0}^i &Y(r,s)_{i-h}Y(r,t)_h Y(s,t)_{j-h}v_{\lambda+i(\ep_r-\ep_s)+j(\ep_s-\ep_t)}
\end{align*}
i.e., \begin{align*} {\lambda_r+i\choose i}&Y(s,t)_jv_{\lambda+j(\ep_s-\ep_t)}\cr
=&\sum_{h=0}^i{\lambda_s-i+j\choose j-h}Y(s,t)_{i-h}Y(r,t)_hv_{\lambda+i(\ep_r-\ep_s)+h(\ep_s-\ep_t)}\cr
=&\sum_{0\leq h< i}{\lambda_s-i+j\choose j-h}{\lambda_r+i\choose h}Y(r,s)_{i-h}v_{\lambda+(i-h)(\ep_r-\ep_s)}\cr
+&{\lambda_s-i+j\choose j-i}Y(r,t)_iv_{\lambda+i(\ep_r-\ep_t)}
\end{align*}
and taking coefficients of $v_\infty$ we get
\begin{align*}  {\rm (T3a)}\hskip 10pt   &{\lambda_r+i\choose i}y(s,t)_j\cr
&= \sum_{0\leq h< i}{\lambda_s-i+j\choose j-h}{\lambda_r+i\choose h}y(r,s)_{i-h}+{\lambda_s-i+j\choose j-i}y(r,t)_i,\cr
& \phantom{} \hskip 200pt  \hbox{ for } 1\leq i\leq j\leq \lambda_t.
\end{align*}

\q Now suppose that $1\leq j\leq \lambda_t$, $j<i\leq \lambda_s+j$.  Then we have \\
$\lambda+(\ep_r-\ep_s)+j(\ep_s-\ep_t)\in \Lambda(n,r)$ and, using 
  Section 2, (2) we have
\begin{align*}Y(s,t)_j &Y(r,s)_iv_{\lambda+i(\ep_r-\ep_s)+j(\ep_s-\ep_t)}\cr
=\sum_{h=0}^j &Y(r,s)_{i-h}Y(r,t)_h Y(s,t)_{j-h}v_{\lambda+i(\ep_r-\ep_s)+j(\ep_s-\ep_t)}
\end{align*}
i.e.,
\begin{align*}
{\lambda_r+i\choose i}&Y(s,t)_jv_{\lambda+j(\ep_s-\ep_t)}\cr
&=\sum_{h=0}^i{\lambda_s-i+j\choose j-h} Y(s,t)_{i-h}Y(r,t)_hv_{\lambda+i(\ep_r-\ep_s)+h(\ep_s-\ep_t)}\cr
&=\sum_{h=0}^i{\lambda_s-i+j\choose j-h} {\lambda_r+i\choose h}
Y(s,t)_{i-h}v_{\lambda+(i-h)(\ep_r-\ep_s)}
\end{align*}
and taking coefficients of $v_\infty$ we get
\begin{align*}  {\rm (T3b)}\hskip 10pt   {\lambda_r+i\choose i}&y(s,t)_j= \sum_{h=0}^j {\lambda_s-i+j\choose j-h}{\lambda_r+i\choose h}y(r,s)_{i-h},\cr
& \phantom{} \hskip 120pt  \hbox{ for } 1\leq j\leq \lambda_t, j<i\leq \lambda_s+j.
\end{align*}

\q Solutions to the equations  (T1)-(T3b) (i.e., what will be called coherent triples) will be analysed extensively in subsequence sections.  There remains one further equation, coming from the commuting relation Section 2,(5).   So suppose $1\leq q,r,s,t \leq n$ and distinct and $q<r$, $s<t$. Then, for $1\leq i\leq \lambda_t$, $1\leq j\leq \lambda_r$,  we obtain
$$Y(q,r)_jY(s,t)_iv_{\lambda+i(\ep_s-\ep_t)+j(\ep_q-\ep_r)}=Y(s,t)_iY(q,r)_jv_{\lambda+i(\ep_s-\ep_t)+j(\ep_q-\ep_r)}$$
i.e.,
$${\lambda_s+i \choose i} Y(q,r)_jv_{\lambda+j(\ep_q-\ep_r)}={\lambda_q+j\choose j} Y(s,t)_i v_{\lambda+i(\ep_s-\ep_t)}$$
and taking coefficients of $v_\infty$ we obtain the \lq\lq commuting condition",
\begin{align*} {\rm (C)} \hskip 10pt {\lambda_s+i \choose i} y(q,r)_j&={\lambda_q+j\choose j} y(s,t)_i, \cr
&\hbox{ for } 1\leq q,r,s,t\leq n \hbox{ distinct with} q<r, s<t.
\end{align*}

\begin{proposition} Let $\lambda$ be a partition and let $(y(r,s)_i)$ be an extension multi-sequence for $\lambda$. Then the elements $y(r,s)_i\in K$ satisfy the equations (E),(T1)-(T3b) and  (C) above. Conversely, given elements $y(r,s)_i\in K$, $1\leq r<s\leq n$, $1\leq i\leq \lambda_s$, satisfying  (E),(T1)-(T3b) and  (C) then $(y(r,s)_i)$ is an extension multi-sequence for some extension of $S^dE$ by $K_\lambda$.
\end{proposition}

\begin{proof} We have already proved the first statement, in deriving (E),(T1)-(T3b) and  (C) above form an extension of $S^dE$ by $K_\lambda$. 

\q Now suppose that we have elements $y(r,s)_i\in K$, $1\leq r<s\leq n$, $1\leq i\leq \lambda_s$, satisfying  (E),(T1)-(T3b) and  (C). We  take a vector space $V(y)$ with basis elements $v_\alpha(y)$, $\alpha\in \Lambda(n,d)$, and $v_\infty(y)$. We make $V(y)$ into a $T$-module giving $v_\alpha(y)$ weight $\alpha$, for $\alpha\in \Lambda(n,d)$, and giving $v_\infty(y)$ weight $\lambda$. We define $K$-linear endomorphisms $\hatY(r,s)_i$,  of $V(y)$, for $1\leq r<s\leq n$, $1\leq i\leq \lambda_s$, by 

$$\hatY(r,s)_iv_\alpha(y)=\begin{cases} {\alpha_r\choose i}v_{\alpha-(\ep_r-\ep_s)}(y)+y(r,s)_iv_\infty(y), &\hbox{ if }   i\leq \alpha_s, \alpha= \lambda+i(\ep_r-\ep_s);\cr
{\alpha_r\choose i}v_{\alpha-(\ep_r-\ep_s)},   &\hbox{ if }   i\leq \alpha_s,\alpha\neq  \lambda+i(\ep_r-\ep_s);\cr
0,& \hbox{otherwise}
\end{cases}
$$
and making $\hatY(r,s)_i$ act as $0$ on $v_\infty(y)$. Then the relations (E),(T1)-(T3b) and  (C) guarantee that the endomorphisms $\hatY(r,s)_i$ satisfy the relations (1)-(5) of Section 2. Hence by Proposition 2.1, there is a $\hy(U)$-module action on $V(y)$ with $Y(r,s)_i$ acting via $\hatY(r,s)_i$, $1\leq r<s\leq n$, $1\leq i\leq \lambda_s$.  Moreover the $T$ action makes $V(y)$ into a $(\hy(U),T)$-module and hence, by Proposition 1.5, a $B$-module. Clearly the extension multi-sequence associated to $V(y)$,  with $v_\alpha=v_\alpha(y)$, $\alpha\in \Lambda(n,d)$ and $v_\infty=v_\infty(y)$ is $(y(r,s)_i)$. 

\end{proof}

\bf Notation: \rm It follows from Proposition 4.3  that  the set of all extension  multi-sequences for a partition  $\lambda$ forms a $K$-vector space, we denote this space by $E(\lambda)$.

\bs

\q We are now going to produce a $K$-space surjection $E(\lambda)\to \Ext^1(S^dE,K_\lambda)$, whose kernel is the space of the standard coherent multi-sequence. For $K_\lambda$ we take the field $K$ with $T$ acting with weight $\lambda$ and $Y(r,s)_i$ acting as $0$, for $1\leq r<s\leq n$, $1\leq i\leq \lambda_s$. For $y\in E(\lambda)$ let $V(y)$ be the $B$-module constructed via $y$, as above. We have a short exact sequence of $B$-modules
$$0\to K_\lambda \xrightarrow{i_y}  V(y)\xrightarrow{\pi_y} S^dE\to 0$$
where $i_y(1)=v_\infty(y)$ and  $\pi_y(v_\alpha(y))=e^\alpha$, for $\alpha\in \Lambda(n,d)$.

\q Starting with any $B$-module extension 
$$0\to K_\lambda \xrightarrow{i}  V \xrightarrow{\pi} S^dE\to 0$$
we have  $v_\infty=i(1)$ and $v_\alpha\in V$ such that $\pi(v_\alpha)=e^\alpha$ and we get an extension multi-sequence $y=(y(r,s)_i)$ given by 
$$Y(r,s)_iv_{\lambda+i(\ep_r-\ep_s)}={\lambda_r+i\choose i}v_\lambda+y(r,s)_iv_\infty$$
for $1\leq i\leq \lambda_s$. Now it is easy to see that the $B$-module map $\phi:V\to V(y)$ such that $\phi(v_\alpha)=v_\alpha(y)$, $\alpha\in \Lambda(n,d)$, and $\phi(v_\infty)=v_\infty(y)$, induces an isomorphisms of extensions (in the sense of for example  \cite{ARS}, Section I.5, or \cite{rotman},  Chapter 7) between the sequences $0\to K_\lambda \xrightarrow{i}  V \xrightarrow{\pi} S^dE\to 0$ and $0\to K_\lambda \xrightarrow{i_y}  V(y)\xrightarrow{\pi_y} S^dE\to 0$.  In particular every extension of $S^dE$ by $K_\lambda$ is equivalent  to one arising via an element of $E(\lambda)$.

\q Now suppose that $y,y'\in E(\lambda)$ give short exact sequences  $0\to K_\lambda \xrightarrow{i_y}  V(y)\xrightarrow{\pi_y} S^dE\to 0$ and $0\to K_\lambda \xrightarrow{i_{y'}}  V(y)\xrightarrow{\pi_{y'}} S^dE\to 0$.  If these sequences are isomorphic then there is a $B$-module isomorphism $\phi:V(y)\to V(y')$ such that the diagram 
 
$$
\begin{matrix} 0&\to   & K_\lambda& \xrightarrow{i_y} & V(y) & \xrightarrow{\pi_y} & S^dE & \to &0\cr
&& \downarrow &&\downarrow && \downarrow\cr
 0&\to   & K_\lambda& \xrightarrow{i_{y'}} & V(y') & \xrightarrow{\pi_{y'}} & S^dE & \to &0\cr
\end{matrix}$$
commutes (where the first and third vertical maps are the identities and the middle vertical map is $\phi$).  Commutativity implies that $\phi(v_\infty(y))=v_\infty(y')$ and also that $\phi(v_\alpha(y))=v_\alpha(y')$, for $\lambda\neq \alpha\in \Lambda(n,d)$. Moreover, we must have $\phi(v_\lambda(y))=v_\lambda(y')+cv_\infty(y')$, for some $c\in K$.  The condition for $\phi$ to be a module homomorphism reduces to the condition
$$\phi(Y(r,s)_iv_{\lambda+i(\ep_r-\ep_s)}(y))=Y(r,s)_i\phi(v_{\lambda+i(\ep_r-\ep_s)}(y))$$
i.e.,
$$\phi({\lambda_r+i\choose i}v_\lambda(y)+y(r,s)_iv_\infty(y))=Y(r,s)_iv_{\lambda+i(\ep_r-\ep_s)}(y')
$$
i.e.,
$${\lambda_r+i\choose i}(v_\lambda(y')+cv_\infty(y'))+y(r,s)_iv_\infty(y')={\lambda_r+i\choose i}v_\lambda(y')+y'(r,s)_iv_\infty(y')
$$
i.e.,
$$c{\lambda_r+i\choose i}+y(r,s)_i=y'(r,s)_i$$
for all $1\leq r<s\leq n$, $1\leq i\leq \lambda_s$.  So we have shown the following result.

\begin{lemma} Every extension $0\to K_\lambda\to V\to S^dE\to 0$ is equivalent to one of the form $0\to K_\lambda \xrightarrow{i_y}  V(y)\xrightarrow{\pi_y} S^dE\to 0$. Moreover, for  $y,y'\in E(\lambda)$ the associated short exact sequences 
$0\to K_\lambda \xrightarrow{i_y}  V(y)\xrightarrow{\pi_y} S^dE\to 0$ and $0\to K_\lambda \xrightarrow{i_{y'}}  V(y)\xrightarrow{\pi_{y'}} S^dE\to 0$ are equivalent  if and only if $y'-y$ is a multiple of the standard multi-sequence.

\end{lemma}

\q We now follow a standard procedure (cf.  \cite{ARS}, Section I.5, or \cite{rotman},  Chapter 7) to relate extensions of $S^dE$ by $K_\lambda$  to elements of  $\Ext^1_B(S^d E,K_\lambda)$ and from this produce a linear epimoprhism  $E(\lambda)$ to $\Ext^1_B(S^d E,K_\lambda)$ with kernel consisting of the standard coherent multi-sequences.

\q We fix an injective module $I$ containing $K_\lambda$.  Let $Q=I/K_\lambda$ and let $\nu:I\to Q$ be the natural map.  Applying $\Hom_B(S^d E,-)$ to the short exact sequence $0\to K_\lambda \to I\to Q\to 0$ we obtain an exact sequence 
$$\Hom_B(S^dE,I)\to \Hom_B(S^dE,Q)\to \Ext^1_B(S^dE,K_\lambda)$$
 so we may identify $\Ext^1_B(S^dE,K_\lambda)$ with  
$\Hom_B(S^dE,Q)/\Hom_B(S^dE,Q)_0$, where 
$$\Hom_B(S^dE,Q)_0=\{\nu\circ\theta \vert \theta\in \Hom_B(S^dE,I)\}.$$

\q  We write $E(S^dE, K_\lambda)$ for the set of equivalence classes of $B$-module extensions  
$$ 0\to K_\lambda\xrightarrow{i}  V\to S^dE\xrightarrow{\pi}  0\eqno{({\mathcal E)}}.$$
 For such an extension we have, by injectivity, a homomoprhism $\theta:V\to I$ such that $\theta\circ i: K_\lambda\to I$ is inclusion.  We obtain a homomorphism \\
   ${\bar\theta}:S^dE\to Q$, given by ${\bar\theta}(\pi(v))=\theta(v)+K_\lambda$, $v\in V$, and hence an element \\
    ${\bar\theta}+\Hom_B(S^dE,Q)_0$ of $\Hom_B(S^dE,Q)/\Hom_B(S^dE,Q)_0$. One checks that this element depends only on the equivalence class $[{\mathcal E}]$ of ${\mathcal E}$. Hence we have  a map $\Phi:E(S^dE,K_\lambda)\to \Hom_B(S^dE,Q)/\Hom_B(S^dE,Q)_0$ and, for extensions ${\mathcal E}$ and ${\mathcal F}$ we have $\Phi([{\mathcal E}])=\Phi([{\mathcal F}])$ if and only if ${\mathcal E}$ and ${\mathcal F}$ are equivalent.
 
 \q Now for $y\in E(\lambda)$ we have the extension $0\to K_\lambda \xrightarrow{i_y}\ V(y)\xrightarrow{\pi_y}  S^dE\to 0$, which we now call ${\mathcal E}(y)$.  Hence we get a map \\
 $\Psi:E(\lambda)\to \Hom_B(S^dE,Q)/\Hom_B(S^dE,Q)_0$, defined by $\Psi(y)=\Phi([{\mathcal E}(y)])$. It is easy to check that $\Phi([{\mathcal E}(y+y')])=\Phi([{\mathcal E}(y)])+\Phi([{\mathcal E}(y')])$, for $y,y'\in E(\lambda)$ and that  $\Phi([{\mathcal E}(cy)])=c\Phi([{\mathcal E}(y)])$, for $y\in E(\lambda)$, $y\in E(\lambda)$. Hence we have a $K$-linear map $\Psi:E(\lambda)\to  \Hom_B(S^dE,Q)/\Hom_B(S^dE,Q)_0$. It is the composte of two surjections and hence surjective. Moreover, for $y\in E(\lambda)$ we have $\Psi(y)=0$ if and only if $\Phi([{\mathcal E}(y)])=0$, i.e.,  if and only if ${\mathcal E}(y)$ is split, which by Lemma 1.3, is if and only if $y$ is  a standard coherent multi-sequence. Since $\Hom_B(S^dE,Q)/\Hom_B(S^dE,Q)_0$ is isomorphic to $\Ext^1(S^dE,K_\lambda)$ we have shown the following.
 
 \begin{proposition} (i) There is a surjective $K$-linear map from $E(\lambda)$ to\\
  $\Ext^1(S^dE,K_\lambda)$ with kernel spanned by the standard multi-sequence.
  
  (ii) We have
  $$\dim \Ext^1_B(S^d E, K_\lambda)= \begin{cases} \dim E(\lambda), & \hbox{ if } \lambda \hbox{ is  James}\cr
  \dim E(\lambda) -1,  & \hbox{ if } \lambda \hbox{ is  not James.}\
 \end{cases}$$
  \end{proposition}

  \q This prompts the following definition.

  \begin{definition}   We shall say that a partition $\lambda$ is split if every coherent multi-sequence  for $\lambda$ is a multiple of the standard sequence.
  \end{definition}
  
  \q Thus  a partition $\lambda$ is split if and only if we have $\Ext^1_{B(N)}(S^d E,K_\lambda)=0$ (where $d$ is the degree of $\lambda$, where $N$ is at least the number of parts of $\lambda$ and $E$ is the natural $G(N)$-module).

 \bs\bs\bs\bs


\section{Extension sequences for two part partitions}

\q Let $\lambda=(\lambda_1,\ldots,\lambda_n)$ be a partition of length at least $2$. 
We fix $1\leq r<s\leq n$ and set $a=\lambda_s$, $b=\lambda_t$.  According to Section 4, (E), we have 
$${a+i+j\choose j}y(r,s)_i={i+j\choose j}y(r,s)_{i+j}$$
for all $i,j\geq 1$ with $i+j\leq b$. 

\bs

\begin{definition} Let $(a,b)$ be a two part partition. By an {\em extension sequence} for $(a,b)$ we mean a sequence $(x_i)=(x_1,\ldots,x_b)$ of elements of $K$ such that 
$${a+i+j\choose j}x_i={i+j\choose j}x_{i+j}$$
for all $i,j\geq 1$ with $i+j\leq b$.
\end{definition}

\q Since a general extension multi-sequence is a family  of extension sequences it is important to thoroughly understand the possible extension sequences.  We describe all such in this section. We then specialise to the case of $n=2$ to recover the main result of Erdmann's paper, \cite{Erdmann}

\q We now make some elementary observations on  extension sequences.

\begin{remark} \rm Suppose $(x_i)$ is an extension sequence for a  two part partition $(a,b)$  and $c+sp^k\leq b$, with  $0<c<p^k$. Then we have 
$${a+c+sp^k\choose c}x_{sp^k}={a+c+sp^k\choose sp^k}x_c={c+sp^k\choose sp^k}x_{c+sp^k}$$
and moreover, we have ${c+sp^k\choose sp^k}=1$, by Lucas's formula. So we deduce
$$x_{c+sp^k}={a+c+sp^k\choose c}x_{sp^k}={a+c+sp^k\choose sp^k}x_c.$$
\end{remark}

\begin{lemma} Suppose  that $(a,b)$ is a two part partition.  An extension sequence for $(a,b)$  which is $0$ at all powers of $p$ is identically $0$.
\end{lemma}

\begin{proof} Suppose $x_1,\ldots,x_b$ is such a sequence.  By induction it is enough to prove that $x_b=0$ and we may assume that $b=sp^k$, for some $1<s\leq p-1$. Taking $i=(s-1)p^k$, $j=p^k$ in the defining relation we get ${sp^k\choose p^k}x_{sp^k}=0$, and ${sp^k\choose p^k}=s\neq 0$ so that $x_{sp^k}=0$, as required.
\end{proof}

\begin{remark} Let $(a,b)$ be a two part partition.  We have the standard extension sequence $(x_i)$ for $(a,b)$ given by  $x^\st_i={a+i\choose i}$, $1\leq i\leq b$. 
\end{remark}

\begin{remark}  \rm A two  part partition $(a,b)$ is James if and only if the standard extension sequence for $(a,b)$ is the zero sequence.
\end{remark}

\begin{remark}   \rm There is a non-zero extension sequence for a James partition $(a,b)$. To see this consider the sequence of integers $N_i={a+i\choose i}$, $1\leq i\leq b$.  Note that these integers satisfy the condition 
$${a+i+j\choose j}N_i={i+j\choose i}N_{i+j}$$
for all $i,j\geq 1$ with $i+j\leq b$. Let $h$ be the largest positive integer such that $p^h$ divides all $N_i$ and set 
$N_i'=N_i/p^h$, for $1\leq i\leq b$.     The we also have 
$${a+i+j\choose j}N_i'={i+j\choose i}N_{i+j}'$$
for all $i,j\geq 1$ with $i+j\leq b$ so that taking $N_1',\dots,N _b'$ modulo $p$ we obtain a non-zero  extension sequence. We call this the {\em canonical extension sequence}. We shall explore this in greater depth (for James partitions of any length) in Section 7.  In particular we have $E(a,b)\neq 0$.
\end{remark}

\q It is easy to describe the extension sequences for a James partition. 

\begin{lemma} Let $(a,b)$ be a two part James partition and let $\beta=\len_p(b)$. A sequence $x_1,\ldots, x_b$ is an extension sequence if and only if 
$$x_{p^\beta}=2x_{2p^\beta}=\cdots=b_\beta x_{b_\beta p^\beta}$$
and $x_i=0$ when $i$ is not  a multiple of $p^\beta$.
\end{lemma}

\begin{proof}  Let $k\leq \beta$ and suppose that $1<s\leq b_k$. Then, taking $i=(s-1)p^k$ and $j=p^k$ in the defining condition for an extension sequence,  we have 
$${a+sp^k\choose p^k}x_{(s-1)p^k}={sp^k\choose (s-1)p^k}x_{sp^k}.$$
Now $k<\val_p(a+1)$ so we have  $a_k=p-1$ and therefore $(a+sp^k)_k=s-1$ and hence, by Lucas's formula,
$(s-1)x_{(s-1)p^k}=sx_{sp^k}$ and so we get 
$$x_{p^k}=2x_{2p^k}=\cdots=b_k x_{b_k p^k}.\eqno{(*)}$$

\q We now apply the defining condition with $i=(p-1)p^{\beta-1}$ and $j=p^{\beta-1}$. We get 
$${a+p^\beta\choose p^{\beta-1}}x_{(p-1)p^{\beta-1}}={p^\beta\choose (p-1)p^{\beta-1}}x_{p^\beta}=0.$$
But ${a+p^\beta\choose p^{\beta-1}}=(a+p^\beta)_{\beta-1}=p-1\neq 0$ so that $x_{(p-1)p^{\beta-1}}=0$ 
and so by (*), $x_{sp^{\beta-1}}=0$ for $1\leq s\leq p-1$.  But now $p^\beta-1<b$ and $(a,p^\beta-1)$ is James so we may assume,  by induction, that $x_i=0$ for $1\leq i<p^\beta$.  Now Remark 5.2 gives that in fact $x_i=0$ for all $i\leq b$, with $i$ not a multiple of $p^\beta$. Furthermore the space of extension sequences for $(a,b)$ is at most one dimensional, by (*). However, we know, by Remark 5.6, that there is a non-zero extension sequence, it follows that the space of such sequences is one dimensional and is as described in the statement of the Lemma.
\end{proof}

\begin{remark}    Lemma 5.7 implies that an extension sequence for a two part James partition is a multiple of the canonical sequence.
\end{remark}

\begin{definition}  A two part partition $(a,b)$ will be called a pointed partition if  $b$ can be written in the form $b=\hatb+p^\beta$, with $\beta=\len_p(b)$ with   $\hatb<p^{\val_p(a+1)}$   and $\beta>\val_p(a+1)$.  In this case we call the extension sequence $(x^\pt_i)$, $1\leq i\leq b$, defined by 
$$x^\pt_i=\begin{cases} 1, & \hbox { if } i=p^\beta;\cr
0, & \hbox{ otherwise}.
\end{cases}$$
the point sequence for $(a,b)$.

\end{definition}

\begin{remark}  \rm  Note that the point sequence is non-standard since, with the above notation, the $p^v$th term in the standard sequence  ${a+p^v\choose p^v}=a_v+1$ is non-zero, where $v=\val_p(a+1)$.

\end{remark}

\begin{lemma} Suppose that $(a,b)$ is a two part partition which is not James.  Let $v=\val_p(a+1)$ and $\beta=\len_p(b)$. 
Then $(a,b)$ is split unless we may write $b=\hatb+p^\beta$, where $\hatb<p^v<p^\beta$   and  in this case  $E(a,b)$  is spanned by the standard sequence and the point sequence.
\end{lemma}

\begin{proof}

First assume that  $b=\hatb+p^\beta$, with $\hatb< p^v<p^\beta$.  We check that the point sequence $x_1,\ldots,x_b$ is indeed  an   extension sequence.   If $\hatb=0$ then we need only to check the defining relation for $i+j=p^\beta$ and in that case ${i+j\choose i}=0$ and $x_i=0$ (for $i<p^\beta$) so the relation holds.

\q So we suppose that $\hatb>0$. We may  assume inductively that $x_1,\ldots,x_{b-1}$ is an extension sequence and again we need only to check the defining relations with $i+j=b$.  Moreover $x_b=0$, as  $b\neq p^\beta$,  and so we must check that ${a+b\choose j}x_i$ is zero for $b=i+j$.   Since $x_i=0$ for $i\neq p^\beta$ we may assume that  $i=p^\beta$, $j=\hatb$ and we must check that ${a+b\choose \hatb}=0$. We choose $k$ minimal such that $\hatb_k\neq 0$. Then $a_k=p-1$ (since $\hatb<p^v$) and hence  $(a+b)_k=b_k-1$ so that ${(a+b)_k\choose \hatb_k}=0$ and ${a+b\choose \hatb}=0$, by Lucas's Formula.

\q We now prove that  $(a,b)$   is split unless it has the above form and in that case the space of extensions sequences is spanned by those described above. If $b=1$ then $(a,b)$ is not pointed and (since $(a,b)$ is not James) we have $v=0$ and $x_1^\st\neq 0$ so that an extension sequence $x_1$ is a multiple of the standard sequence.

\q We now consider the case  $b>1$. Let $x_1,\ldots,x_b$ be an extension sequence for $(a,b)$. 

\q Assume that $(a,b-1)$ is not James. If  $(a,b-1)$ is  split then $x_1,\ldots,x_{b-1}$ is a standard sequence  and hence, subtracting a standard sequence  for $(a,b)$ we may assume that $x_i=0$ for $i<b$. Hence, by Lemma 5.3 $(x_i)$ is the zero sequence unless $b=p^\beta$.  But we have $p^\beta-1\geq p^v$  so that $\beta>v$ and $(a,b)$ is of the required form and the extension sequence is also of the required form.  If $(a,b-1)$ is not split we may assume inductively that we can  write $b-1=c+p^t$ with $c<p^v <p^t$. Moreover, after subtracting a  linear combination of the standard solution and the point solution for $(a,b-1)$, we may assume that $x_i=0$ for all $i<b$. If $b$ is not a power of $p$ this implies that $x_b$ is also zero by Remark 5.2.  Hence we can assume that $b=p^\beta$. Then we get $p^{\beta-1}-1< p^v <p^{\beta-1}$, and this is not possible.

\q It remains to consider the case in which $(a,b-1)$ is James.  So we have $b-1<p^v$  and, since $(a,b)$ is not James, $b=p^v$.   By  Lemma 5.7   we have
$$x_{p^{v-1}}=2x_{2p^{v-1}}=\cdots=(p-1)x_{(p-1)p^{v-1}}$$
and $x_i=0$ for $i<b$ and $i$ not divisible by  $p^{v-1}$. 

\q Taking $i=(p-1)p^{v-1}$, $j=p^{v-1}$ in the defining relation we get
$${a+p^v\choose p^{v-1}}x_{(p-1)p^{v-1}}=0.$$
Moreover $a_{v-1}=p-1$ so that ${a+p^v\choose p^{v-1}}=p-1$, by Lucas's Formula, and 
 $x_{(p-1)(p^{v-1})}=0$. Hence we have $x_i=0$ for $0<i<b=p^v$.  Hence the space of extensions sequences for $(a,b)$ is at most one dimensional. Hence the space of extension sequences consists of standard solutions and $(a,b)$ is split.
\end{proof}

\q We now take $n=2$ and consider the extensions for the general linear group $G(2)$, over $K$. Let $(r,s), (t,u)\in \Lambda^+(n,r)$ with $r+s=t+u$. Then we have $\Ext^1_{G(2)}(\nabla(r,s),\nabla(t,u))=0$ unless $(r,s) >  (t,u)$.  Then have $s < u$. Moreover, we have $\nabla(r,s)=\nabla(r-s,0)\otimes D^{\otimes s}$ and \\
$\nabla(t,u)=\nabla(t-s,,u-s)\otimes D^{\otimes s}$, where $D$ is the one dimensional determinant module.  Hence we have 
$$\Ext^1_{G(2)}(\nabla(r,s),\nabla(t,u))
=\Ext^1_{G(2)}(\nabla(r-s,0),\nabla(t-s,u-s))$$ i.e.,
$$\Ext^1_{G(2)}(\nabla(r,s),\nabla(t,u))=\Ext^1_{G(2)}(S^d E, \nabla(t-s,u-s)).$$
where $d=r-s$.  Hence, from Lemmas 5.7  and 5.11, we have the following.

\begin{corollary}   Let $(r,s), (t,u)$ be two part partitions of the same degree. Then 
$$\Ext^1_{G(2)}(\nabla(r,s),\nabla(t,u))=\begin{cases} K, & \hbox{ if }  (t-s,u-s) \hbox{ is James  or pointed};\cr
0, & otherwise
\end{cases}
$$
i.e, 
$$\Ext^1_{G(2)}(\nabla(r,s),\nabla(t,u))=
\begin{cases} K, & \hbox{ if } u-s< p^v\cr
 &\hbox{ or } u-s-p^l < p^v< p^l;\cr
0, & \hbox{ otherwise}
\end{cases}
$$
where $v=\val_p(t-s+1)$ and $l=\len_p(u-s)$.
\end{corollary}

\q We now consider extensions for ${\rm SL}_2(K)$, as in \cite{Erdmann}. For $r\geq 0$ we write  simply $\nabla(r)$ for the $r$th symmetric power of the natural representation. We have $\Ext^1_{{\rm SL}_2(K)}(\nabla(r),\nabla(s))=0$ unless $r-s=2m$ for some positive integer  $m$. Moreover, in this case we have $\Ext^1_{{\rm SL}_2(K)}(\nabla(r),\nabla(s))=\Ext^1_{G(2)}(\nabla(r,0),\nabla(s+m,m))$. Hence we obtain the following result of Erdmann, \cite{Erdmann},   (3.6) Theorem. 

\begin{corollary} For $r,s\geq 0$ we have 
$$\Ext^1_{{\rm SL}_2(K)}(\nabla(r),\nabla(s))
=\begin{cases} K, &\hbox{ if } r-s=2m \hbox{ is  positive even and } m < p^v \cr
& \hbox{ or } m-p^l<p^v<p^l;\cr
0, & \hbox{ otherwise.}
\end{cases}
$$
where $v=\val_p(s+m+1)$ and $l=\len_p(m)$.

\end{corollary}

\q In fact A. Parker has given a precise recursive description of all higher extensions  $\Ext^i_{{\rm SL}_2(K)}(\nabla(r),\nabla(s))$, $r, s\geq 0$,  in \cite{Parker}, Section 5 (in the equivalent dual formulation for Weyl modules).

\bs\bs\bs\bs

\newpage


\section{Coherent Triples}

In the sections that follow, particularly Sections 8, 10, and 11 we shall be working extension sequences for $3$ part partitions. We here abstract the notion of a coherent triple from the relations given in Section 4.

\begin{definition}

We fix a three part partition $\lambda=(a,b,c)$. By a   coherent triple of extension sequences for $\lambda$ we mean a triple  \\
 $((x_i)_{1\leq i\leq b}, (y_j)_{1\leq j\leq c}, (z_k)_{1\leq i\leq c})$ of extension sequences satisfying the following relations:

\begin{align*}
&{\rm (T1)}\     {{a+i+k}\choose{k}}x_i={{a+i+k}\choose{i}}z_k,  \hskip 85pt 1\leq i\leq b, 1\leq k\leq c; \\ \cr
&{\rm (T2)}\    {{a+k}\choose{k}}y_j={{b+j}\choose{j}}z_k, \hskip 110pt 1\leq j,k\leq c,   j+k\leq c; \\ \cr
&{\rm (T3a)}\    {{a+i}\choose{i}}y_j=\sum_{s=0}^{i-1}{{b+j-i}\choose{j-s}}{{a+i}\choose{s}}x_{i-s}+{{b+j-i}\choose{j-i}}z_i,  \hskip 0pt 1\leq  i\leq j\leq c;\\ \cr
&{\rm (T3b)}\   {{a+i}\choose{i}}y_j=\sum_{s=0}^{j}{{b+j-i}\choose{j-s}}{{a+i}\choose{s}}x_{i-s},  \hskip 30pt 1\leq j\leq c,  j<i\leq b+j. \\ \cr
\end{align*}

\end{definition}

\q Thus if $\lambda=(\lambda_1,\ldots,\lambda_n)$ is a partition of length $n$ and $(y(t,u)_i)$ is a coherent multi-sequence for $\lambda$ then for 
$1\leq q<r<s\leq n$ we may extract the coherent triple $(x_i),(y_j),(z_k)$ for $(\lambda_q,\lambda_r,\lambda_s)$  given by $x_i=y(q,r)_i$, $y_j=y(r,s)_j$,  and $z_k=y(q,s)_k$,    for $1\leq i\leq r$,  $1\leq j,k\leq \lambda_s$.

\begin{remark} \rm  In any coherent triple $((x_i),(y_j),(z_k))$ the extension sequence $(z_k)$ is determined by $(x_i)$ and $(y_j)$. (Take $i=j$ in relation (3a).) 

\end{remark}

\q  We have the space $E(\lambda)$ of all coherent triples for $\lambda=(a,b,c)$. It will emerge (as a result of explicit calculation)  that in fact $E(\lambda)$ is at most two dimensional.

\bs

\q Note that for   $\lambda=(a,b,c)$ the standard multi-sequence for $\lambda$ gives the standard triple 
$x_i^\st={a+i\choose i}$, $1\leq i\leq b$,  $y_j^\st={b+j\choose j}$, $1\leq j\leq c$, $z_k^\st={a+k\choose k}$, $1\leq k\leq c$.

\bs\bs\bs\bs


\section{Extensions dimensions for James partitions}

\q Recall that for a partition $\lambda$ we write $E(\lambda)$ for the space of extension multi-sequences for $\lambda$. In this section we  determine the dimension of $E(\lambda)$, for $\lambda$ a James partition. Note that, by Proposition 4.5 this is the dimension of $\Ext^1_{B(N)}(S^d E,K_\lambda)$  (where $d$ is the degree of $\lambda$, $N$ is at least the number of parts of $\lambda$ and $E$ is the natural $G(N)$-module).

\begin{definition}  
A  partition  $\lambda$  will be called constrained if   $\Ext^1_{B(N)}(S^d E,K_\lambda)$  is at most one dimensional (where $d$, $N$ and $E$ are as above).
\end{definition}

\q For a James partition $\lambda$, this is equivalent to the condition that $E(\lambda)$ is at most one dimensional.

\q  We extend the notion of the definition of the canonical multi-sequence defined in Section 5 for a  two part James partition to an arbitrary James  partition $\lambda=(\lambda_1,\ldots,\lambda_n)$ of length $n\geq 2$. 

 \begin{definition} Let  $\lambda=(\lambda_1,\ldots,\lambda_n)$ be a James partition with $n\geq 2$ parts.  We define the James index $\JI(\lambda)$ to be the largest positive integer $k$ such that $p^k$ divides all integers ${\lambda_i+j\choose j}$, for $1\leq i<n$ and $1\leq j\leq \lambda_{i+1}$.
 \end{definition}

 \begin{definition}    Let $\lambda=(\lambda_1,\ldots,\lambda_n)$ be a James partition of length $n$ then we have the canonical coherent multi-sequence $(y^\can(r,s)_i)$, where $y^\can(r,s)_i$, for $1\leq r<s\leq n$, $1\leq i\leq \lambda_s$,  is obtained by taking ${\lambda_r+i\choose i}/p^{\JI(\lambda)}$ modulo $p$.
 
 \end{definition}

\q  We begin by deriving a concrete description of the canonical multi-sequence.

\q Let $(a,b)$ be a James $2$-part partition. We put $v=\val_p(a+1)$, $w=\val_p(b)$. We write $a=(p^v-1)+p^vA$ and $b=p^w B$.  Thus we have 
$${a+b\choose b-1}={(p^v-1)+p^vA+p^wB\choose p^wB-1}={p^wB-1+p^v(A+1)\choose p^wB -1}$$
and hence 
$${a+b\choose b-1}\equiv 1 \hskip 10pt \hbox{ mod $p$}  \eqno{(1)}$$
 by Lucas's Formula.
 
 \q Now we have
 $${a+b\choose b}={a+1 \over b}{a+b\choose b-1}$$
 which shows that $\val_p{a+b\choose b}\geq v-w$ and also gives 
 $$p^wB{a+b\choose b}=p^v(A+1){a+b\choose b-1}$$
 so that 
 $$B\left[ { {a+b\choose b} \over  p^{v-w}}\right]=(A+1){a+b\choose b-1}$$
 so  from (1) we get 
 $$b_w \left[ { {a+b\choose b} \over  p^{v-w}}\right] \equiv a_v+1  \hskip 10pt \hbox{ mod $p$}. \eqno{(2)}$$
 In particular ${a+b\choose b}/p^{v-w}$ is non-zero modulo $p$ so that 
 $$\val_p{a+b\choose b}=v-w \eqno{(3).}$$

 \q We now assume that $\lambda=(\lambda_1,\ldots,\lambda_n)$ is a partition of length $n\geq 2$. We adopt the following notation:  $v_i=\val_p(\lambda_i+1)$, for $1\leq i<n$, and  $l_i=\len_p(\lambda_i)$ for $2\leq i\leq n$. In the light of (2),(3) we now have have the following description of the James index.
 
 \begin{proposition}  Let $\lambda=(\lambda_1,\ldots,\lambda_n)$ be a James partition with $n\geq 2$ parts. Then $\JI(\lambda)$ is the minimum  of $v_1-l_2,\ldots,v_{n-1}-l_n$.  Moreover we have
 $$y^\can(r,s)_i=\begin{cases} ((\lambda_r)_{v_r}+1)/ t, & \hbox{ if }  v_r-l_s=\JI(\lambda) \hbox{ and } i=tp^{l_s}; \cr
 0, & \hbox{ otherwise.} 
 \end{cases}$$
 
 \end{proposition}

\q We fix a three part James partition $(a,b,c)$ and study the corresponding coherent triples. We set $v=\val_p(a+1), w=\val_p(b+1)$ and $\beta=\len_p(b), \gamma=\len_p(c)$.

\q We specialise the above to the case of $(a,b,c)$.   (For a two part partition the canonical sequence, at least up to multiplication by scalars, is given by Lemma 5.8.)

 \begin{proposition}  Let $\lambda=(a,b,c)$. Then   $\JI(\lambda)$ is the minimum of $v-\beta$ and $w-\gamma$.   Moreover the canonical coherent triple $(x^\can_i),(y^\can_j),(z^\can_k)$ for $\lambda$ is given as follows:

   \begin{align*}  x^\can_i=\begin{cases}
   (a_v+1)/t, & \hbox{ if } v-\beta=\JI(\lambda) \hbox{ and } i=tp^\beta \hbox{\,  for some\, } 1\leq t\leq b_\beta ;\cr
   0, & \hbox{ otherwise}
   \end{cases}
    \end{align*}
    \begin{align*}  y_j^\can=\begin{cases}
   (b_w+1)/t, & \hbox{ if } w-\gamma=\JI(\lambda) \hbox{ and } j=tp^\gamma \hbox{\,  for some\, } 1\leq t\leq c_\gamma  ;\cr
   0, & \hbox{ otherwise}
   \end{cases}
    \end{align*}
        \begin{align*}  z_k^\can=\begin{cases}
   (a_v+1)/t, & \hbox{ if } v-\gamma=\JI(\lambda) \hbox{ and } k=tp^\gamma \hbox{\,  for some\, } 1\leq t\leq c_\gamma  ;\cr
   0, & \hbox{ otherwise.}
   \end{cases}
    \end{align*}

 \end{proposition}

\begin{remark}    \rm Suppose $\beta>\gamma$. Then  $v=\len_p(b+p^\gamma)$ if and only if $b=p^v-1$.  Certainly  $v=\len_p(b+p^\gamma)$ if  $b=p^v-1$. Now suppose $v=\len_p(b+p^\gamma)$.  We write $b=(p^w-1)+p^wB$. Then we have $b+p^\gamma=(p^\gamma-1)+p^w(B+1)$ so that $\len_p(B+1)=v-w$.  If $B$ does not have the form $p^k-1$ for some $k$ then also $\len_p(B)=v-w$ and $\len_p(b)=v$, which is impossible since $\beta<v$ by the James condition. Hence we have $B=p^k-1$ and $\len_p(B+1)=k=v-w$ so that $b=(p^w-1)+p^w(p^{v-w}-1)=p^v-1$.

\end{remark}

\q We note that in checking the coherence  conditions for a multi-sequence $(y(r,s)_i)$ of extension sequences  for a James  partition sequence $\lambda$ it is sufficient to check coherent of all relevant triples since in the  commuting relation  (C), for a commuting quadruple,  both sides are zero (by the James condition).

\begin{lemma}    (i) If $\beta=\gamma$ then $\lambda$ is constrained and   \\ 
$x^\can_{p^\beta}=z^\can_{p^\gamma}=(a_{\beta+1}+1)y^\can_{p^\gamma}.$

(ii) If $\beta>\gamma$ and $(a,b+p^\gamma)$ is not  James   then $\lambda$ is constrained and $(y^\can_j)=0$,  $(z^\can_k)=0$.

(iii)  If $\beta>\gamma$ and $(a,b+p^\gamma)$ is James  then $\dim E(\lambda)=2$ and in any coherent triple we have $(z_k)=0$ and $(x_i)$ and $(y_j)$ are multiples of the canonical sequences, for $(a,b)$ and $(b,c)$.

\end{lemma}

\begin{proof}  We leave it to the reader to check that in (i) and (ii) the canonical triples are as described.

(i) We have $w\leq \beta+1\leq v$ so that $v-\beta\geq w-\gamma$. Hence $(y^\can_k)$ is non-zero and, subtracting a multiple of the canonical sequence from an arbitrary coherent triple we can obtain a coherent triple $(x_i),(y_j),(z_k)$ in which $(y_j)=0$.  Now $a_\beta=p-1$ so that, taking $i=k=p^\beta$ in (T1), we obtain $x_{p^\beta}=z_{p^\beta}$.  

\q  We note that $\beta=\gamma<w$ implies that $b=p^{\beta+1}-1$. We take $j=p^\beta$, $i=p^{\beta+1}$ in relation (T3b).  The only non-zero term $x_{i-s}$ is obtained by taking $s=p^\beta$ and we obtain 
${a+p^{\beta+1}\choose p^\beta}x_{(p-1)p^\beta}=0$. Moreover, we have ${a+p^{\beta+1}\choose p^\beta}=(p-1)$, by Lucas's Formula. So we have $x_{(p-1)p^\beta}=0$ and hence $(x_i)=0$ and therefore also $(z_k)=0$. Hence $(x_i),(y_j),(z_k)$ is the zero triple and every coherent triple is a multiple of the canonical triple.

(ii) We have $\len_p(b+p^\gamma)\leq \beta+1$ and, since $(a,b)$ is James but  $(a,b+p^\gamma)$ is not James we have $\len_p(b+p^\gamma)=\beta+1\geq v$. Hence $\beta+1=v$ and so $b=p^v-1$, by Remark 7.6 and so $\beta+1=v=w$. Hence $v-\beta<w-\gamma,v-\gamma$ so that $(x^\can_i)\neq 0$.  Subtracting a multiple of the canonical triple  from an arbitrary coherent triple we can thus obtain a coherent triple $(x_i),(y_j),(z_k)$ in which $(x_i)=0$.  Taking $i=j=p^\gamma$ in (T1) we obtain $z_{p^\gamma}=0$ and hence the sequence $(z_k)=0$. Taking $j=p^\gamma$, $i=p^{\beta+1}$ in (T3b) we get ${a+p^{\beta+1}\choose p^{\beta+1}}y_{p^\gamma}=0$, i.e., $(a_v+1)y_{p^\gamma}=0$, therefore $y_{p^\gamma}=0$ and hence the sequence $(y_j)$ is $0$.  Hence $(x_i),(y_j),(z_k)$ is the zero triple and every coherent triple is a multiple of the canonical triple.

(iii) We show first that in this case we have that in any coherent triple $(z_k)=0$. Since $(a,c)$ is James it is enough to show that $z_{p^\gamma}=0$. Now this follows by (T1) for $i=k=p^\gamma$, using the fact that $x_{p^\gamma}=0$, since $\beta>\gamma$ and $(a,b)$ is James.  Therefore $\dim E(\lambda)\leq 2$ and so it is enough to show that there is a coherent triple $(x_i),(y_j),(z_k)$ in which  $(x_i)$ non-zero and $(y_j)=0$, $(z_k)=0$ and one in which $(x_i)=0$, $(y_j)$ is non-zero and $(z_k)=0$.

\q We first check that the triple $(x_i),(y_j),(z_k)$ with 
$$x_i=\begin{cases}1/t,  &\hbox{ if } i=tp^\beta, 1\leq t\leq b_\beta; \cr
0, & \hbox{ otherwise}
\end{cases}$$
and $(y_j)=0$, $(z_k)=0$ is coherent. 

\q For (T1) we require ${a+tp^\beta+k\choose k}=0$, for $1\leq k\leq c$, $1\leq t\leq b_\beta$. This is true since $(a+tp^\beta,c)$ is James, as $\val_p(a+tp^\beta+1)\geq \beta>\gamma$. In (T2) both sides are $0$. For (T3a) we require 
$$\sum_{s=0}^{i-1} {b+j-i\choose j-s}{a+i\choose s}x_{i-s}=0$$
for $1\leq i\leq j\leq c$. However, if in this equation $x_{i-s}\neq 0$ then we have $i-s=tp^\beta$, for some $1\leq t\leq b_\beta$ and so $i\geq p^\beta$ and this is impossible since $i\leq c$ and $\gamma<\beta$. Hence (T3a) holds. 

\q For (T3b) we require 
$$\sum_{s=0}^{j} {b+j-i\choose j-s}{a+i\choose s}x_{i-s}=0$$
for $1\leq j\leq c$, $j<i\leq b+j$. Now, $x_{i-s}$ is zero unless $i-s=tp^\beta$ and then ${a+i\choose s}={a+tp^\beta+s\choose s}$ is zero for $s\neq 0$ (as $(a+tp^\beta,c)$ is James). Hence the condition to be checked is ${b+j-tp^\beta\choose j}=0$ for $1\leq j\leq c$, and this is true since $(b-tp^\beta,c)$ is James.


\q It remains to check that the triple 
$(x_i),(y_j),(z_k)$ with $(x_i)=0$, with 
$$y_j=\begin{cases}1/t,  &\hbox{ if } j=tp^\gamma, 1\leq t\leq c_\gamma; \cr
0, & \hbox{ otherwise}
\end{cases}$$
and $(z_k)=0$,  is coherent. 

\q Certainly (T1) is satisfied since both sides are $0$.  Also in  (T2) both sides are zero since $(a,c)$ and $(b,c)$ are James. We also have  that (T3a) holds since both sides are $0$ again (since $(a,c)$ is James).  Finally, (T3b) reduces to the condition that ${a+i\choose i}=0$ for $tp^\gamma< i\leq b+tp^\gamma$, $1\leq t\leq c_\gamma$.   Now if $a_h+i_h\geq p$ for some $0\leq h<v$ then ${a+i\choose i}=0$ so we may assume that $i=p^vI$, for some $I\neq 0$.  We get 
$$p^vI\leq b+tp^\gamma<p^{\beta+1}-1+p^w=(p^w-1)+p^{\beta+1}.$$
This gives $v\leq \beta+1$ but $\beta<v$ (since $\lambda$ is James) and so $v=\beta+1$ and $I=1$. Hence, $\len_p(b+p^\gamma)=\beta+1=v$. But this is impossible since $(a,b+p^\gamma)$ is James.

\end{proof}

\begin{lemma}  If $l_2=\cdots=l_n$ then $\lambda$ is constrained.
\end{lemma}

\begin{proof}  We proceed by induction on $n$. Certainly the result holds for $n=2,3$, by Lemmas 5.8 and Lemma 7.7(i). Now assume $n\geq 4$ and the result holds for shorter partitions. Note that $\JI(\lambda_1,\ldots,\lambda_n)=v_{n-1}-l$, where $l=l_2=\cdots=l_n$ and so $y^\can(n-1,n)_{p^l}\neq 0$. 

\q  By subtracting a multiple of the canonical multi-sequence from an arbitrary coherent multi-sequence  we can obtain a coherent multi-sequence $(y(r,s)_i)$ in which  $(y(n-1,n)_i)=0$. By applying the inductive hypothesis to the partition $(\lambda_2,\ldots,\lambda_n)$ we obtain $(y(r,s)_i)=0$ for all $2\leq r<s\leq n$. In particular the sequences $(y(2,3)_i), (y(3,4)_i)$, $ \ldots  , (y(n-1,n)_i)$ are zero. But now applying Lemma 7.7(i) to the coherent triple $(y(1,2)_i), (y(2,3)_j), (y(1,3)_k))$ we obtain that 
$(y(1,2)_i)=0$. Hence the sequences  $(y(1,2)_i), (y(3,4)_i)$, $\ldots, (y(n-1,n)_i)$ are all zero and $(y(r,s)_i)$ is the zero multi-sequence. Hence every coherent multi-sequence is a multiple of the canonical multi-sequence and we are done.
\end{proof}

\q Note that for  $1\leq m\leq n$  we have a natural linear map \\
$E(\lambda)\to E(\lambda_1,\ldots,\lambda_m)$, which we call restriction, sending a coherent multi-sequence $(y(r,s)_i)$, $1\leq r<s\leq n$, $1\leq i\leq \lambda_s$, to the coherent multi-sequence $(y(r,s)_i)$, $1\leq r<s\leq m$, $1\leq i\leq \lambda_s$.

\begin{lemma} Suppose that $1\leq m\leq n-2$ and that $l_{m+1}>l_{m+2}$. Then restriction  $E(\lambda)\to E(\lambda_1,\ldots,\lambda_{m+1})$ is surjective.

\end{lemma}

\begin{proof} Let $(y(r,s)_i)$  ($1\leq r <s\leq m+1$) be a coherent multi-sequence for $(\lambda_1,\ldots,\lambda_{m+1})$. We define a multi-sequence $(\haty(r,s)_i)$, $1\leq r<s\leq n$, $1\leq i\leq \lambda_s$,  for $\lambda$ by 
$$\haty(r,s)_i=\begin{cases}y(r,s)_i, & \hbox{ if }  1\leq r<s\leq m+1,  1\leq i\leq \lambda_s;\cr
0, & \hbox{ otherwise.}
\end{cases}$$

\q We claim that $(\haty(r,s)_i)$ is coherent.  To do this we check that for all $1\leq r < s<t \leq n$, the triple $(\haty(r,s)_i), (\haty(s,t)_j), (\haty(r,t)_k)$ is coherent. 

\q If $1\leq r<s<t\leq m+1$ this is true since $(y(r,s)_i)$ is a coherent multi-sequence. Suppose now that $1\leq r<s\leq m+1$ and $m+1<t\leq n$. Then the sequences $(\haty(s,t)_j)$ and $(\haty(r,t)_k)$ are identically $0$. Moreover we have $l_s\geq l_{m+1}> l_{m+2}\geq l_t$ so that $l_s>l_t$ and the triple  $(\haty(r,s)_i), (\haty(s,t)_j), (\haty(r,t)_k)$ is coherent, by Lemma 7.7(ii) and (iii).  Finally, if $s> m+1$ then $(\haty(r,s)_i), (\haty(s,t)_j), (\haty(r,t)_k)$ is the zero triple (which is coherent).

\q Thus $(\haty(r,s)_i)$ is a coherent multi-sequence mapping to $(y(r,s)_i)$ and the restriction map is surjective.

\end{proof}

\begin{lemma} If $1\leq m\leq n-2$ and $l_{m+2}=\cdots=l_n$ then the kernel of restriction $E(\lambda)\to E(\lambda_1,\ldots,\lambda_{m+1})$ has dimension at most $1$.

\end{lemma}

\begin{proof}  In addition to the restriction map $E(\lambda)\to E(\lambda_1,\ldots,\lambda_{m+1})$ we also have the  natural map $\rho:E(\lambda)\to E(\lambda_{m+1},\ldots,\lambda_n)$, taking a coherent multi-sequence $(y(r,s)_i)$, $1\leq r<s\leq n$, $1\leq i\leq \lambda_s$ for $\lambda$  to the coherent multi-sequence $(\bary(r,s)_i)$, $1\leq r<s\leq n-m$, $1\leq i\leq \lambda_{m+s}$,  for $(\lambda_{m+1},\ldots,\lambda_n)$, where $\bary(r,s)_i=y(m+r,m+s)_i$, for $1\leq r<s\leq n-m$, $1\leq i\leq \lambda_{m+s}$. 

\q We obtain an injective map $E(\lambda)\to E(\lambda_1,\ldots,\lambda_{m+1})\oplus E(\lambda_{m+1},\ldots,\lambda_n)$. But  $\dim E(\lambda_{m+1},\ldots,\lambda_n)=1$, by Lemma 7.8   and the kernel of restriction $E(\lambda)\to E(\lambda_1,\ldots,\lambda_{m+1})$ embeds in $E(\lambda_{m+1},\ldots,\lambda_n)$ so we are done.

\end{proof}

\q The next result essentially determines the dimension of the extension spaces for James partitions.

\begin{proposition}   Suppose that $1\leq m\leq n-2$ and $l_{m+1}>l_{m+2}=\cdots=l_n$. Then we have
$$\dim E(\lambda)=\begin{cases}
\dim E(\lambda_1,\ldots,\lambda_{m+1}), & \hbox{ if } m+2 =n \hbox{ and } \lambda_{m+1}=p^{v_m}-1;\cr
\dim E(\lambda_1,\ldots,\lambda_{m+1})+1, & \hbox{ otherwise.}
\end{cases}$$

\end{proposition}

\begin{proof} First suppose that $m+2=n$ and $\lambda_{m+1}=p^{v_m}-1$. Let $(y(r,s)_i)$ be a coherent multi-sequence in the kernel of restriction $E(\lambda)\to E(\lambda_1,\ldots,\lambda_{m+1})$.  Thus the sequences $(y(1,2)_i), (y(2,3)_i), \ldots, (y(m,m+1)_i)$ are identically $0$.   We have the coherent triple $(y(m,m+1)_i), (y(m+1,m+2)_j),$ $(y(m,m+2)_k)$.   Moreover, $((y(m,m+1)_i)$ is identically $0$  and 
$l_{m+1}>l_{m+2}$ so that, by Lemma 7.7(ii),  $(y(m+1,m+2)_j)$ is $0$. Hence the sequences  $(y(1,2)_i), (y(2,3)_i),\ldots,$ \\$ (y(n-1,n)_i)$ are $0$ and so $(y(r,s)_i)$ is the zero multi-sequence. Thus the kernel  of the restriction  is $0$ in this case and by Lemma 7.9, we have $\dim E(\lambda)= \dim E(\lambda_1,\ldots,\lambda_{m+1})$.

\q We next consider the case $v_m>\len_p(\lambda_{m+1}+p^{l_{m+2}})$.  We have the canonical multi-sequence $(y^\can(r,s)_i)$ for $(\lambda_{m+1},\ldots,\lambda_n)$ and we make a multi-sequence $(y(r,s)_i)$ for $\lambda$ by putting
$y(m+r,m+s)_i=y^\can(r,s)_i$, for $1\leq r<s\leq n-m$, $1\leq i\leq \lambda_s$, and all other terms $0$. We check that this is a coherent multi-sequence and hence the kernel of restriction $E(\lambda)\to E(\lambda_1,\ldots,\lambda_{m+1})$ is not zero. 

\q We have to check that each triple $(r,s,t)$ with $1\leq r<s<t\leq n$ gives a coherent triple of extension sequences $(y(r,s)_i), (y(s,t)_j),(y(r,t)_k)$.

\q  If $t\leq m$ then this is the zero triple (which is coherent).

\q  Suppose next that $1\leq r<s\leq m$, $m+1\leq t$. Again $(y(r,s)_i), (y(s,t)_j),$ $(y(r,t)_k)$ is the zero triple and so coherent. 

\q Suppose now that $1\leq r\leq m$, $m+1\leq s<t\leq n$.  Then the sequences $(y(r,s)_i)$ and $(y(r,t)_k)$ are identically $0$ and since we have
$$v_r\geq v_m> \len_p(\lambda_{m+1}+p^{l_{m+2}})\geq \len_p(\lambda_s+p^{l_t})$$
the triple $(y(r,s)_i), (y(s,t)_j),(y(r,t)_k)$ is coherent by Lemma 7.7(iii).  

\q Finally if $m+1\leq r<s<t\leq n$ then we may write $r=r_0+m$, $s=s_0+m$, $t=t_0+m$ and 
$$((y(r,s)_i), (y(s,t)_j),(y(r,t)_k))=((y^\can(r_0,s_0)_i), (y^\can(s_0,t_0)_j), (y^\can(r_0,t_0)_k))$$
which is certainly coherent. Hence $(y(r,s)_i)$  is a coherent multi-sequence in the kernel of restriction  $E(\lambda)\to E(\lambda_1,\ldots,\lambda_{m+1})$ and so by Lemmas 7.9 and 7.10 we have $\dim E(\lambda)=\dim E(\lambda_1,\ldots,\lambda_{m+1})+1$.

\q It remains to consider the case $m+2<n$ and $\lambda_{m+1}=p^{v_m}-1$.  Once again we use the canonical multi-sequence $(y^\can(r,s)_i)$ for $(\lambda_{m+1},\ldots,\lambda_n)$ to define the multi-sequence $(y(r,s)_i)$ for $\lambda$ with $y(m+r_0,m+s_0)_i=y^\can(r_0,s_0)_i$, for $1\leq r_0<s_0\leq n-m$, $1\leq i\leq \lambda_{m+s_0}$, and all other terms $0$.

\q We need to check that $(r,s,t)$, with $1\leq r<s<t\leq n$, gives a coherent triple $((y(r,s)_i), (y(s,t)_j),(y(r,t)_k))$. 

\q If $1\leq r<s\leq m$ then this is the zero triple.

\q Now suppose that $1\leq r\leq m$, $s=m+1<t\leq n$.  Then $(y(r,m+1)_i)$ and $(y(r,t)_k)$ are identically $0$. Now $l_{m+2}<l_{m+1}$ so that $\lambda_{m+2}\neq \lambda_{m+1}$ so that $\lambda_{m+2}<p^{v_m}-1$ and so $v_{m+2}<v_m$. Thus we have
$$v_{m+1}-l_t=v_m-l_t>v_{m+2}-l_t\geq v_{n-1}-l_n=\JI(\lambda_{m+1},\ldots,\lambda_n).$$ 
Writing $t=m+t_0$ we thus have that $(y^\can(1,t_0)_i)$ is the zero sequence, by Proposition 7.4  and hence $(y(m+1,t)_i)$ is the zero sequence and \\
$((y(r,m+1)_i), (y(m+1,t)_j), (y(r,t)_k))$ is the zero triple, and hence coherent. 

\q Now suppose that $r\leq m$, $m+1<s<t\leq n$. Again we have that the sequences $(y(r,s)_i))$ and $(y(r,t)_k)$ are zero.  We claim that $v_r>\len_p(\lambda_s+p^{l_t})$ and hence, by Lemma 7.7(iii), the triple $((y(r,s)_i), (y(s,t)_j),(y(r,t)_k))$ is coherent.  We have $v_r\geq v_m$ and $\len_p(\lambda_s+p^{l_t})\leq \len_p(\lambda_{m+2}+p^{l_t})$ so it is enough to check that $v_{m+1}>\len_p(\lambda_{m+2}+p^{l_t})$.   If this fails then we have $\lambda_{m+2}=p^{v_{m+1}}-1=p^{v_m}-1$ and this is not possible since $l_{m+1}>l_{m+2}$. 

\q Finally if $m+1\leq r<s<t\leq n$ then we may write $r=m+r_0$, $s=m+s_0$, $t=m+t_0$ and 
$$((y(r,s)_i), (y(s,t)_j),(y(r,t)_k))=((y^\can(r_0,s_0)_i), (y^\can(s_0,t_0)_j), (y^\can(r_0,t_0)_k))$$
which is certainly coherent. 

\q Thus the kernel of restriction $E(\lambda)\to E(\lambda_1,\ldots,\lambda_{m+1})$ is not zero and by 
Lemmas 7.8 and 7.9 we have $\dim E(\lambda)=\dim E(\lambda_1,\ldots,\lambda_{m+1})+1$.
\end{proof}

\begin{definition} Let $\lambda=(\lambda_1,\ldots,\lambda_n)$ be a James partition of length $n$. We define the  {\em segments} of $\lambda$ to be  the equivalence classes of $\{1,\ldots,n\}$ for the relation $r=s$ if and only if $l_r=l_s$. For $1\leq r,s\leq n$ we define $r$ and $s$ to be adjacent if they belong  to the same segment or if $1<r<n$, $s=r+1$ and $r+1$ is the only element in its segment  and $\lambda_r=p^{v_{r-1}}-1$. We define the  $p$-segments of $\lambda$ to be  the equivalent classes of $\{1,\ldots,n\}$ for the equivalence relation generated by adjacency.
\end{definition}

\q For a James partition $\lambda$ we define $e(\lambda)=\dim E(\lambda)$.

\begin{corollary} Let $\lambda=(\lambda_1,\ldots,\lambda_n)$ be a James partition of length $n\geq 2$.

(i) If $l_1=l_2$ then $e(\lambda)$ is   the number of $p$-segments of $\{1,\ldots,n\}$.

(ii) If $l_1>l_2$ then $e(\lambda)$ is one less than  the number of $p$-segments of $\{1,\ldots,n\}$.

\end{corollary}

\begin{proof}  We argue by induction on $n$.

(i) If $l_1=l_2=\cdots=l_n$ then the result is true by Lemma 7.8, in particular the result holds for $n=2$. Suppose now that there are at least $2$ segments and let $1\leq r\leq n$ be such that $l_r>l_{r+1}=\ldots=l_n$. Then $r\geq 2$. If $n-r\geq 2$ then we have $e(\lambda)=e(\lambda_1,\ldots,\lambda_r)+1$, by Proposition  7.11,  and inductively we have that this is one more than the  number of $p$-segments in $(\lambda_1,\ldots,\lambda_r)$, i.e., the number of $p$-segments in $\lambda$. So we may assume that $r=n-1$. If $\lambda_r\neq p^{v_{r-1}}-1$ then by Proposition 7.11,  we have $e(\lambda)=e(\lambda_1,\ldots,\lambda_{n-1})+1$, i.e.,  one more than the number of $p$-segments in $(\lambda_1,\ldots,\lambda_{n-1})$ and this is the number of $p$-segments in $\lambda$. If $\lambda_r= p^{v_{r-1}}-1$ then, by Proposition 7.11 we have $e(\lambda)=e(\lambda_1,\ldots,\lambda_{n-1})$, the number of $p$-segments in $(\lambda_1,\ldots,\lambda_{n-1})$, and this is also the number of $p$-segments in $\lambda$.

(ii) If $n=2$ we have $e(\lambda)=1$ and $\lambda$ has two $p$-segments so the result is correct. If $\lambda$ has two segments then we have $l_1> l_2=\ldots=l_n$ so $\lambda$ also has two $p$-segments and the result is correct, by Lemma 7.8. Now suppose that $\lambda$ has at least three $p$-segments and that $1\leq r<n$ is such that $l_r>l_{r+1}=\cdots=l_n$. Then we have $r\geq 2$.  If $n-r\geq 2$ then, by Proposition 7.11, we have $e(\lambda)=e(\lambda_1,\ldots,\lambda_r)+1$, which is the number of $p$-segments of $(\lambda_1,\ldots,\lambda_r)$ and this is one less than the number of $p$-segments of $\lambda$, as required.  So we may suppose that $r=n-1$.  Again if $\lambda_r\neq p^{v_{r-1}}-1$ then Proposition 7.11 gives $e(\lambda)=e(\lambda_1,\ldots,\lambda_{n-1})+1$, which is the number of $p$-segments of $(\lambda_1,\ldots,\lambda_{n-1})$ and hence one less than the number of $p$-segments of $\lambda$.  If $\lambda_r\neq p^{v_{r-1}}-1$ then Proposition 7.11 gives $e(\lambda)=e(\lambda_1,\ldots,\lambda_{n-1})$ which is one less than the number of $p$-segments of $(\lambda_1,\ldots,\lambda_{n-1})$ and hence also one less than the number of $p$-segments of $\lambda$.

\end{proof}

\bs\bs\bs\bs


\section{Three part partitions $(a,b,c)$ with $(a,b)$  James and $(b,c)$ not James}

\q We fix a three part partition $\lambda=(a,b,c)$.

\q We set  $\beta=\len_p(b)$, $\gamma=\len_p(c)$,  $\v=\val_p(a+1)$ and $\w=\val_p(b+1)$.  We have the standard coherent triple of extension sequences given $(x_i)$, $(y_j)$, $(z_k)$, given by $x_i^\st={a+i\choose i}$, $y_j^\st={b+j\choose j}$, $z_k^\st={a+k\choose k}$, for $1\leq i\leq b$, $1\leq j,k\leq c$.

\begin{lemma} Assume that $(a,b)$ is James  and $(b,c)$ is not James.  Then 
 $\lambda$ is non-split if and only if $(b,c)$ is  pointed and  $\v>\len_p(b+p^\gamma)$. In this case $E(\lambda)$ is spanned by the standard triple and the triple with  $(x_i)=0$, $(z_k)=0$ and $(y_j)$ given by 
$$y_j=\begin{cases} 1, & \hbox{ if } j=p^\gamma;\cr
0, & \hbox{ otherwise.}
\end{cases}
$$

\end{lemma}

 \begin{proof}    We claim  that in any coherent triple $(x_i),(y_j),(z_k)$ we have $(x_i)=0$ and $(z_k)=0$. Since $(a,b)$ and $(a,c)$ are James it is enough to show that $x_{p^\beta}=0$ and that $z_{p^\gamma}=0$. Suppose first that $w<\beta$. Now $\beta\geq\gamma \geq w$ (since $(b,c)$ is not James) so we may apply relation (T3b) in the case $j=p^w$, $i=p^\beta$. We get 
 $$0={{b+p^w-p^\beta}\choose{p^w}}x_{p^\beta}$$
 (since $x_i=0$ for $i<p^\beta$). By Lucas's Formula we get $(b_w+1)x_{p^\beta}=0$ so that $x_{p^\beta}=0$ and $(x_i)=0$.  Also by (T3a) for $i=j=p^\gamma$ we have that $z_p^{\gamma}=0$ and so $(z_k)=0$.
 
 \q We now suppose $\beta\leq w$. Since  $\beta\geq\gamma\geq w$ this implies $\beta=w$ and $\beta=\gamma$.  Since $(a,b)$ is James we have $a_\beta=p-1$ and (T1) with $i=k=p^\beta$ gives 
 $x_{p^\beta}=z_{p^\beta} .$
 But (T3a), with $i=j=p^\beta$ gives $0={b\choose p^\beta}x_{p^\beta}+z_{p^\beta}=0$, i.e., $b_\beta x_{p^\beta}+z_{p^\beta}=0$ and hence $(b_\beta+1)x_{p^\beta}=0$. Now $b_\beta=b_w\neq p-1$ so that $(b_\beta+1)\neq 0$ and so $x_{p^\beta}=0$ and $(x_i)$ and $(z_k)$ are the zero sequences. 
 
 \q  Assume that $(b,c)$ is split. Then we can subtract a multiple of the standard triple from an arbitrary coherent triple for $\lambda$ to obtain a coherent triple 
   $(x_i),(y_j),(z_k)$ in which $(y_j)=0$. But,  then by the claim, $(x_i)=0$ and $(z_k)=0$ and hence   $(x_i),(y_j),(z_k)$  is the zero triple. Hence $E(\lambda)$ consists of multiples of the standard triple and $\lambda$ is split.

  \q Now suppose that $(b,c)$ is pointed. After subtracting a multiple of the standard triple from any coherent triple we can obtain a coherent triple with $(x_i)=0$, $(z_k)=0$ and $y_j=0$ for $j\neq p^\gamma$.
  
  \q  We consider the triple of extension sequences given by $(x_i)=0$, $(z_k)=0$ and
$$y^\pt_j=\begin{cases} 1, & \hbox{ if } j=p^\gamma;\cr
0, & \hbox{ otherwise.}
\end{cases}.
$$

\q Then (T1),(T2),(T3a) hold trivially for this triple and so it is coherent if and only if (T3b) holds, i.e., if and only if ${a+i\choose i}y^\pt_j=0$ for all $1\leq j\leq c$, $j<i\leq b+j$, i.e., if and only if ${a+i\choose i}=0$ for all $p^\gamma<i\leq b+p^\gamma$, i.e., if and only if $a\geq b+p^\gamma$ and $(a,b+p^\gamma)$ is James, i.e., if and only if $v>\len_p(b+p^\gamma)$.  It follows that $E(\lambda)$ is spanned by the standard solution if $v=\len_p(b+p^\gamma)$ and by the standard solution and the point solution just constructed if $v>\len_p(b+p^\gamma)$. 
 
 \end{proof}

\bs\bs\bs\bs


\section{General results for $n$-rows}

\q Let $\lambda=(\lambda_1,\ldots,\lambda_n)$ be a partition of length $n\geq 4$.  Let $(y(r,s)_i)$, be a coherent extension multi-sequence for $\lambda$. We shall use the commuting relation, Section 4, (C), in the following special case.

\q For $1\leq  r< s< n$ with $s\neq r+1$ we have

$$
(C')  \   {\lambda_s+j\choose j} y(r,r+1)_i={{\lambda_r+i}\choose{i}}y(s,s+1)_j,\  1\leq  i\leq \lambda_{r+1}, 1\leq j\leq \lambda_{s+1}.
$$

We shall use the notation    $v_i=\val_p(\lambda_i+1)$, for  $1\leq i\leq n-1$ and $l_i=\len_p(\lambda_i)$, for  $2\leq i\leq n$. 

\begin{lemma} Suppose that $1\leq r<n$ and  $(\lambda_r,\lambda_{r+1})$ is not James.  Let $(y(t,u)_i)$ be  a coherent multi-sequence for $\lambda$ and suppose that $(y(r,r+1)_i)$ is either zero or the point extension sequence. Then, for $1\leq s<n$ with  $s\neq r-1,r,r+1$  the extension sequence  $(y(s,s+1)_j)$  is zero.
\end{lemma}

\begin{proof}  If  $(\lambda_r,\lambda_{r+1})$ is pointed and $(y(r,r+1)_i)$ is the point sequence we have $v_r<l_{r+1}$ so in all cases we have $y(r,r+1)_{p^{v_r}}=0$. We  have the commuting relation
$${\lambda_s+j\choose j}y(r,r+1)_{p^{v_r}}={\lambda_r+p^{v_r}\choose p^{v_r}}y(s,s+1)_j$$
for $1\leq j\leq \lambda_{s+1}$. The left hand side is $0$ and $\lambda_r$ has $v_r$th component in the base $p$ expansion different from  $p-1$ so  Lucas's Formula ${\lambda_r+p^{v_r}\choose p^{v_r}}$ is non-zero and therefore  $(y(s,s+1)_j)=0$.

\end{proof}

We give now a short proof of Weber's main theorem, \cite{Web} Theorem 1.2.
 
\begin{lemma} Let $\lambda=(\lambda_1,\dots,\lambda_n)$ be a partition with $n\geq 5$. Assume that there exist two pairs of consecutive rows of $\lambda$, say  $(\lambda_r,\lambda_{r+1})$ and $(\lambda_s,\lambda_{s+1})$, with $r+2<s<n$ which are not James pairs. Then  $\lambda$ is split.

\end{lemma} 

\begin{proof}  After subtracting a  standard multi-sequence from an arbitrary coherent multi-sequence  we can obtain a multi-sequence $(y(t,u)_i)$) in which either $(y(r,r+1)_i)=0$ (if $(\lambda_r,\lambda_{r+1})$ is split) or $(y(r,r+1)_i)$ is a multiple of the point sequence (if $(\lambda_r,\lambda_{r+1})$ is pointed).   By the above Lemma we have $(y(t,t+1)_j)=0$ for $t\neq r-1,r,r+1$.   In particular we have $(y(s,s+1)_j)=0$. Applying the above Lemma again  with $s$ in place of $r$ we get $(y(t,t+1)_j)=0$ if $t\neq s-1,s,s+1$.   Thus we get $(y(t,t+1)_j)=0$ for all $1\leq t<n$ and hence $(y(t,u)_i)$ is the zero multi-sequence for $\lambda$. We have shown  that we can  obtain the zero multi-sequence by subtracting  from an arbitrary coherent multi-sequence a multiple of the standard multi-sequence. Hence $E(\lambda)$ consists of multiples of the standard multi-sequence and $\lambda$ is split.

\end{proof}

We formulate this result  in the following way. 

\begin{remark} \rm
 Let $\lambda=(\lambda_1,\dots,\lambda_n)$ be a partition which is not James. Let $r$ be minimal such that $(\lambda_r,\lambda_{r+1})$  is not  James. If $\lambda$ is  not split then either:

(i)  the partitions  $(\lambda_1,\dots,\lambda_r)$ and $(\lambda_{r+3},\dots, \lambda_n)$ (if $n\geq r+3$) are James;  or \\
(ii)   $(\lambda_1,\dots,\lambda_{n-1})$ is  James  but $(\lambda_{n-1},\lambda_n)$  is not.
\end{remark}

We show now that whether or not  $\lambda$ is split  depends heavily  on the triple $(\lambda_r,\lambda_{r+1},\lambda_{r+2})$ in  the first case and on the pair  $(\lambda_{n-1},\lambda_n)$ in  the second case. 

\begin{proposition}

Let $\lambda=(\lambda_1,\dots,\lambda_n)$ be a partition which is not James and $r\leq n-1$ minimal such that $(\lambda_r,\lambda_{r+1})$ is not a James pair.

(i) If $r\leq n-2$,  then the restriction map $E(\lambda)\to E(\lambda_r,\lambda_{r+1},\lambda_{r+2})$ is injective. 

\q Equivalently,  if $(y(t,u)_i)$ is a coherent multi-sequence in which the extension sequences $(y(r,r+1)_i)$ and $(y(r+1,r+2)_i)$ are zero then $(y(t,u)_i)$ is identically zero. In  particular if $(\lambda_{r},\lambda_{r+1},\lambda_{r+2})$ is split  then $\lambda$ is split.

(ii) If $r=n-1$, then the restriction map $E(\lambda)\to E(\lambda_{n-1},\lambda_n)$ is injective. 

\q Equivalently,  if $(y(t,u)_i)$ is a coherent multi-sequence in which the extension sequence $(y(n-1,n)_i)$ is zero then $(y(t,u)_i)$ is identically zero. In  particular if $(\lambda_{n-1},\lambda_n)$ is split then $\lambda$ is split.

(iii) If $r=n-1$, we also have that the restriction map \\ $E(\lambda)\to E(\lambda_{n-2},\lambda_{n-1},\lambda_n)$ is injective. 

\q Equivalently,  if $(y(t,u)_i)$ is a coherent multi-sequence in which the extension sequences $(y(n-2,n-1)_i)$ and $(y(n-1,n)_i)$ is zero then $(y(t,u)_i)$ is identically zero. In  particular if $(\lambda_{n-2},\lambda_{n-1},\lambda_n)$ is split  then $\lambda$ is split.

\end{proposition}

\begin{proof}

We give the details of the proof of (i). The proofs of (ii) and (iii) are similar. Let $(y(t,u)_i)$ be a coherent  multi-sequence for $\lambda$ in which $(y(r,r+1)_i)$ and $(y(r+1,r+2)_i)$ are zero. By Lemma 9.1 we have then that $(y(t,t+1)_i)=0$ for $t\neq r-1$. Moreover, by Lemma 8.1 we get that $(y(r-1,r)_i)=0$, since $(\lambda_{r-1},\lambda_r)$ is James. Hence we have $(y(t,t+1)_i)=0$ for all $1\leq t<n$ and  $(y(t,u)_i)$ is the zero multi-sequence. This proves everything.
\end{proof}

\q We finish this section  with the case where in which the   first pair of rows of $\lambda$ which is not a James pair is the last two rows.

\begin{proposition}   Suppose $n\geq 3$ and  that $\lambda=(\lambda_1,\dots,\lambda_n)$ is not a James partition, but $(\lambda_1,\ldots,\lambda_{n-1})$ is James.

\q The partition   $\lambda$ is non-split if and only if $(\lambda_{n-1},\lambda_n)$ is  pointed  and  $v_{n-2}>\len_p(\lambda_{n-1}+p^{l_n})$.  In this case 
$E(\lambda)$ is spanned by the standard multi-sequence and the point sequence $(y^\pt(r,s)_i)$,  given by 
$$y^\pt(r,s)_i=\begin{cases} 1, & \hbox{ if } (r,s)=(n-1,n), i=p^{l_n};\cr
0, & \hbox{ otherwise.}
\end{cases}
$$
\end{proposition}

\begin{proof} 

Let $\lambda$ be non-split. By Proposition 9.4(iii) and Lemma 8.1 we can assume that  $(\lambda_{n-1},\lambda_n)$ is pointed and $v_{n-2}> \len_p(\lambda_{n-1}+p^{l_n})$. It remains to prove that  the point multi-sequence $(y^\pt(r,s)_i)$, as described in the statement of the lemma, is coherent.

\q Notice that for $(s,t)$ with $s<t<n-1$ the pairs $(s,t)$ and $(n-1,n)$ are related via the commuting relation (C). Since for any such $(s,t)$ the partition $(\lambda_s,\lambda_t)$ is James and $(y(s,t)_i)=0$ we have that both sides in the commuting relation are zero and so this relation is satisfied.

\q Hence it remains to check the coherence conditions for the triples \\
$(\lambda_q,\lambda_{n-1},\lambda_{n-2})$ with $q\leq n-2$. For $q\leq n-2$ we have that \\
$v_q\geq v_{n-2}> \len_p(\lambda_{n-1}+p^{l_n})$. Therefore by Lemma 8.1 our multi-sequence satisfies the coherent conditions for all these triples.

\end{proof}

\q Therefore,  from now on,  we need only  consider a partition $\lambda$ which contains a non-split  triple  $(\lambda_r,\lambda_{r+1},\lambda_{r+2})$ such that $(\lambda_r,\lambda_{r+1})$ is not James but   $(\lambda_1,\dots,\lambda_r)$ and $(\lambda_{r+3},\dots, \lambda_n)$ are James. We shall take up our analysis in these terms in Section 12, after dealing with the remaining cases of $3$-part partitions in Sections 10 and 11.

\bs\bs\bs\bs


\section{Partitions $(a,b,c)$ with $(a,b)$ split.}

\renewcommand{\ell}{{\rm len}}

\renewcommand{\v}{v}

\renewcommand{\w}{w}

 \q We continue  our analysis of $E(\lambda)$ for a $3$ part partition $\lambda=(a,b,c)$.  We have already dealt with the case in which $\lambda$ is James, in Section $7$ and with the case where $\lambda$ is not James and $(a,b)$ is James in Section 8. So we assume from now on that $(a,b)$ is not James. In this section we consider the case in which $(a,b)$ is split.

 \q As usual we write $\beta=\ell_p(b)$ and $\gamma=\ell_p(c)$ for $p$-adic lengths of $b$ and $c$ and $\v=\val_p(a+1)$ and $\w=\val_p(b+1)$ for the $p$-adic evaluations of $a+1$ and $b+1$.

\q Since $(a,b)$ is not James we have $\beta\geq \v$.  Moreover since $(a,b)$ is split,  it is not pointed.

\begin{lemma} If  $(a,c)$ is James, i.e., $\gamma<v$,  then $\lambda=(a,b,c)$ is split.
\end{lemma}

\begin{proof}   Let $(x_i),(y_j),(z_k)$ be a coherent triple for $\lambda=(a,b,c)$. Since $(a,b)$ is split after subtracting a multiple of the standard triple we obtain a coherent triple, which we also denote $(x_i),(y_j),(z_k)$, in which $(x_i)=0$. By Remark 6.2 It is enough to show that $(y_j)=0$.

\q Since $\gamma<v$ we have that $c<p^v$ and so we may take $i=p^v$ in (T3b) to get 

$${a+p^v\choose p^v}y_{j}=0$$

and ${a+p^v\choose p^v}=a_v+1\neq 0$ so that $y_{j}=0$ for $j\leq c$. 

\q We have shown that from any coherent triple  we may subtract a multiple of the standard coherent triple to obtain  the zero triple,  i.e.,  all coherent triples are standard, $E(\lambda)$ is one dimensional,  and  $\lambda$ is split.

 

\end{proof}

\q We therefore assume  now on  that $\gamma\geq\v$, i.e., $(a,c)$ is not James.

\begin{lemma}

If  $(a+p^\v,b)$ is not James then $\lambda=(a,b,c)$ is split.

\end{lemma}

\begin{proof}  We must show that every coherent triple is standard. After subtracting  a multiple of the standard triple from an arbitrary coherent triple  we obtain a coherent triple $(x_i), (y_j),(z_k)$ with $(x_i)=0$ and it is enough to prove that this triple is zero.

\q Taking $i=k=p^t<p^v$ in relation (T1) we get 
$$0={{a+p^t+p^t}\choose{p^t}}z_{p^t}$$
and, since $a_t=p-1$ (as $t<v$) we get $(a+p^t+p^t)_t=1$ and hence, 
by Lucas's Formula,  ${a+p^t+p^t\choose p^t}=1$ and 
 $z_{p^t}=0$.  Now suppose $v\leq t\leq \gamma$. From (T1), we get 
$$0={a+i+p^t\choose p^t}z_{p^t}$$
for $1\leq i\leq b$. If $t=v$ then  we have  ${a+p^v+i\choose i}\neq 0$ for some $1\leq i\leq b$ (since $(a+p^v,b)$ is not James) and we get $z_{p^v}=0$.  If $v<t$ then for  $i=p^v$ we have 
${a+i+p^t\choose i}={a+p^v+p^t\choose p^v}=a_v+1\neq 0$, and again $z_{p^t}=0$. Hence $(z_k)=0$. 

\q For $1\leq t<v$, choosing $i=p^v$, $j=p^t$ in (T3b) we  get $y_{p^t}=0$. For $v\leq t\leq \gamma$ we  take $i=p^v$, $j=p^t$ in (T3a) to get $y_{p^t}=0$. Hence $(y_j)=0$ and we are done.

\end{proof}

\begin{remark}  \rm  (i) Therefore form now on we can assume that $\val_p(a+p^v+1)>\beta$ (i.e. $(a,b+p^v)$ is James), which means that $a_t=p-1$ for $0\leq t \leq \beta, t\neq v$ and $a_v=p-2$.

(ii) Recall that we are assuming that $\gamma\geq v$.

\end{remark}

These assumptions will be in force for the rest of this section.

\begin{lemma} Assume that $(a,b)$ is split. Under the assumptions of Remark 10.3, in any coherent triple $(x_i),(y_j),(z_k)$ for $\lambda=(a,b,c)$ we have: 

(i) $x_i=0$ for  $i\neq p^\v$; and 

(ii) $z_k=0$ for $k\neq p^\v$.

\end{lemma}

\begin{proof}     Suppose  $1\leq i\leq b$. If  ${a+i\choose i}\neq 0$ then we have $a_h+i_h\leq p-1$ for all $0\leq h\leq \beta$, by Remark 3.1.  But by  Remark 10.3 the only possibility is $i=p^v$.   This proves that $x_i^\st=0$, for $1\leq i\leq b$, $i\neq p^v$ and also proves  that $z_k^\st=0$, for $1\leq k\leq c$, $k\neq p^v$.  Moreover, since $(a,b)$ is split $(x_i)$ is a multiple of the standard sequence  and we have (i).

\q Since $(a,c)$ is not James, it is either split or pointed. If $(a,c)$ is split then $(z_k)$ is a multiple of the standard sequence and hence $z_k=0$ for $k\neq p^v$. If $(a,c)$ is pointed then  $(z_k)$ is a linear combination of the sequences  $(z_k^\st)$ and $(z_k^\pt)$.  Now   $z_k^\st=z_k^\pt=0$ for $k\neq p^v,p^\gamma$.  If $\gamma=v$ we are done. If $\gamma>v$ we take $i=j=p^\gamma$ in (T1) to get 
$${a+p^\gamma+p^\gamma\choose p^\gamma}z_{p^\gamma}=0.$$
Now $a_\gamma=p-1$ and we get ${a+p^\gamma+p^\gamma\choose p^\gamma}=1$ so that $z_{p^\gamma}=0$ and we are done.

\end{proof}

\begin{remark}  \rm Assume that $(a+p^v,b)$ is James. If $(x_i),(y_j),(z_k)$ is a triple of extension sequences for $\lambda$ and $x_i=0$ for $i\neq p^v$ and $z_k=0$ for  $k\neq p^v$ then relation (T1) is satisfied. In fact both sides of the equation (T1) are zero. If $x_i\neq 0$ then $i=p^v$ and in that case ${a+i+k\choose k}=0$ (since $(a+p^v,b)$ is James). This shows that the left hand side is $0$.  The argument also shows that the right hand side is $0$.

\end{remark}

 \begin{remark} \rm
 
 Assume that $(a,b)$ is split. Under the assumptions of Remark 10.3, we have by Lemma 10.4 that in any coherent triple $(x_i), (y_j), (z_k)$ for $\lambda=(a,b,c)$,  $x_i=0$ for $i\neq p^v$. Here we give a very simple expression of the relations (T3a) and (T3b) for these partitions.  If a term ${a+i\choose s}x_{i-s}\neq 0$ appearing in the right hand side is non-zero then we have $i-s=p^v$ so that ${a+i\choose s}={a+p^v+s\choose s}\neq 0$. Since $s\leq c$ and $(a+p^v,  c)$ is James the only possibility is $s=0$. Hence the equations take the following forms

$${a+i\choose i}y_j={b+j-i\choose j}x_i+{b+j-i\choose j-i}z_i\eqno{(*)}$$
for $1\leq i\leq j\leq c$, and 
$${a+i\choose i}y_j={b+j-i\choose j}x_i\eqno{(**)}$$
for $1\leq j\leq c$, $j<i\leq b+j$.

\q It will be easier sometimes to use these simplified forms of (T3a) and (T3b) in order to prove that certain partitions $\lambda$ are not split.

\end{remark}

\begin{lemma} If  $(a,b)$ is split and $(b,c)$ is James then $\lambda=(a,b,c)$ is non-split if and only if $\gamma=v$, $c_v=1$ and $\len_p(b+p^v)<\val_p(a+p^v+1)$. 

\q In this case $E(\lambda)$ is spanned by the standard triple and the  triple $(x_i),(y_j),(z_k)$ with $(x_i)=0$, $y_{p^v}=1$, $z_{p^v}=-1$ and $y_k=z_k=0$ for $k\neq p^v$.

\end{lemma}

\begin{proof}  Assume that $\lambda$ is non-split. Then there exists a non-standard coherent triple $(x_i),(y_j),(z_k)$  for $\lambda$ and, subtracting a multiple of the standard triple we  obtain a non-zero coherent triple, also denoted $(x_i),(y_j),(z_k)$, in which $(x_i)=0$.  Taking $i=j=p^v$ in (T3a) we obtain, via Remark 10.3, that   $(p-1)y_{p^v}=z_{p^v}$, i.e., $y_{p^v}+z_{p^v}=0$. If $\gamma>v$ then $y_{p^v}=0$ since $(b,c)$ is James and so $z_{p^v}=0$. Moreover, taking $i=p^\v$ and $j=p^\gamma$ in (T3a) we get that $y_{p^\gamma}=0$ and so $\lambda$ is split. Therefore we must have that $v=\gamma$. 

\q Now if $p^v< tp^v\leq c$ then, taking $i=p^v$, $j=tp^v$ in (T3a) we get $-y_{tp^v}=0$ since $(b,c)$ is James. But $(y_j)$ is an extension sequence for the James partition $(b,c)$ so this is only possible if $(y_j)=0$. But then $(x_i),(y_j),(z_k)$ is the zero triple and $\lambda$ is split. Hence (if $\lambda$ is non-split) we must have  $c_v=1$.  

\q Moreover, we have $y_k=z_k=0$ for $k\neq p^v$ and $y_{p^v}+z_{p^v}=0$. Hence $\lambda$ is non-split if and only if $E(\lambda)$ is spanned by the standard coherent triple and the triple  $(x_i),(y_j),(z_k)$ with $(x_i)=0$, $y_{p^v}=1$, $z_{p^v}=-1$ and $y_k=z_k=0$ for $k\neq p^v$. 

\q  It remains to determine when this triple  is coherent. The condition (T1) is satisfied by Remark 10.5. The condition (T2) is that ${a+k\choose k}=0$ for $p^v+k\leq c$ (using the fact that $(b,c)$ is James). This implies that $k<p^v$, since $c_v=c_\gamma=1$ and so (T2)  holds. 

\q We now consider (T3a).  For $i=j$ it is enough to check  (T3a) for  $i=j=p^v$ and we have already done  this. For $i<j$ the right hand side is always zero (since $(b,c)$ is James) and  it is enough to take $j=p^v$. But then  ${a+i\choose i}=0$ so that (T3a) holds.   

\q Now (T3b) holds if and only if we have
$${a+i\choose i}=0$$
for $p^v<i\leq p^v+b$.  We claim that this is exactly the condition for \\
$(a+p^v,b+p^v)$ to be James, i.e.,  $\len_p(b+p^v)<\val_p(a+p^v+1)$.   First assume that this condition holds.  Then we get $a_t=p-1$ for all $p^v<p^t\leq p^v+b$ and so $(a+p^v)_t=p-1$ for all $p^t\leq b+p^v$, i.e., $(a+p^v,b+p^v)$ is James.   We leave it to the reader to check the reverse and thus complete the proof.

\end{proof}

\q Therefore from now on we consider cases in which  $\gamma\geq\w$ (i.e., $(b,c)$ is not James). Recall that we also always assuming  that $\gamma\geq \v$.

\begin{remark} \rm Assume that  $\gamma\geq \w$ (and so $(b,c)$  is not James). Then  after subtracting  a multiple of a standard solution from an arbitrary coherent triple we  can get a coherent triple $(x_i),(y_j),(z_k)$ in which  $y_j=0$ for $j\neq p^\gamma$ (with $y_{p^\gamma}$  also $0$ if $(b,c)$ is split).

\end{remark}
 
\begin{lemma}

If $\gamma\geq \w>\v$ then $\lambda=(a,b,c)$ is split.

\end{lemma}

\begin{proof} By Remark 10.8, we can subtract a standard triple from an arbitrary coherent triple to obtain a coherent triple $(x_i),(y_j),(z_k)$ in which $y_j=0$ for $j\neq p^\gamma$. Also by Lemma 10.4,  we have  $x_i=0$ for $i\neq p^\v$ and $z_k=0$ for $k\neq p^\v$. Hence, it is enough to prove that $x_{p^\v}=z_{p^\v}=y_{p^\gamma}=0$.

\q First notice that $y_{p^\w}=0$. This is certainly true if $\gamma>w$.   If $\gamma=\w$ then $(b,c)$ is split and so we have again that $y_{p^\w}=0$.
 
\q Setting now $i=p^\v$ and $j=p^\w$ in relation (T3a) we have that, 

$$0={{b+p^\w-p^\v}\choose{p^\w}}x_{p^\v}+{{b+p^\w-p^\v}\choose{p^\w-p^\v}}z_{p^\v}.$$

Moreover, we have that ${{b+p^\w-p^\v}\choose{p^\w}}=(b_\w+1)$. Also, ${{b+p^\w-p^\v}\choose{p^\w-p^\v}}=0$ as $(b-p^\v+p^\w)_\v=p-2$ and $(p^\w-p^\v)_\v=p-1$. Hence we conclude that $(b_\w+1)x_{p^\v}=0$ and so $x_{p^\v}=0$.  Now taking  $i=j=p^\v$ in  (T3a)  we get that $z_{p^\v}=0$, since $y_{p^\v}=x_{p^\v}=0$.

\q  Finally, for $i=p^\v$ and $j=p^\gamma$ in (T3a) we get again that ${{a+p^\v}\choose{p^\v}}y_{p^\gamma}=0$, and since $a_\v=p-2$ we obtain $-y_{p^\gamma}=0$ and  we are done, i.e., we have shown that $E(\lambda)$ is spanned by the standard triple.

\end{proof}

Therefore we can assume from now on that $\gamma\geq \v\geq \w$.

\newpage

\begin{lemma}

If $\gamma>\v>\w$ then $\lambda=(a,b,c)$ is split.

\end{lemma}

\begin{proof}
Again after subtracting  a standard triple from an arbitrary coherent triple  we can obtain a coherent triple $(x_i),(y_j),(z_k)$ in which  $y_j=0$ for $j\neq p^\gamma$. Also by Lemma 10.4 we have that $x_i=0$ for $i\neq p^\v$ and $z_k=0$ for $k\neq p^\v$ and so  it is enough to prove that $x_{p^\v}=z_{p^\v}=y_{p^\gamma}=0$.

\q Since $\gamma>\v>\w$ we have that $p^\v+p^\w\leq c$. Applying relation  (T2) with  $j=p^\w$ and $k=p^\v$,  we have 

$$0={{b+p^\w}\choose{p^\w}}z_{p^\v}=(b_\w+1)z_{p^\v},$$  

and  so $z_{p^\v}=0$. 

\q Now taking  $i=p^\v$ and $j=p^\w$ in relation (T3b) we get

$$0={{b+p^\w-p^\v}\choose{p^\w}}x_{p^\v}=(b_\w+1)x_{p^\v}$$

and so $x_{p^\v}=0$.

\q Finally, taking  $i=p^\v$ and $j=p^\gamma$ in (T3a) we get ${{a+p^\v}\choose{p^\v}}y_{p^\gamma}=0$, and since $a_\v=p-2$,  we have that $-y_{p^\gamma}=0$. Hence $E(\lambda)$ is spanned by the standard triple.

\end{proof}

Therefore we are left with the following three cases $\gamma=\v=\w$, $\gamma=\v>\w$ and $\gamma>\v=\w$.

Note that if $\gamma=v=w$ then $(b,c)$ is split.

\begin{lemma} Suppose that $(a,b)$ is split and that $(b,c)$ is split. Then $\lambda=(a,b,c)$ is non-split if and only if  $(a+p^v,b)$ is James and :

(i) if $b_v=0$ then $\val_p(b-p^v+1)>\gamma$ (in particular $b-p^v\geq c$);  

(ii) if $b_v\neq 0$ then $\gamma=v=w$ and $c_\gamma=1$.

\q  If $\lambda$ is non-split then $E(\lambda)$ is spanned by the standard triple and the triple $(x_i),(y_j),(z_k)$ with $(y_j)=0$, $x_i=0$ for $i\neq p^v$, $z_k=0$ for $k\neq p^v$, $x_{p^v}=1$ and $z_{p^v}=-b_v$.

\end{lemma}

\begin{proof} Assume that $\lambda$ is non-split. Let $(x_i),(y_j),(z_k)$ be a non-standard coherent triple for $\lambda$.  By subtracting a multiple of the standard triple we obtain a non-zero coherent triple, also denoted $(x_i),(y_j),(z_k)$,  in which $(y_j)=0$ and by Lemma 10.4 $x_i=0$ for $i\neq p^v$ and $z_k=0$ for $k\neq p^v$.  We apply (T3a) in the form of (*) of Remark 10.6, with $i=j=p^v$ (which we may do since $\gamma\geq v$) and we get ${b\choose p^v}x_{p^v}+z_{p^v}=0$. Hence we must have that $x_{p^v}\neq 0$ and, after scaling, we may assume that $x_{p^v}=1$ and $z_{p^v}=-b_v$. Now we just need to check the conditions for this triple  to be  coherent.

\q We check (T3b) in the form of (**) of Remark 10.6. This is clearly satisfied unless $i=p^v$, in which case we have that (T3b) will be satisfied if and only if, for all $1\leq j<p^v$ we have ${b-p^v+j\choose j}=0$, i.e., $(b-p^v,p^v-1)$ is James (in particular $b\geq 2p^v-1$).  We now assume this. Notice that this gives directly that $w\geq v$ and since $v\geq w$ we have that $v=w$.

\q Condition (T1) is satisfied by Remark 10.5.

\q We consider the cases:  (i) $b_v=0$ and; (ii) $b_v\neq 0$.

\bs

Case (i).  Suppose $b_v=0$. The condition (T2) is satisfied trivially.  In (T3a) (in the form of (*) of Remark 10.6) we can assume that $i=p^v$ and the condition is that ${b+j-p^v\choose j}=0$ for all $p^v\leq j\leq c$. However, we also have that $(b-p^v,p^v-1)$ is James and ${b+p^v-p^v\choose p^v}=b_v=0$ so the condition is precisely that $(b-p^v,c)$ is James, i.e., $\val_p(b-p^v+1)>\gamma$ (in particular $b- p^v\geq c$).

\bs

Case (ii). Suppose $b_v\neq 0$. Then (T2) gives ${b+j\choose j}=0$ for $1\leq j\leq c-p^v$, i.e., $(b,c-p^v)$ is James.   Note that this implies $\len_p(c-p^v)<\len_p(c)$ (since $(b,c)$ is not James) and since $\gamma\geq v= w$ we get that $v=\gamma$ and $c_\gamma=1$. 
We now consider (T3a), in the form of (*) of Remark 10.6. Once again we only need to consider the case $i=p^v$  and the condition is then ${b+j-p^v\choose j}-{b+j-p^v\choose j-p^v}b_v=0$ for $p^v \leq j\leq c$.  The condition holds if $j=p^v$.  Writing $j=p^v+t$, with $1\leq t\leq c-p^v\leq p^v-1$ the condition is ${b+t\choose t+p^v}={b+t\choose t}b_v$, for $1\leq t\leq c-p^v$ or, since $(b,c-p^v)$ is James,   ${b+t\choose t+p^v}=0$, for $1\leq t\leq c-p^v$.
So let $1\leq t\leq c-p^v$. Hence $\len_p(t)\leq \len_p(c-p^v)<v$ (since $\gamma=v$ and $c_v=1$).  Now we have ${(b+t)_i\choose t_i}=0$, for some $0\leq i<v$, by Lucas's Formula and the fact that $(b,c-p^v)$ is James. But then, since $i<v$ we have $(t+p^v)_i=t_i$ and so ${b+t\choose t+p^v}=0$, by Lucas's Formula again.

\end{proof}

\q We  have dealt with the situation in which  $(b,c)$ is split  and we move on to the case in which  $(b,c)$ is  pointed.

\begin{lemma}

Assume that $(b,c)$ is pointed  and $\gamma=\v>\w$. Then $\lambda=(a,b,c)$ is non-split  if and only if $\len_p(b+p^v)<\val_p(a+p^v+1)$.

 In that  case $E(\lambda)$ is spanned by the standard triple and the triple $(x_i),(y_j),(z_k)$ with $(x_i)=0$, $y_j=z_k=0$ for $j,k\leq c$ and $j,k\neq p^\v$, and $y_{p^\v}=1$, $z_{p^\v}=-1$. 

\end{lemma}

\begin{proof} Assume that $\lambda$ is not split. Since $(b,c)$ is pointed we may write $c=\chat +p^\gamma$, with $\chat<p^w<p^\gamma$.  After subtracting a standard triple from an arbitrary coherent triple and using Lemma 10.4 we can obtain a coherent triple $(x_i),(y_j),(z_k)$ in which $x_i=0$, $y_j=0$ and $z_k=0$ for $i,j,k\neq p^v$. Moreover,  taking $i=p^v$, $j=p^w$, in the form (**) of (T3b) from Remark 10.6, we get
$$0={{b+p^\w-p^\v}\choose{p^\w}}x_{p^\v}=(b_\w+1)x_{p^\v}$$
and so $x_{p^\v}=0$. 

\q Taking $i=j=p^v$  in the form (*) of (T3a) from Remark 10.6, we get 
$${{a+p^\v}\choose{p^\v}}y_{p^\v}=z_{p^\v}$$
and, since $a_v=p-2$, we get $y_{p^v}+z_{p^v}=0$.  Hence $\lambda$ is non-split if and only if the triple $(x_i),(y_j),(z_k)$ with $(x_i)=0$, $y_j=z_k=0$ for $j,k\leq c$ and $j,k\neq p^\v$, and $y_{p^\v}=1$, $z_{p^\v}=-1$ is coherent.

\q Certainly condition (T1) is satisfied.  Note that in (T2) the left hand side is $0$ since we either have $j\neq p^v$ or $1\leq k\leq \chat<p^v$. Similarly the right hand side is $0$ so that (T2) is satisfied.

\q We now consider (T3a)  in the form (*) from Remark 10.6. For $i=j$ we need only to consider $i=j=p^v$, and this we have already done. So we assume that $i<j$. We may assume that $i=p^v$ or $j=p^v$.  For $i=p^v<j\leq c $  since $\gamma=v$ we have  $j=p^v+t$ with $1\leq t\leq \chat$. So the condition reads $0={b+j-p^v\choose j-p^v}={b+t\choose t}$ i.e., $(b,\chat)$ is James, but this is true by hypothesis, since $\chat<p^w$. For $1\leq i<j=p^v$ both sides are $0$  and the condition holds. 

\q Finally, we consider (T3b) in the form (**) from Remark 10.6.  The condition is simply that ${a+i\choose i}=0$ for all $p^v<i\leq b+p^v$.  Now we already have ${a+p^v+p^h\choose p^h}=a_h+1=0$ for $1\leq h<v$ and ${a+p^v+p^v\choose p^v}=a_v+2=0$. Moreover for  $p^v<p^h\leq p^v+b$ we have ${a+p^h\choose p^h}={a+p^v+p^h\choose p^h}$, so the condition implies that ${a+p^v+t\choose t}=0$ for all $t=p^h\leq b+p^v$, i.e., $(a+p^v,b+p^v)$ is James.  Moreover it is easy to reverse the argument to show that if $(a+p^v,b+p^v)$ is James then (T3b) holds.

\end{proof}

\begin{lemma}

Assume that $(b,c)$ is pointed  and $\gamma>\v=\w$. Then $\lambda=(a,b,c)$ is non-split if and only if $(a+p^v,b)$ is James and either:

(i) $\val_p(b-p^v+1)>\gamma$; or 

(ii) $\val_p(b-p^v+1)=\gamma$ and $\len_p(b+p^\gamma)<\val_p(a+p^v+1)$.

If $\lambda$ is  non-split then  $E(\lambda)$ is spanned by the standard triple and the triple $(x_i),(y_j),(z_k)$ with $(z_k)=0$, $x_i=0$ for $i\neq p^v$, $x_{p^v}=1$, $y_j=0$ for $j\neq p^\gamma$ and $y_{p^v}=-b_\gamma$.

\end{lemma}

\begin{proof} Assume that $\lambda$ is not split. We write $c=\chat+p^\gamma$ with $\chat<p^v<p^\gamma$.

\q By Lemma 10.4, from an arbitrary coherent triple we may subtract a multiple of the standard triple to obtain a coherent triple $(x_i),(y_j),(z_k)$ with 
 $x_i=0$ for $i\neq p^\v$, $z_k=0$ for $k\neq p^\v$ and $y_j=0$ for $j\neq p^\gamma$.

\q We show first that $z_{p^\v}=0$. Indeed, since $\v<\gamma$ the relation (T2) applies for $j=k=p^\v$ and we get 

$$0={{b+p^\v}\choose{p^\v}}z_{p^\v}=(b_\v+1)z_{p^\v}$$

and since $\v=\w$ we have  $b_\v+1\neq 0$ and so $z_{p^\v}=0$.

\q Further, taking $i=p^v$, $j=p^\gamma$ in relation (T3a), in the form (*) from Remark 10.6, we get 
$${a+p^v\choose p^v}y_{p^\gamma}={b+p^\gamma-p^v\choose p^\gamma}x_{p^v}\eqno{(\dagger)}.$$
Since  $a_v=p-2$ the left hand side is $-y_{p^\gamma}$. Thus if $x_{p^v}=0$ then $y_{p^\gamma}=0$ and $(x_i),(y_j),(z_k)$ is the zero triple.  Hence if $\lambda$ is non-split there exists a non-standard  triple with $x_{p^v}\neq 0$.

\q Now taking $i=p^v\leq j<p^\gamma$ in (T3a), in the form (*) from Remark 10.6,  we deduce that if $\lambda$ is non-split then ${b-p^v+j\choose j}=0$.  Moreover, if $1\leq j<i=p^v$ in (T3b),in the form (**) from Remark 10.6, we have $y_j=0$ and so ${b-p^v+j\choose j}=0$.  Thus, if $\lambda$ is non-split, then ${b-p^v+j\choose j}=0$ for all $j<p^\gamma$, i.e., $\val_p(b-p^v+1)\geq \gamma$.  We now assume that this is the case. 

\q But now ($\dagger$) gives $-y_{p^\gamma}=b_\gamma x_{p^v}$. Thus $\lambda$ is non-split if and only if the triple
 $(x_i),(y_j),(z_k)$, with $(z_k)=0$, $x_i=0$ for $i\neq p^v$, $x_{p^v}=1$, $y_j=0$ for $j\neq p^\gamma$ and $y_{p^\gamma}=-b_\gamma$, is coherent, and in this case $E(\lambda)$ is spanned by the standard triple and  this triple.  It remains to check when this triple is coherent.
 
 \q Certainly (T1) is satisfied.  Condition (T2) is satisfied, since either $y_j=0$ or $j=p^\gamma$ and then ${a+k\choose k}=0$ as $1\leq k\leq \chat<p^v$.

 \q In equation (T3a), in the form (*) from Remark 10.6,  the left hand side is $0$ unless $i=p^v$ and $ j=p^\gamma$ and we have already checked this case. Thus we require that in all   other cases  the right hand side is $0$. As usual, it is enough to check the case in which $i=p^v$ and the condition is then (if $\lambda$ is non-split)  ${b+j-p^v\choose j}=0$ for all $p^v\leq j\leq c$, $j\neq p^\gamma$.  We have ${b-p^v+j\choose j}=0$ for $1\leq j<p^\gamma$, since we are assuming $\val_p(b-p^v+1)\geq \gamma$.  Hence the condition is that for all $1\leq t\leq \chat$ we have ${b-p^v+p^\gamma+t\choose p^\gamma +t}=0$. However, we may write $b=p^v-1+p^\gamma B$, for some $B\geq 1$, and so ${b-p^v+p^\gamma+t\choose p^\gamma +t}={t-1+p^\gamma B+p^\gamma\choose t+p^\gamma}$, which from Lucas's Formula  is   ${t-1 \choose t}(B+1)$ and therefore $0$. Hence (T3a) is satisfied.
 
 \q Finally we  consider (T3b) in the form (**) from Remark 10.6. If the right hand side is non-zero then $i=p^v$ and in this case the right hand side is ${b+j-p^v\choose j}$, with $j<p^v$. But $\val_p(b-p^v+1)\geq \gamma$ so that ${b+j-p^v\choose j}=0$. Hence the right hand side is always zero. If $\val_p(b-p^v+1)>\gamma$ then $b_\gamma=0$ and so $y_{p^\gamma}=0$. Thus in this case the left hand side is also $0$ and the coherence conditions are satisfied. This proves part (i).
 
\q  So we can assume that $\val_p(b-p^v+1)=\gamma$. As the right hand side in (T3b) is always $0$ the condition to be checked is ${a+i\choose i}y_j=0$ for $1\leq j<i\leq b+j$, i.e., ${a+i\choose i}=0$ for $p^\gamma <i\leq b+ p^\gamma$.   If ${a+i\choose i}\neq 0$ then we have $a_h+i_h\leq p-1$ for all $h\geq 0$, by Remark 3.1. Hence by Remark 10.3 we have $i=i_vp^v+p^{\beta+1}I$, for some $I$ and $i_v=0$ or $1$.   Since $p^\gamma<i$ we must have $I\neq 0$.  But also $i\leq b+p^\gamma$ gives $p^{\beta+1}I\leq p^{\beta+1}-1+p^\gamma<2p^{\beta+1}$ and so  $I=1$.  Hence (T3b) is not satisfied if and only if $p^{\beta+1}\leq b+p^\gamma$ and ${a+p^{\beta+1}\choose p^{\beta+1}}\neq 0$, i.e. $\len_p(b+p^\gamma)=\val_p(a+p^v+1)$. So we must have that $\len_p(b+p^\gamma)<\val_p(a+p^v+1)$ and this proves part (ii).

\end{proof}

\q We summarise  our findings from this section in the following result.

\begin{proposition}
Let $(a,b,c)$ be a partition with $(a,b)$  split. Then $\lambda=(a,b,c)$ is non split if and only if $(a+p^v,b)$ is James  and one of the following holds.

(i) We have $\gamma\geq \v=\w$ and $\val_p(b-p^v+1)>\gamma$. In this case $E(\lambda)$ is spanned by the standard triple and the coherent triple $(x_i),(y_j),(z_k)$ with $x_{p^v}=1, x_i=0$ for $i\neq p^v$ and $(y_j)=0$, $(z_k)=0$.

(ii) We have  $\gamma=\v=\w$,  $b_\v\neq 0$, and $c_v=1$. In this case $E(\lambda)$ is spanned by the standard triple and the coherent triple $(x_i),(y_j),(z_k)$ with $x_{p^v}=1, x_i=0$ for $i\neq p^v$, with $(y_j)=0$ and $z_{p^\v}=-b_v, z_k=0$ for $k\neq p^v$.

(iii) We have $\gamma>\v=\w$,  $\val_p(b-p^v+1)=\gamma$, $c=\chat +p^\gamma$ with $\chat<p^\v$ and $\ell_p(b+p^\gamma)<\val_p(a+p^\v+1)$. In this case $E(\lambda)$ is spanned by the standard triple and the coherent triple $(x_i),(y_j),(z_k)$ with $x_{p^v}=1, x_i=0$ for $i\neq p^v$, with $y_{p^v}=-b_\gamma, y_j=0$ for $j\neq p^\gamma$ and with $(z_k)=0$.

(iv) We have $\gamma=\v< \w$,  $c_\gamma=1$ and $\ell_p(b+p^\v)<\val_p(a+p^\v+1)$. In this case $E(\lambda)$ is spanned by the standard triple and the coherent triple $(x_i),(y_j),(z_k)$ with $(x_i)=0$, with $y_{p^v}=1, y_j=0$ for $j\neq p^v$ and with $z_{p^v}=-1,z_k=0$ for $k\neq p^v$.

(v) We have $\gamma=\v>\w$,  $c=\chat +p^\v$ with $\chat <p^\w$ and $\ell_p(b+p^\v)<\val_p(a+p^\v+1)$. In this case $E(\lambda)$ is spanned by the standard triple and the coherent triple $(x_i),(y_j),(z_k)$ with $(x_i)=0$, with $y_{p^v}=1, y_j=0$ for $j\neq p^v$ and with $z_{p^v}=-1,z_k=0$ for $k\neq p^v$.

\end{proposition}

\begin{proof}
Part (i) follows by Lemma 10.11 (i) and Lemma 10.13 (i). Part (ii) is Lemma 10.11 (ii) and part (iii) is Lemma 10.13 (ii). Now (iv) is Lemma 10.7 and (v) is Lemma 10.12.

\end{proof}

\bs\bs\bs\bs


\section{Partitions $(a,b,c)$ with $(a,b)$  pointed.}

\q We continue with our analysis of $E(\lambda)$ for a $3$-part partition $\lambda=(a,b,c)$. We deal here with the remaining case in which $(a,b)$ is pointed.

 \q As usual we write $\beta=\ell_p(b)$ and $\gamma=\ell_p(c)$ for the $p$-adic lengths of $b$ and $c$ and $v=\val_p(a+1)$ and $w=\val_p(b+1)$ for the $p$-adic evaluations of $a+1$ and $b+1$.

\q Since  $(a,b)$ is pointed and may write  $b=\hat{b}+p^\beta$ with $\hat{b}<p^\v<p^\beta$.  Notice  that here we  have that $w\leq v$ since $b_v=0$.

\begin{lemma}

If  $\gamma< w$, i.e., $(b,c)$ is James,  then    $\lambda=(a,b,c)$ is non-split and $E(\lambda)$ is spanned by the standard triple and the triple $(x_i),(y_j),(z_k)$ with $x_{p^\beta}=1$,  $x_i=0$ for $i\neq p^\beta$ and  $(y_j)=0$, $(z_k)=0$.

\end{lemma}

\begin{proof} Let $(x_i),(y_j),(z_k)$ be a coherent triple for $\lambda$. Now  $(a,b)$ is pointed and by subtracting a multiple of the standard triple, we obtain a coherent triple, which we also denote $(x_i),(y_j),(z_k)$, in which $x_i=0$ for $i\neq p^\beta$.

\q We claim that $(y_j)=0$ and $(z_k)=0$. Now $(b,c)$ is James and since $\gamma<w\leq v$ so is  $(a,c)$. Thus it is enough to prove that $y_{p^\gamma}=z_{p^\gamma}=0$. 

\q From (T1) with  $i=k=p^\gamma$  we get
$$0={{a+p^\gamma+p^\gamma}\choose{p^\gamma}}z_{p^\gamma}$$

and since $\gamma<v$ we have  $a_\gamma=p-1$  so ${{a+p^\gamma+p^\gamma}\choose{p^\gamma}}=1$ and  $z_{p^\gamma}=0$.

\q From  (T3b)  with  $i=p^v$ and $j=p^\gamma$ we get

$${{a+p^v}\choose{p^v}}y_{p^\gamma}=0$$

and so $(a_v+1)y_{p^\gamma}=0$. Since $a_v+1\neq 0$ we get  $y_{p^\gamma}=0$ as required.

\q Thus we will be done provided that we show that the triple  $(x_i),(y_j),(z_k)$ described in the statement of the Lemma is coherent.  For (T1) we require ${a+p^\beta +k\choose k}=0$ for $1\leq k\leq c$, i.e., $(a+p^\beta,c)$ is James. But we have \\
$\val_p(a+p^\beta+1)=v\geq w>\gamma$, so this is indeed the case.

\q In (T2) both sides are $0$.

\q We have $c<p^w< p^\beta$ so that  $x_i=0$ for $1\leq i\leq c$ and therefore (T3a) holds. 

\q A summand on the  right hand side of (T3b) can only be non-zero if  $i-s=p^\beta$ and moreover, since ${a+p^\beta+s\choose s}=0$, for $1\leq s\leq c$ (as $(a+p^\beta,c)$ is James) so we may restrict attention to the case  $s=0$.  Hence the condition (T3b) is that  ${b+j-p^\beta\choose j}=0$ for $j\leq  c$. 
But $\val_p(b+1-p^\beta)= \val_p(b+1)>\gamma$ so that $(b-p^\beta,c)$ is James and the condition is satisfied.

\end{proof}

\begin{lemma} If  $w\leq \gamma<v$ then  $\lambda=(a,b,c)$ is split.
\end{lemma}

\begin{proof} Subtracting a multiple of the standard triple from an arbitrary coherent triple for $\lambda$ we obtain a coherent triple $(x_i),(y_j),(z_k)$ with $x_i=0$ for $i\neq p^\beta$. Since $\gamma<v$  the partition $(a,c)$ is James so that, to prove $(z_k)=0$, it suffices to prove that $z_{p^\gamma}=0$. Now $\gamma<v< \beta$ so, we may apply (T1) with $i=k=p^\gamma$, which gives
$$0={{a+p^\gamma+p^\gamma}\choose{p^\gamma}}z_{p^\gamma}$$
and,  again since $\gamma<v$,  we have  $a_\gamma=p-1$ so that  ${{a+p^\gamma+p^\gamma}\choose{p^\gamma}}=1$ and  $z_{p^\gamma}=0$.

\q Taking $i=p^v$ in (T3b) we get ${a+p^v\choose p^v}y_j=0$ for $1\leq j\leq c$ (as $c<p^v$). Thus we have  $(a_v+1)y_j=0$ and therefore $y_j$, for $1\leq j\leq c$, i.e., $(y_j)=0$.

\q Finally, taking $i=p^\beta$, $j=p^w$ in (T3b) we get 
$$0={{b+p^w-p^\beta}\choose{p^w}}x_{p^\beta}=(b_w+1)x_{p^\beta}$$
and since  $b_w+1\neq0$ we get $x_{p^\beta}=0$ and therefore $(x_i)=0$. We have shown that all triples in  $E(\lambda)$ are standard,   i.e., $\lambda$ is split.
\end{proof}

\q We therefore  assume from now on that $\gamma\geq v$ and so $w\leq v\leq  \gamma\leq \beta$.

\begin{lemma} If $\val_p(a+p^v+1)<\beta$ then   $\lambda=(a,b,c)$ is split.
\end{lemma}

\begin{proof}     Since $(a,b)$ is pointed  we may subtract a standard triple from an arbitrary coherent triple to obtain a coherent triple $(x_i),(y_j),(z_k)$ in which $x_i=0$ for $i\neq p^\beta$.   We show first that $(z_k)=0$.  By Lemma 5.3  it is enough to show that $z_{p^t}=0$ for $t\leq \gamma$. For $t<v$ this follows directly from  relation (T1)  with  $i=k=p^t$  since $a_t=p-1$.
   Taking  $k=p^v$ in relation (T1) we have that

$$0={{a+i+p^v}\choose{i}}z_{p^v}$$

for  $1\leq i<p^\beta$. Since $\val_p(a+p^v+1)<\beta$ the is some $1\leq i < p^\beta$ such that 
 ${{a+i+p^v}\choose{i}}\neq0$ and hence   $z_{p^v}=0$. Now suppose that $t>v$.  Taking  $k=p^t$ and $i=p^v$ in  (T1) we have 

$${{a+p^v+p^t}\choose{p^v}}z_{p^t}=0$$

and since ${{a+p^v+p^t}\choose{p^v}}=a_v+1\neq0$ we conclude that $z_{p^t}=0$ and hence $(z_k)=0$.
  
 \q We now show  that $(x_i)=0$. Setting $i=p^\beta$ and $k=p^v$ in  (T1) we have 

$${{a+p^\beta+p^v}\choose{p^v}}x_{p^\beta}=0$$

i.e., $(a_v+1)x_{p^\beta}=0$ so that    $x_{p^\beta}=0$ and hence $(x_i)=0$.

\q Finally,  setting $i=p^v$ in (T3a) and (T3b) we get that $(a_v+1)y_j=0$ for  $1\leq  j\leq c$ and hence $(y_j)=0$. 

\q We have shown that all triples in  $E(\lambda)$ are standard,   i.e., $\lambda$ is split.

\end{proof}

\begin{remark} \rm
{\rm{ (i) Therefore from now on we can assume that $\val_p(a+p^v+1)\geq \beta$, which means that $a_t=p-1$ for $0\leq t<\beta$, $t\neq v$ and $a_v=p-2$.}

(ii) Recall  that we  are also assuming   $w\leq v\leq \gamma\leq \beta$.

\q These assumptions will be in force for the rest of this section.}

\end{remark}

\begin{lemma}
Assume that $(a,b)$ is  pointed. Under the assumptions of Remark 11.4, in any coherent triple $(x_i), (y_j), (z_k)$ for $\lambda=(a,b,c)$  we have:

(i) $x_i=0$ for $i\neq p^\v,p^\beta$;

(ii) $z_k=0$ for  $k\neq p^\v,p^\beta$;

(iii) $z_k=0$ for  $k\neq p^\v,$ if $\gamma<\beta$;

\end{lemma}

\begin{proof}  Let $1\leq i\leq b$. If ${a+i\choose i}\neq 0$ then we have $a_h+i_h\leq p-1$ for all $0\leq h\leq \beta$. Hence we must have $i_v=0$ or $1$ and $i_h=0$ for $0\leq h<  \beta$, $h\neq v$. Moreover, $b_\beta=1$ so we must have $i=p^v,p^\beta$ or $i=p^v+p^\beta$. However, the last case is not possible since $b=\hatb+p^\beta<p^v+p^\beta$.
This proves (i) since $(x_i)$ is a linear combination of the standard sequence and the point sequence. It also proves $z_k^\st=0$ for $k\neq p^v,p^\beta$  since $c\leq b$. Hence also $z_k=0$ for $k\neq p^v,p^\beta$  if $(a,c)$ is split. 

\q Now $v\leq \gamma$ so $(a,c)$ is not James and hence is pointed if it is not split. In this case $v<\gamma$. If $\gamma=\beta$ then since $(z_k)$ is a linear combination of the standard extension sequence and the point sequence we also get 
 $z_k=0$ for $k\neq p^v,p^\beta$.  Now suppose $v<\gamma<\beta$.  Applying  (T1) with $i=j=p^\gamma$ we get $0={a+p^\gamma+p^\gamma\choose p^\gamma}z_{p^\gamma}$.  But $a_\gamma=p-1$ so this gives $z_{p^\gamma}=0$. This proves (ii) and (iii).

\end{proof}

We consider now the situation in which  $\beta=\gamma$.

\begin{lemma}
Assume that $\beta=\gamma$. Then $E(\lambda)$ is spanned by the standard triple and triples  $(x_i),(y_j),(z_k)$ in which  $x_i=0$ for $i\neq p^\v$, $y_j=0$ for $j\neq p^\beta$  and  $z_k=0$ for $k\neq p^\beta$.
\end{lemma}

\begin{proof}  Since $v<\gamma=\beta$ we have that $(b,c)$ is not James and hence is either split or pointed. Hence after subtracting a multiple of the standard triple from an arbitrary coherent triple for $\lambda$ we can  obtain a coherent triple $(x_i),(y_j),(z_k)$  in which  $y_j=0$ for $j\neq p^\gamma$. 
By Lemma 11.5  we have   $x_i=0$ for $i\neq p^\v,p^\beta$ and $z_k=0$ for $k\neq p^\v, p^\beta$. 

\q We now show that $z_{p^\v}=x_{p^\beta}=0$. Since $\v< \beta$ we have  $p^\v+p^\v\leq c$ and we may apply (T2) with  $j=k=p^\v$ to obtain 
  $${{b+p^\v}\choose{p^\v}}z_{p^\v}=0.$$
  Since $(a,b)$ is pointed we have 
 $b_\v=0$ so that ${{b+p^\v}\choose{p^\v}}=1$ and hence $z_{p^\v}=0$.
 
 \q Now taking $i=p^\beta$, $k=p^v$ in (T1) we obtain 
 $${{a+p^\beta+p^\v}\choose{p^\v}}x_{p^\beta}=0.$$
  Hence  $(a_\v+1)x_{p^\beta}=0$ and so $x_{p^\beta}=0$ and we are done.
  
  \end{proof}

 \begin{lemma} If $\beta=\gamma$ and $w<v$ then  $\lambda=(a,b,c)$ is split.
 \end{lemma}
 
 \begin{proof} Let $(x_i),(y_j),(z_k)$ be a coherent triple for $\lambda$.   By Lemma 11.6, we can subtract a standard triple to obtain a coherent triple, which we also denote $(x_i),(y_j),(z_k)$ in which  $x_i=0$ for $i\neq p^\v$, $y_j=0$ for $j\neq p^\beta$ and  $z_k=0$ for $k\neq p^\beta$. 
 
 \q Taking $i=p^\v$ and $j=p^w$ in relation (T3b) we get 
    $$0={{b+p^w-p^\v}\choose{p^w}}x_{p^\v}$$
   and since ${{b+p^w-p^\v}\choose{p^w}}=b_w+1\neq 0$ we get  $x_{p^\v}=0$.
   
      \q Now taking   $k=p^\beta$ and $i=p^\v$ in (T1) we have  
   
   $$0={{a+p^\v+p^\beta}\choose{p^\v}}z_{p^\beta}=(a_\v+1)z_{p^\beta}$$

 and since $a_\v+1\neq 0$, we get  $z_{p^\beta}=0$.
 
\q  Finally, taking  $i=p^\v$ and $j=p^\beta$ in   (T3a) we get  $(a_\v+1)y_{p^\beta}=0$, and so $y_{p^\beta}=0$. Hence $(x_i),(y_j),(z_k)$ is the zero triple and $E(\lambda)$ contains only standard triples, i.e.,  $\lambda$ is split.

\end{proof}


\begin{lemma}

Suppose  $\beta=\gamma$. If $\lambda=(a,b,c)$ is non-split  then either $\val_p(a+p^v+1)>\beta$ or $\val_p(a+p^v+p^\beta+1)>\beta$.

\end{lemma}

\begin{proof}  By Lemma 11.3 we have that if $(a,b,c)$ is non-split then we may write  $a=(p^{\beta}-1)-p^\v+p^{\beta}a'$. We shall assume that  $a_\beta\neq p-1,p-2$ and show that  $\lambda=(a,b,c)$ is split.  By Lemma 11.6 it is enough to prove if  $(x_i),(y_j),(z_k)$ is a coherent triple for $\lambda$ in which  $x_i=0$ for $i\neq p^\v$, $y_j=0$ for $j\neq p^\beta$  and  $z_k=0$ for $k\neq p^\beta$ then this triple is $0$. 

\q Taking $i=k=p^\beta$ is (T1) we get 
 $$0={{a+p^\beta+p^\beta}\choose{p^\beta}}z_{p^\beta}=(a_\beta+2)z_{p^\beta}$$

 and since $a_\beta\neq p-2$ we obtain  $z_{p^\beta}=0$.
 
 \q  Now by  (T1) again, this time with  $i=p^\v$ and $k=p^\beta$,  we have 

  $${{a+p^\v+p^\beta}\choose{p^\beta}}x_{p^\v}=0$$
  
  and so $(a_\beta+1)x_{p^\v}=0$. Since  $a_\beta\neq p-1$ we have $x_{p^\v}=0$.
  
  \q   Finally, taking  $i=p^\v$ and $j=p^\beta$ in (T3a) we get that $(a_\v+1)y_{p^\beta}=0$  so  that $y_{p^\beta}=0$ and we are done.

\end{proof}

\q We now describe  the non-split triples for the case $\beta=\gamma$.

\begin{lemma}

Suppose  $\beta=\gamma$. Then $\lambda=(a,b,c)$ is non-split triple if and only if $\v=w$,  so $b=p^\v-1+p^\beta$, $c=\chat+p^\beta$ with $\chat<p^\v$,  and  one of the following holds.

(i)  $\len_p(b+p^\beta)<\val_p(a+p^v+1)$. 

 \q In that case $E(\lambda)$ is spanned by the standard triple and the triple $(x_i),(y_j),(z_k)$, with $x_{p^v}=1$, $x_i=0$ for $i\neq p^v$, with $y_{p^\beta}=-1$, $y_j=0$ for $j\neq p^\beta$, and $(z_k)=0$.

(ii) $\len_p(b+p^\beta)<\val_p(a+p^v+p^\beta+1)$.  

\q In that  case $E(\lambda)$ is spanned by the standard triple and the triple $(x_i),(y_j),(z_k)$ with $x_{p^v}=1$, $x_i=0$ for $i\neq p^\beta$, with $y_{p^\beta}=-1$, $y_j=0$ for $j\neq p^\beta$ and $z_{p^\beta}=1$, $z_k=0$ for $k\neq p^\beta$.

\end{lemma}

\begin{proof} By Lemma 11.7 we have that if $\lambda$ is not split then $v=w$ and since $(a,b)$ is pointed we have $b=p^v-1+p^\beta$. By Lemma 11.6 we can subtract  a multiple of a standard triple from a coherent triple to obtain a triple $(x_i),(y_j),(z_k)$ in which  $x_i=0$ for $i\neq p^\v$, $y_j=0$ for $j\neq p^\beta$ and  $z_k=0$ for $k\neq p^\beta$. 

\q By Lemma 11.8 we have either either $\val_p(a+p^v+1)>\beta$ or\\
 $\val_p(a+p^v+p^\beta+1)>\beta$.
 
 \bs

 Case (i) :   $\val_p(a+p^v+1)>\beta$.

\q We note that equations (T3a) and (T3b) in this case take a simplified form. If a term ${a+i\choose s}x_{i-s}\neq 0$ appearing in the right hand side is non-zero then we have $i-s=p^v$ so that ${a+i\choose s}={a+p^v+s\choose s}\neq 0$. Since $s\leq c$ and $(a+p^v,  c)$ is James (as $\val_p(a+p^v+1)>\beta$) the only possibility is $s=0$. Hence the equations take the following forms

$${a+i\choose i}y_j={b+j-i\choose j}x_i+{b+j-i\choose j-i}z_i\eqno{(*)}$$
for $1\leq i\leq j\leq c$, and 
$${a+i\choose i}y_j={b+j-i\choose j}x_i\eqno{(**)}$$
for $1\leq j\leq c$, $j<i\leq b+j$.

\q We  have $a_\beta=p-1$ and so for  $i=k=p^\beta$ in relation (T1) we get 
$$0={{a+p^\beta+p^\beta}\choose{p^\beta}}z_{p^\beta}=z_{p^\beta}.$$

\q Now taking $i=p^\v$ and $j=p^\beta$  in (T3a) we obtain 
$$-y_{p^\beta}={{b+p^\beta-p^\v}\choose{p^\beta}}x_{p^\v}$$
and since  $b-p^\v=p^{\beta}-1$ we have that ${{b+p^\beta-p^\v}\choose{p^\beta}}=1$ and so we get that $-y_{p^\beta}=x_{p^\v}$. Thus our analysis of case (i) will be complete if we check the conditions for the triple   $(x_i),(y_j),(z_k)$,  with  $x_{p^v}=1$, $x_i=0$ for $i\neq p^v$, and  $y_{p^\beta}=-1$, $y_j=0$ for $j\neq p^\beta$, and  $(z_k)=0$ to be  coherent. 
 
\q Equation  (T1) says that ${a+p^v+k\choose k}=0$,  for $1\leq k\leq c$, and this is true since $v_p(a+p^v+1)> \beta=\gamma$ so that $(a+p^v,c)$ is James.

\q Equation (T2) says that ${a+k\choose k}=0$ for $k\geq 1$ and $p^\beta+k\leq c$. For such $k$ we have $p^\beta+k\leq b=\hatb +p^\beta$, hence $k\leq \hatb<p^v$ and therefore ${a+k\choose k}=0$.

\q Now (T3a), in the form (*), says 
$${a+i\choose i}y_j={b+j-i\choose j}x_i$$
for $1\leq i\leq j\leq c$. We may assume that $i=p^v$ or $j=p^\beta$. For $i=p^v$, writing $b=p^v-1+p^\beta$ this becomes
$$-y_j={p^\beta-1+j\choose j}$$
for $p^v\leq j\leq c$. Both sides are $0$ for $j<p^\beta$ and $1$ for $j=p^\beta$.  For $j>p^\beta$ we may write $j=t+p^\beta$, with $0<t<p^v$ and then 
$${p^\beta-1+j\choose j}={p^\beta-1+t+p^\beta\choose t+p^\beta}={t-1+p^\beta+p^\beta\choose t+p^\beta}$$
which is $0$. Hence (T3a) holds in case $i=p^v$. Now suppose $j=p^\beta$. The condition is then
$$-{a+i\choose i}={b+p^\beta-i\choose p^\beta}x_i$$
for $1\leq i\leq p^\beta$. We have already checked this for $i=p^v$ so the condition is that ${a+i\choose i}=0$ for $1\leq i\leq p^\beta$, $i\neq p^v$. This holds for $i\neq p^\beta$ by applying Lemma 11.5 (i) to the standard triple.   For $i=p^\beta$ the condition is that ${a+p^\beta\choose p^\beta}=0$ and this is true as $a_\beta=p-1$. Hence (T3a) holds.

\q We now consider (T3b) in the form (**).  We may assume $i=p^v$ or $j=p^\beta$.  For $i=p^v$ the condition is 

$${a+p^v\choose p^v}y_j={b+j-p^v\choose j}$$
for $1\leq j<p^v$ or, since $y_j=0$, that ${b+j-p^v\choose j}=0$, for $1\leq j<p^v$. But ${b+j-p^v\choose j}={p^\beta-1+j\choose j}$ so this is certainly true. 

\q For $j=p^\beta$ the condition is 
$$-{a+i\choose i}={b+p^\beta-i\choose p^\beta}x_i$$
for $p^\beta<i\leq b+p^\beta$. For such $i$ we have $x_i=0$ so the condition is 
$${a+i\choose i}=0$$
for $p^\beta<i\leq b+p^\beta$. Now if this condition holds then, for $p^v<p^t\leq b+p^\beta$ we have ${a+p^v+p^t\choose p^t}={a+p^t\choose p^t}=0$. But we also have ${a+p^v+p^t\choose p^t}=0$ for $p^t\leq p^v$ so that $(a+p^v,b+p^\beta)$ is James.  Suppose conversely that $(a+p^v,b+p^\beta)$ is James and that ${a+i\choose i}\neq 0$ for some $p^v<i\leq b+p^\beta$.  Then we have $a_h+i_h\leq p-1$ for all $h\geq 0$ and so $i_v=0$ or $1$ and $i_h=0$ for $h\neq v$.  But then $i$ is $0$ or $p^v$, neither of which is possible since $p^v<i$. Hence (T3b) holds if and only if $(a+p^v,b+p^\beta)$ is James, i.e. $\len_p(b+p^\beta)<\val_p(a+p^v+1)$.

\bs

Case (ii): $\val_p(a+p^v+p^\beta+1)>\beta$.

\q In this case we have $a_\beta=p-2$. Applying (T1) with $i=p^v$, $k=p^\beta$ gives
$${a+p^v+p^\beta\choose p^\beta}x_{p^v}={a+p^v+p^\beta\choose p^v}z_{p^\beta}$$
i.e., $-x_{p^v}=-z_{p^\beta}$ so that $x_{p^v}=z_{p^\beta}$.  Taking $i=p^v$, $j=p^\beta$ in (T3a) we get
$${a+p^v\choose p^v}y_{p^\beta}={b+p^\beta-p^v\choose p^\beta}x_{p^v}$$
i.e., $-y_{p^\beta}={p^\beta+p^\beta-1\choose p^\beta}x_{p^v}$ and hence $-y_{p^\beta}=x_{p^v}$.  Thus our analysis of Case (ii) will be complete if we check  the conditions for the triple   $(x_i),(y_j),(z_k)$,  with  $x_{p^v}=1$, $x_i=0$ for $i\neq p^v$, and  $y_{p^\beta}=-1$, $y_j=0$ for $j\neq p^\beta$, and  $z_{p^\beta}=1$, $z_k=0$ for $k\neq p^\beta$,  to be  coherent. 

\q In checking (T1) we may assume that $i=p^v$ or $j=p^\beta$.  For $i=p^v$ the condition is that ${a+p^v+k\choose k}=0$ for all $1\leq k\leq c$, $k\neq p^\beta$ (we have already checked the case $i=p^v$, $k=p^\beta$).  However, if ${a+p^v+k\choose k}\neq 0$ then we have $(a+p^v)_h+k_h\leq p-1$ for all $h\geq 0$. This gives $k_h=0$ for all $h \neq \beta$ and $k=p^\beta$, which is not the case. It remains to check $k=p^\beta$, $1\leq i\leq c$, $i\neq p^v$.  The condition is 
$$0={a+i+p^\beta\choose i}.$$
If ${a+i+p^\beta\choose i}\neq0$ then $(a+p^\beta)_h+i_h\leq p-1$ for all $h$. This implies $i_h=0$ for $h\neq v$ and so $i=p^v$ which is not the case.

\q We now consider (T2).  Suppose $1\leq j,k\leq c$ and $j+k\leq c$.  If ${a+k\choose k}y_j\neq 0$ then $j=p^\beta$, but then $p^\beta+k\leq c\leq b$ gives $k<p^v$ so that ${a+k\choose k}=0$. Thus the left hand side of (T2) is $0$ in all cases and, similarly so is the right hand side.

\q We now consider (T3a). We notice again that if a term ${a+i\choose s}x_{i-s}$ appearing in the right hand side is non-zero then we have $i-s=p^v$ so that ${a+i\choose s}={a+p^v+s\choose s}\neq 0$. Since $\val_p(a+p^v+p^\beta+1)>\beta$ this can only happens if $s=0$ or $s=p^\beta$. However for the latter case this would force $i=p^v+p^\beta$ and as $i\leq j\leq c <p^v+p^\beta$ this is illegal. Hence the only possible case is when $s=0$ and so (T3a) takes again the form 

$${a+i\choose i}y_j={b+j-i\choose j}x_i+{b+j-i\choose j-i}z_i$$
for $1\leq i\leq j\leq c$, and 

\q Suppose $i=p^v$. Then the condition is 

$${a+p^v\choose p^v}y_j={b+j-p^v\choose j}$$
i.e., 
$$-y_j={p^\beta-1 +j\choose j}$$
for $p^v\leq j\leq c$.  This is clear for $j<p^\beta$ and for $j=p^\beta$. For $p^\beta<j\leq c$ we require ${p^\beta-1+j\choose j}=0$.  This is clear by writing $j=p^\beta+t$, with $1\leq t<p^v$, and appealing to Lucas's Formula.

\q Assume now that $i=p^\beta$. Then the condition is ${a+p^\beta\choose p^\beta}y_j={b+j-p^\beta\choose j-p^\beta}$, or $-y_j={p^v-1+j\choose j-p^\beta}$ for $p^\beta\leq j\leq c$. This is certainly true for $j=p^\beta$ and so we require ${p^v-1+j\choose 
j-p^\beta}=0$ for $p^\beta<j\leq c$. Again this becomes  clear on writing $j=p^\beta+t$, $1\leq t < p^v$, and using Lucas's Formula.

\q Now suppose $j=p^\beta$. The condition is 

$$-{a+i\choose i}={b+p^\beta -i\choose p^\beta}x_i+{b+p^\beta-i\choose p^\beta-i}z_i$$
for $1\leq i\leq p^\beta$.  If $i <p^v$ then both sides are $0$. If $i=p^v$ the equation is  $1={p^\beta-1+p^\beta\choose p^\beta}$, which is true.  For $p^v<i<p^\beta$ both sides are $0$. For $i=p^\beta$ both sides of the equation are $1$.

\q  We now consider (T3b).  

\q Let $j<p^\beta$. Then the left hand side in this relation is $0$. Now, we have that if a term ${a+i\choose s}x_{i-s}$ appearing in the right hand side is non-zero then we have $i-s=p^v$ so that ${a+i\choose s}={a+p^v+s\choose s}\neq 0$. Since $\val_p(a+p^v+p^\beta+1)>\beta$ this can only happen for $s=0$, $s=p^\beta$ or $s=p^v+p^\beta$. However, we also have that $s\leq j<p^\beta$ and so this leaves $s=0$ as the only option and so for $j<p^\beta$, (T3b) takes  the form 

$$0={b+j-i\choose j}x_i$$
for $1\leq j< p^\beta$ and $j<i\leq b+j$. Now for $i\neq p^v$ the right hand side is $0$ and so it remains to check this relation for $i=p^v$ where we get

$$0={b+j-p^v\choose j}={p^\beta-1+j\choose j}$$

for $1\leq j< p^\beta$, which is clearly true.

\q Assume now that $j>p^\beta$. Then the left hand side again is $0$ and if a term ${a+i\choose s}x_{i-s}$ appearing in the right hand side is non-zero then we have $i-s=p^v$ so that ${a+i\choose s}={a+p^v+s\choose s}\neq 0$. Since $\val_p(a+p^v+p^\beta+1)>\beta$ this can only happen for $s=0$, $s=p^\beta$ or $s=p^v+p^\beta$. Now since $s\leq c<p^v+p^\beta$ we have that the last case is impossible. Moreover, if $s=0$ that would force $i=p^v$ which is illegal since $p^\beta<j<i$. Hence this leaves $s=p^\beta$ as the only option and so for $j>p^\beta$, (T3b) takes  the form 

$$0={b+j-i\choose j-p^\beta}x_{i-p^\beta}$$
for $1<p^\beta< j\leq c$ and $j<i\leq b+j$. Now for $i\neq p^v+p^\beta$ the right hand side is $0$ and so it remains to check this relation for $i=p^v+p^\beta$ where we get

$$0={b+j-p^v-p^\beta\choose j-p^\beta}$$

for $p^\beta< j\leq c$. But such $j$ has the form $j=p^\beta+t$ with $0<t<p^v$ and so we have 

$$0={b+t-p^v\choose t}={p^\beta-1+t\choose t}$$

and since $0<t<p^v$ this is clearly true.

\q Hence, it remains to consider the case $j=p^\beta$. Making the same remarks as before for the right hand side of the relation we have that the relation takes the form

$$-{a+i\choose i}=x_{i-p^\beta}$$ 

for  $p^\beta<i\leq b+p^\beta$. This is true for $i=p^v+p^\beta$ and so the condition remaining   to be checked  is ${a+i\choose i}=0$ for $p^\beta<i\leq b+p^\beta$. One easily check (cf. the final step of Case (i)) that this is exactly the condition for $(a+p^v+p^\beta,b+p^\beta)$ to be James, i.e. $\len_p(b+p^\beta)<\val_p(a+p^v+p^\beta+1)$.
 
   \end{proof}


We continue  with the cases in which  $\beta>\gamma$.

\begin{lemma}

Assume that $v<\gamma<\beta$.  Then $E(\lambda)$ is spanned by standard triples and triples $(x_i),(y_j),(z_k)$ in which $x_i=0$ for $i\neq p^v$, $(y_j)=0$ and $(z_k)=0$.
\end{lemma}

\begin{proof}   Since $(a,b)$ is pointed we may write $b=\hatb + p^\beta$ with $\hatb<p^v<p^\beta$, in particular we have $b_v=0$ and $b_\gamma=0$. Since $\gamma>v$ we can subtract a multiple of a standard triple from an arbitrary coherent triple to obtain a coherent triple $(x_i),(y_j),(z_k)$ in which $y_j=0$ for $j\neq p^\gamma$.  By Lemma 11.5 we have $x_i=0$ for $i\neq p^v,p^\beta$ and $z_k=0$ for $k\neq p^v$. Since $\v< \gamma$ we have  $p^v+p^v\leq c$ and so (T2) with  $j=k=p^v$ gives 
 ${{b+p^v}\choose{p^v}}z_{p^v}=0$, i.e. , $z_{p^v}=0$ so that the sequence $(z_k)$ is $0$. Now taking $i=p^\beta$, $k=p^v$ in (T1) gives   ${{a+p^\beta+p^v}\choose{p^v}}x_{p^\beta}=0$ i.e., $(a_v+1)x_{p^\beta}=0$ so that $x_{p^\beta}=0$.  Finally, taking $i=p^v$ and $j=p^\gamma$ in (T3a) we get 
 $$-y_{p^\gamma}={{b+p^\gamma-p^v}\choose{p^\gamma}}x_{p^v}$$
 i.e., $-y_{p^\gamma}=b_\gamma x_{p^v}$  and since $b_\gamma=0$ we get $y_{p^\gamma}=0$ and therefore the sequence $(y_j)$ is $0$.

\end{proof}

 \begin{lemma} If $w<v<\gamma<\beta$ then   $\lambda=(a,b,c)$ is split.
 
 \end{lemma}
 
 \begin{proof} By Lemma 11.10 we can subtract a standard triple to obtain a coherent triple $(x_i),(y_j),(z_k)$ with $x_i=0$ for $i\neq p^v$, $(y_j)=0$ and $(z_k)=0$. 
 
 Taking $i=p^v$ and $j=p^w$ in (T3b)  gives   $0={{b+p^w-p^v}\choose{p^w}}x_{p^v}$ i.e., $0=x_{p^v}(b_w+1)$ and so $x_{p^v}=0$. Hence $(x_i),(y_j),(z_k)$ is the zero triple and $E(\lambda)$ contains only standard triples, i.e.,  $\lambda$ is split.

 \end{proof}

 \begin{lemma} If $w=v<\gamma<\beta$ then $\lambda=(a,b,c)$ is non-split and $E(\lambda)$ is spanned by standard triples  and the triple $(x_i),(y_j),(z_k)$ in which, $x_{p^v}=1, x_i=0$ for $i\neq p^v$, $(y_j)=0$ and $(z_k)=0$
 
 \end{lemma}
 
 \begin{proof}   Since $(a,b)$ is pointed and $w=v$ we have $b=p^v-1+p^\beta$.  This implies that $(b-p^v,c)$ is James and that $b_v=0$.
By Lemma 11.10 it is enough to prove that the given triple is coherent.  Condition (T1) reduces to the statement that ${a+p^v+k\choose k}=0$ for $1\leq k\leq c$ and this is true since $(a+p^v,c)$ is James.   Condition (T2) is satisfied since all terms are $0$.   

\q Condition (T3a) holds provided  that 
 $$0={b+j-p^v-s\choose j-s}{a+p^v+s\choose s}$$
 for $i=p^v+s\leq j\leq c$. If $s\neq 0$ then ${a+p^v+s\choose s}=0$ since $(a+p^v,c)$ is James and so (T3a) reduces to the condition that  
 $$0={b+j-p^v\choose j}$$
  i.e., ${p^\beta-1+j\choose j}=0$,   for $p^v\leq j\leq c$, which holds since $(p^\beta-1,c)$ is James.

 \q Finally, for (T3b) we require
  $$0={b+j-p^v-s\choose j-s}{a+p^v+s\choose s}$$
  for $1\leq j\leq c$, $j<p^v+s\leq b+j$, and with $s\leq j$.  Again this reduces to $0={b+j-p^v\choose j}$ i.e., ${p^\beta-1+j\choose j}=0$, and this holds since $(p^\beta-1,c)$ is James.

\end{proof}

We finish off with the last cases where $\gamma=\v\geq w$.

\begin{lemma} If $w=v=\gamma$ then $\lambda=(a,b,c)$ is non-split and $E(\lambda)$ is spanned by standard triple  and the triple $(x_i),(y_j),(z_k)$ in which $x_{p^v}=1$, $x_i=0$ for $i\neq p^v$, $(y_j)=0$ and $(z_k)=0$.

\end{lemma}

\begin{proof}  Since $\gamma=w$ the partition $(b,c)$ is split.  Subtracting a standard triple from an arbitrary coherent triple we can obtain a triple $(x_i),(y_j),(z_k)$ in which $(y_j)=0$.  Further, by Lemma 11.5, we have $x_i=0$ for $i\neq p^v,p^\beta$ and $z_k=0$ for $k\neq p^v$.

\q Taking $i=j=p^v$ in (T3a) we have ${b\choose p^v}x_{p^v}+z_{p^v}=0$ and $b_v=b_w=0$ (since $(a,b)$ is pointed) and hence $z_{p^v}=0$.  Now taking $i=p^\beta$, $k=p^v$ in (T1) we get ${a+p^\beta+p^v\choose p^v}x_{p^\beta}=0$, i.e. $(a_v+1)x_{p^\beta}=0$ and so $x_{p^\beta}=0$.

\q It remains to check that the triple $(x_i),(y_j),(z_k)$ described in the statement of the lemma is coherent. The condition (T1) reduces to the statement that ${a+p^v+k\choose k}=0$ for all $1\leq k\leq c$, and this is true since $(a+p^v,c)$ is James.  Condition (T2) is satisfied trivially.

\q We now consider (T3a). We check that $0=\sum_{s=0}^{i-1} {b+j-i\choose j-s}{a+i\choose s}x_{i-s}$, for $1\leq i\leq j\leq c$. Now $x_{i-s}$ is zero unless $i-s=p^v$ and then ${a+i\choose s}={a+p^v+s\choose s}$ is zero for $s\neq 0$ (as $(a+p^v,c)$ is James). Hence the condition to be checked is ${b+j-p^v\choose j}=0$, for $p^v\leq j\leq c$. But, since $v=\gamma$ we have $b=p^v-1+p^\beta$ and $(b-p^v,c)$ is James so that ${b+j-p^v\choose j}=0$, for $p^v\leq j\leq c$, as required.

\q Similarly, condition (T3b) reduces to the condition that ${b+j-p^v\choose j}=0$  for all $1\leq j<p^v$  and this is also guaranteed by the fact that $(b-p^v,c)$ is James.

\end{proof}

\begin{lemma} If  $w<v=\gamma$ and $(b,c)$ is split then $\lambda=(a,b,c)$ is split.
\end{lemma}

\begin{proof} Subtracting a standard triple from an arbitrary coherent triple for $\lambda$ we can obtain a coherent triple $(x_i),(y_j),(z_k)$ in which $(y_j)=0$. By Lemma 11.5 we have $x_i=0$ for $i\neq p^v, p^\beta$ and $z_k=0$ for $k\neq p^v$.  

\q Taking  $i=p^v$ and $j=p^w$ in (T3b)  we have 

$$0={{b+p^w-p^v}\choose{p^w}}x_{p^v}=(b_w+1)x_{p^\v}$$
and hence $x_{p^v}=0$. 

\q Taking  $i=p^\beta$ and $j=p^w$  in (T3b)  we have 

$$0={{b+p^w-p^\beta}\choose{p^w}}x_{p^\beta}=(b_w+1)x_{p^\beta}$$

and so $x_{p^\beta}=0$.

\q Taking $1\leq i=j\leq c$ in (T3a) we now obtain $z_i=0$, for all $1\leq i\leq c$. Hence $(x_i),(y_j),(z_k)$ is the $0$ triple and every coherent triple is standard, i.e., $\lambda$ is split.

\end{proof}

\begin{lemma} If  $w<v=\gamma$ and $(b,c)$ is pointed then $\lambda=(a,b,c)$ is non-split if and only if $\beta<\val_p(a+p^v+1)$.

In that case $E(\lambda)$ is spanned by the standard solutions and the triple\\
 $(x_i),(y_j),(z_k)$ with $(x_i)=0$, $y_{p^v}=1$, $y_j=0$ for $j\neq p^v$, $z_{p^v}=-1$ and $z_k=0$ for $k\neq p^v$.


\end{lemma}

\begin{proof} Assume that $\lambda$ is non-split.  By Lemma 11.5 we may subtract, from an arbitrary coherent triple, a standard triple to obtain a coherent  triple $(x_i)$, $(y_j)$, $(z_k)$ in which 
$x_i=0$ for $i\neq p^v,p^\beta$, $y_j=0$ for $j\neq p^\beta$ and $z_k=0$ for $k\neq p^v$. 

\q Taking  $i=p^v$ and $j=p^w$ in (T3b) we obtain 
$$0={{b+p^w-p^\v}\choose{p^w}}x_{p^\beta}=(b_w+1)x_{p^v}$$
and so $x_{p^v}=0$.

\q Again by  (T3b), with $i=p^\beta$ and $j=p^w$ we have 

$$0={b+p^w-p^\beta\choose{p^w}}x_{p^\beta}=(b_w+1)x_{p^\beta},$$

so that  $x_{p^\beta}=0$ and hence $(x_i)=0$.

\q Now by (T3a) with $i=j=p^v$ we get $-y_{p^v}=z_{p^v}$.  

\q By Remark 11.4 we may write $a=(p^\beta-1)-p^v+p^\beta a'$. We show that $a_\beta=p-1$.  If $a_\beta\neq p-1$ then for $i=p^\beta$ and $k=p^v$ in relation (T1) we get

$$0={a+p^v+p^\beta\choose{p^\beta}}z_{p^v}=(a_\beta+1)z_{p^v}$$

i.e. $z_{p^v}=0$ and so $\lambda$ is split. Hence we have that $a_\beta=p-1$ and so $(a+p^v,b)$ is James, i.e. $\beta<\val_p(a+p^v+1)$.

\q It remains to prove  that the triple specified in the statement of the Lemma is coherent.

\q For (T1) we require ${a+i+p^v\choose i}=0$, for $1\leq i\leq c$, and this is true since $(a+p^v,b)$ is James.

\q We now consider   (T2).  If ${a+k\choose k}y_j\neq 0$, with $1\leq j,k\leq c$, $j+k\leq c$, then we must have $j=p^v=p^\gamma$ and so $p^\gamma+k <p^w+p^\gamma$ and 
$k<p^w<p^v$ so that ${a+k\choose k}=0$. Hence the left hand side of (T2) is zero for all relevant values of $j$ and $k$.  Similarly, if ${b+j\choose j}z_k\neq 0$  with $1\leq j,k\leq c$, $j+k\leq c$, then $k=p^v=p^\gamma$ and so $j<p^w$ and ${b+j\choose j}=0$. Hence the right hand side of (T2) is also zero for  all relevant values of $j$ and $k$.

\q We now consider (T3a).  If $i=j=p^v$ then the equation reads ${a+p^v\choose p^v}=-1$, which is true. If $i=p^v$ and $p^v<j\leq c$ then the condition is ${b+j-p^v\choose j-p^v}=0$. Writing $t=j-p^v=j-p^\gamma$  we have $t<p^w$ and so ${b+j-p^v\choose j-p^v}={b+t\choose t}=0$. For $i\neq p^v$ then condition is ${a+i\choose i}=0$, and this is true.

\q Finally we consider the condition (T3b). The condition is  that ${a+i\choose i}=0$ for $p^v<i  \leq b+p^v$.   But if ${a+i\choose i}\neq 0$ then $a_h+i_n\leq p-1$ for all $h\geq 0$. But then $i_h=0$ for $h\neq v$ and $i_v=1$. Hence $i=p^v$. But $p^v<i$ so this is not possible.

\end{proof}
 
We gather together the main results of this section  in the following result.

\begin{proposition}
Let $\lambda=(a,b,c)$ be a partition with $(a,b)$ pointed, so that $b$ has the form $b=\hat{b}+p^\beta$ with $\hat{b}<p^\v<p^\beta$. Then $(a,b,c)$ is non split if and only if it has one of the following forms: 

(i) $\beta>\gamma\geq \v=w$ and $\val_p(a+p^v+1)\geq \beta$. In this case $E(\lambda)$ is spanned by the standard triple and the coherent triple $(x_i),(y_j),(z_k)$ with $x_{p^v}=1, x_i=0$ for $i\neq p^v$ and $(y_j)=0$, $(z_k)=0$.

(ii) $\beta=\gamma>\v=w$ and  $\len_p(b+p^\beta)<\val_p(a+p^v+1)$.  In this case $E(\lambda)$ is spanned by the standard triple and the coherent triple $(x_i),(y_j),(z_k)$ with $x_{p^v}=1, x_i=0$ for $i\neq p^v$, with $y_{p^\beta}=-1, y_j=0$ for $j\neq p^\beta$ and with $(z_k)=0$.

(iii) $\beta=\gamma>\v=w$ and $\len_p(b+p^\beta)<\val_p(a+p^v+p^\beta+1)$. In this case $E(\lambda)$ is spanned by the standard triple and the coherent triple $(x_i),(y_j),(z_k)$ with $x_{p^v}=1, x_i=0$ for $i\neq p^v$, with $y_{p^\beta}=-1, y_j=0$ for $j\neq p^\beta$ and with $z_{p^\beta}=1, z_k=0$ for $k\neq p^\beta$.

(iv) $v\geq w>\gamma$.  In this case $E(\lambda)$ is spanned by the standard triple and the coherent triple $(x_i),(y_j),(z_k)$ with $x_{p^\beta}=1, x_i=0$ for $i\neq p^\beta$ and $(y_j)=0$, $(z_k)=0$.

(v) $\gamma=\v >w$, $\val_p(a+p^v+1)> \beta$ and $c=\chat+p^\v$ with $\chat<p^w$. In this case $E(\lambda)$ is spanned by the standard triple and the coherent triple $(x_i),(y_j),(z_k)$ with $(x_i)=0$, with $y_{p^v}=-1, y_j=0$ for $j\neq p^v$ and with $z_{p^v}=-1,z_k=0$ for $k\neq p^v$.

\end{proposition}

\begin{proof}

Part (i) follows by Lemmas 11.3, 11.12 and 11.13. Part (ii) is Lemma 11.9 (i) and part (iii) is Lemma 11.9 (ii). Now (iv) is Lemma 11.1 and (v) is Lemma 11.15.

\end{proof}

\bs\bs\bs\bs


\section{Non-split partitions}

\subsection{A partition is either James or constrained.}

\q Recall that a partition $\lambda$ is constrained if 
$ \Ext^1_{B(N)} (S^d E, K_\lambda)$,  has dimension at most  $1$ where $d=|\lambda|$, $N$ is at least the number of parts of $\lambda$ and $E$ is the natural $G(N)$-module.   Equivalently $\lambda$ is constrained if $E(\lambda)$ is spanned by the standard multi-sequence and at most one other multi-sequence, so if $\lambda$ is not James,  the condition is that $E(\lambda)$ has dimension at most $2$. 

\q We have already dealt with James partitions, in Section 7, so the following,  which is a consequence of Propositions 10.14 and  11.16,  is very reassuring. 

\begin{theorem}

A partition is either James or constrained. 

\end{theorem}

\begin{proof} 

 We may assume that $\lambda$ is not James. Let $1\leq r<n$ be minimal such that $(\lambda_r,\lambda_{r+1})$ is not James. If $r=n-1$ then the result follows by Proposition 9.5. Hence we may assume that $r<n-1$ and then by Proposition 9.4(i) we have that the restriction map $E(\lambda)\to E(\lambda_r,\lambda_{r+1},\lambda_{r+2})$ is injective. Now by Propositions 10.14 and 11.16 we have that the partition $(\lambda_r,\lambda_{r+1},\lambda_{r+2})$ is constrained and we are done.

\end{proof}

\begin{remark} {\rm{From now on  $\lambda=(\lambda_1,\ldots,\lambda_n)$ denotes a partition of length $n\geq 4$  which is  not James  and  non-split. We let $1\leq r<n$ be minimal such that $(\lambda_r,\lambda_{r+1})$ is not James. Since we have done already the case in which  $r=n-1$,  in Proposition 9.5,  we also assume that $r<n-1$. Moreover, by Lemma 9.2 and Proposition 9.4(i) we know that $(\lambda_r,\lambda_{r+1},\lambda_{r+2})$ is non-split  and that (if $n\geq r+3$)  $(\lambda_{r+3},\dots,\lambda_n)$ is James. These properties will be our standard assumptions in this section and from now on we will assume that $\lambda$ has always such a form without further reference. In this section we give a complete  description of the non-split, non-James partitions}}.
\end{remark}

\q We note  that our results so far do not  give any information for the pair $(\lambda_{r+2},\lambda_{r+3})$ (for $n\geq r+3$). This pair may or may not be James.  We will investigate these two possibilities separately, starting with the case in which  $(\lambda_{r+2},\lambda_{r+3})$ is not a James partition.

\subsection{The case in which $(\lambda_{r+2},\lambda_{r+3})$ is not  James.}

\q We state  now the general result for the case in which  $(\lambda_{r+2},\lambda_{r+3})$ is not a James partition.  As usual we set  $v_i=\val_p(\lambda_i+1)$, $1\leq i\leq n-1$, and $l_i=\len_p(\lambda_i)$, for $2\leq i\leq n$. 

\q We assume throughout this subsection that $n> r+2$ and that $(\lambda_{r+2},\lambda_{r+3})$ is not James.

\newpage

\begin{proposition} Suppose that $\lambda$ is not James and $r<n-2$ is minimal such that $(\lambda_{r},\lambda_{r+1})$ is not a James partition. Assume further that $(\lambda_{r+2},\lambda_{r+3})$ is not James and that $(\lambda_{r+3},\dots,\lambda_n)$ is James. Then the partition  $\lambda$ is non-split if and only if the quadruple
$(\lambda_r,\lambda_{r+1},\lambda_{r+2}, \lambda_{r+3})$ satisfies the following conditions:

(i)  $\lambda_r=(p^{l_{r+1}+1}-1)-p^{v_r}+p^{l_{r+1}+1}\lambda_r'$, for some $\lambda_r'\geq 0$; and 

(ii) $\lambda_{r+1}=(p^{v_r+1}-1)-p^{v_r}+p^{v_r+1}\lambda_{r+1}'$, $(\lambda_{r+1})_{\v_r}\neq 0$ and  $\lambda_{r+1}'\geq 0$; and 

(iii) $\lambda_{r+2}=p^{v_r}-1+p^{v_r}$; and 

(iv)  $\lambda_{r+3}=\hat{\lambda}_{r+3}+p^{v_r}$, with $0\leq \hat{\lambda}_{r+3}<p^{v_r}$.

\q  If these conditions hold then $E(\lambda)$ is spanned by the standard multi-sequence and a non-zero  multi-sequence $(y(t,u)_i)$ satisfying 
$$y(r,r+2)_{p^{v_r}}=y(r+1,r+3)_{p^{v_r}}=-y(r+1,r+2)_{p^{v_r}}=-y(r,r+3)_{p^{v_r}}$$
and $y(s,t)_i=0$ for all the other choices  of $s,t$ and $i$.

\end{proposition}

\smallskip

\q Note that since we require in condition (ii) that $\lambda_{r+1}=(p^{v_r+1}-1)-p^{v_r}+p^{v_r+1}\lambda_{r+1}'$ with $(\lambda_{r+1})_{\v_r}\neq 0$, this case does not appear in characteristic $2$. In particular, in characteristic $2$ the assumption that $(\lambda_{r+2},\lambda_{r+3})$ is not James forces $\lambda$ to split. 

\q We give a proof of this proposition in a series of Lemmas.  Since \\
$(\lambda_r,\lambda_{r+1},\lambda_{r+2})$ is non-split and $(\lambda_r,\lambda_{r+1})$ is not James,  the three part partition must have one of the forms described in Proposition 10.14 or Proposition 11.16. Though there are many possibilities the remarkable fact is that only in a single case, namely that described in Proposition 10.14(ii), is it  possible for $\lambda$ to be non-split. 

\q Our method is to work through all these possibilities in turn, but leaving the case described in Proposition 10.14(ii) until last. The consideration of all these split cases occupies Lemmas 12.5 to 12.13. The final exceptional case is dealt with in Lemma 12.14.

\q We assume until further notice that $n\geq r+3$. Before we embark on  Lemmas 12.5  to  12.14, we prove the following useful result that we will use repeatedly.

\begin{lemma} The restriction map $E(\lambda)\to E(\lambda_{r+1},\lambda_{r+2},\lambda_{r+3})$ is injective, equivalently if $(y(t,u)_i)$ is a coherent multi-sequence in which the extension sequences $(y(r+1,r+2)_i)$ and $(y(r+2,r+3)_i)$ are zero then $(y(t,u)_i)$ is identically zero.
In  particular if $(\lambda_{r+1},\lambda_{r+2},\lambda_{r+3})$ is split  then $\lambda$ is split.
\end{lemma}

\begin{proof}  Let $(y(t,u)_i)$ be a coherent multi-sequence in which $(y(r+1,r+2)_i)$ and $(y(r+2,r+3)_i)$ are zero. By  the commuting relation (C) for the pairs $(r,r+1)$ and $(r+2,r+3)$, and   using the facts that $(y(r+2,r+3)_j)=0$ and $(\lambda_{r+2},\lambda_{r+3})$ is not James we deduce that 
 $(y(r,r+1)_i)=0$. Hence  we have that $(y(r,r+1)_i)=0$ and  $(y(r+1,r+2)_j)=0$ and so by Proposition 9.4 (i) $\lambda$ splits.
\end{proof}

\q The  Lemmas 12.5 to 12.8 which follow correspond to the cases in which $(\lambda_r,\lambda_{r+1},\lambda_{r+2})$ has the form described in Proposition 10.14 (i), (iii), (iv),(v). We therefore adopt as  standing assumptions that $(\lambda_r,\lambda_{r+1})$ is split and $l_{r+1}< \val_p(\lambda_r+p^{v_r}+1)$.

\q Our strategy is the same in all cases, namely we show that by subtracting a standard multi-sequence from an arbitrary coherent multi-sequence for $\lambda$ we can obtain a coherent multi-sequence $(y(t,u)_i)$ in which the extension sequences $(y(r+1,r+2)_i$ and $(y(r+2,r+3)_j)$ are zero. Thus, by Lemma 12.4, $(y(t,u)_i)$ is identically zero and hence any coherent multi-sequence is standard.

\q Our detailed analysis begins with the case in which $(\lambda_r,\lambda_{r+1},\lambda_{r+2})$ has the form described in Proposition 10.14(i).

\begin{lemma} Suppose  $l_{r+2}\geq v_r=v_{r+1}$, with $l_{r+1}<\val_p(\lambda_r+p^{v_r}+1)$ and $\val_p(\lambda_{r+1}-p^{v_r}+1)>l_{r+2}$.   Then $\lambda$ is split.

\end{lemma}

\begin{proof}  By Proposition 10.4(i), by subtracting a standard multi-sequence from an arbitrary multi-sequence for $\lambda$ we can obtain a multi-sequence $(y(t,u)_i)$ in which the sequences $(y(r+1,r+2)_j)=0$ and $(y(r,r+2)_k)=0$ and that  $y(r,r+1)_i=0$ for $i\neq p^{v_r}$. It remains  to prove that \\
$(y(r+2,r+3)_i)=0$. 

\q We first show  $(y(r+1,r+3)_k)=0$.  The pairs $(r,r+2)$ and $(r+1,r+3)$ are  related via the commuting relation (C), and so we have

$${{\lambda_r+k}\choose{k}}y(r+1,r+3)_i={{\lambda_{r+1}+i}\choose{i}}y(r,r+2)_k.$$

Since $(y(r,r+2)_k)=0$ we have 

$${{\lambda_r+k}\choose{k}}y(r+1,r+3)_i=0$$
for $1\leq k \leq \lambda_{r+2}$. Now as $l_{r+2}\geq v_r$ we can choose $k=p^{v_r}$ and we have that ${{\lambda_r+p^{v_r}}\choose{p^{v_r}}}\neq 0$. Therefore, we deduce that    $(y(r+1,r+3)_i)=0$.

\q We consider  $(\lambda_{r+1},\lambda_{r+2},\lambda_{r+3})$. We have the coherent triple $(x_i)$, $(y_j)$, $(z_k)$ with $x_i=y(r+1,r+2)_i$, $y_j=y(r+2,r+3)_j$, $z_k=y(r+1,r+3)_k$, for $1\leq i\leq \lambda_{r+2}$,   $1\leq j,k \leq \lambda_{r+3}$

\q Since 
$y((r+1,r+3)_k)=0$ and  $(y(r+1,r+2)_j)=0$ we have  $(x_i)=0$ and $(z_k)=0$.  Since   $v_r=v_{r+1}$, by setting $i=p^{v_r}$ in relations (T3a) and (T3b) we deduce  that $(y_j)=0$, i.e, $(y(r+2,r+3)_j)=0$ and we are done.

\end{proof}

\q We now treat the case in which $(\lambda_{r+1},\lambda_{r+2},\lambda_{r+3})$ has the form described in Proposition 10.14(iii).

\begin{lemma}  Suppose $l_{r+2}>v_r=v_{r+1}$, with $\len_p(\lambda_{r+1}+p^{l_{r+2}})<\val_p(\lambda_r+p^{v_r}+1)$,  $\val_p(\lambda_{r+1}-p^{v_r}+1)=l_{r+2}$ and $\lambda_{r+2}=\hat{\lambda}_{r+2}+p^{l_{r+2}}$, with $\hat{\lambda}_{r+2}<p^{v_r}$. Then $\lambda$ is split. 

\end{lemma}

\begin{proof}   By Proposition 10.4(iii), by subtracting a standard multi-sequence from an arbitrary multi-sequence for $\lambda$ we can obtain a multi-sequence $(y(t,u)_i)$ in which $(y(r,r+2)_k)=0$, $y(r,r+1)_i=0$ for $i\neq p^{v_r}$, $y(r+1,r+2)_j=0$ for $j\neq p^{l_{r+2}}$ and 

 $$y(r+1,r+2)_{p^{l_{r+2}}}+(\lambda_{r+1})_{l_{r+2}}y(r,r+1)_{p^{v_r}}=0. \eqno{(*)}$$ 
 
 Hence, by Lemma 12.4, it suffices to show that $y(r+1,r+2)_{p^{l_{r+2}}}=0$ and $(y(r+2,r+3)_j)=0$.  Notice though that if $(y(r+2,r+3)_j)=0$, then by the commuting relation for the pairs $(r,r+1)$ and $(r+2,r+3)$ (and the fact that $(\lambda_{r+2},\lambda_{r+3})$ is not James) we get that $(y(r,r+1)_i)=0$,  in particular $y(r,r+1)_{p^{v_r}}=0$. Then from (*)  we deduce directly that  $y(r+1,r+2)_{p^{l_{r+2}}}=0$. Therefore, it is enough to show that $(y(r+2,r+3)_j)=0$. We show first that $(y(r,r+3)_k)=0$. The  pairs $(r+1,r+2)$ and $(r,r+3)$ are  related via the commuting relation (C), and
 
$${{\lambda_{r+1}+j}\choose{j}}y(r,r+3)_k={{\lambda_{r}+k}\choose{k}}y(r+1,r+2)_j.$$

Now since $l_{r+2}>v_r=v_{r+1}$ and $y(r+1,r+2)_{p^{v_r}}=0$, setting $j=p^{v_r}$ we get directly that $y(r,r+3)_k=0$.

\q We consider now  $(\lambda_r,\lambda_{r+2},\lambda_{r+3})$.  We have the coherent triple $(x_i)$, $(y_j)$, $(z_k)$ with $x_i=y(r,r+2)_i$, $y_j=y(r+2,r+3)_j$, $z_k=y(r,r+3)_k$, for $1\leq i\leq \lambda_{r+2}$,   $1\leq j,k \leq \lambda_{r+3}$. Moreover, we have $(x_i)=0$, $(z_k)=0$. 
 Setting $i=p^{v_r}$ in relations (T3a) and (T3b) we obtain $(y_j)=0$, i.e.,  $y(r+2,r+3)_j=0$ for all $1 \leq j \leq \lambda_{r+3}$,  and we are done.  

\end{proof}

\q We now treat the case in which $(\lambda_{r+1},\lambda_{r+2},\lambda_{r+3})$ has the form described in Proposition 10.14(iv).

\begin{lemma}   Suppose $l_{r+2}=v_r<v_{r+1}$, with $\len_p(\lambda_{r+1}+p^{l_{r+2}})<\val_p(\lambda_r+p^{v_r}+1)$ and $\lambda_{r+2}=\hat{\lambda}_{r+2}+p^{l_{r+2}}$ with $\hat{\lambda}_{r+2}<p^{v_r}$. Then  $\lambda$ is split.
\end{lemma}

\begin{proof}  By Proposition 10.4(iv), by subtracting a standard multi-sequence from an arbitrary multi-sequence for $\lambda$ we can obtain a multi-sequence $(y(t,u)_i)$ in which $(y(r,r+1)_i)=0$ and $y(r+1,r+2)_j=y(r,r+2)_k=0$ for $j,k\neq p^{v_r}$, with

$$y(r+1,r+2)_{p^{v_r}}+y(r,r+2)_{p^{v_r}}=0.$$ 

Since  $(y(r,r+1)_i)=0$ (and $(\lambda_{r},\lambda_{r+1})$ is not James) from  the commuting relation (C) for the pairs $(r,r+1)$ and $(r+2,r+3)$ we get directly that 
 $(y(r+2,r+3)_j)=0$.  Hence it remains  to prove that $(y(r+1,r+2)_i)=0$. We consider  $(\lambda_{r+1},\lambda_{r+2},\lambda_{r+3})$. Since $l_{r+2}<v_{r+1}$, we have that $(\lambda_{r+1},\lambda_{r+2})$ is a James pair. Since $(\lambda_{r+2},\lambda_{r+3})$ is not a James partition we get by Lemma 8.1 that in any coherent triple  for  $(\lambda_{r+1},\lambda_{r+2},\lambda_{r+3})$ we have  $(y(r+1,r+2)_i)=0$,  and we are done.

\end{proof}

\q We now treat the case in which $(\lambda_{r+1},\lambda_{r+2},\lambda_{r+3})$ has the form described in Proposition 10.14(v).

\begin{lemma} Suppose $l_{r+2}=v_r>v_{r+1}$, with $\len_p(\lambda_{r+1}+p^{l_{r+2}})<\val_p(\lambda_r+p^{v_r}+1)$ and $\lambda_{r+2}=\hat{\lambda}_{r+2}+p^{l_{r+2}}$, with $\hat{\lambda}_{r+2}<p^{v_{r+1}}$. Then  $\lambda$ is split.

\end{lemma}

\begin{proof}  By Proposition 10.4(v), by subtracting a standard multi-sequence from an arbitrary multi-sequence for $\lambda$ we can obtain a multi-sequence $(y(t,u)_i)$ in which $(y(r,r+1)_i)=0$,  and  $y(r+1,r+2)_j=y(r,r+2)_k=0$ for $j,k\neq p^{v_r}$ and 

$$y(r+1,r+2)_{p^{v_r}}+y(r,r+2)_{p^{v_r}}=0.$$ 

Since  $(y(r,r+1)_i)=0$ (and $(\lambda_{r},\lambda_{r+1})$ is not James)  the commuting relation (C) for the pairs $(r,r+1)$ and $(r+2,r+3)$ gives directly that \\
$(y(r+2,r+3)_j)=0$.  Hence it remains  to prove that $y(r+1,r+2)_{p^{v_r}}=0$, or equivalently by the equation above, that $y(r,r+2)_{p^{v_r}}=0$.

\q We show first that $(y(r,r+3)_k)=0$. By the commuting relation (C) for the pairs $(r,r+3)$ and $(r+1,r+2)$  we have,
 
$${{\lambda_{r+1}+j}\choose{j}}y(r,r+3)_k={{\lambda_{r}+k}\choose{k}}y(r+1,r+2)_j.$$
Setting $j=p^{v_{r+1}}$ and using the fact that $y(r+1,r+2)_{p^{v_{r+1}}}=0$, as $v_r>v_{r+1}$, we get that $(y(r,r+3)_k)=0$.

\q We consider now  $(\lambda_r,\lambda_{r+2},\lambda_{r+3})$ and the corresponding coherent triple $(y(r,r+2)_i)$, $(y(r+2,r+3)_j)$, $(y(r,r+3)_k)$. Setting $i=p^{v_r}$ and $j=p^{v_{r+2}}$ in relation (T3b) we get

$$0={{\hat{\lambda}_{r+2}+p^{v_{r+2}}}\choose{p^{v_{r+2}}}} y(r,r+2)_{p^{v_r}}$$

where  $\lambda_{r+2}=\hat{\lambda}_{r+2}+p^{v_r}$ with $\hat{\lambda}_{r+2}<p^{v_{r+1}}$.  Now  $v_{r+1}<v_r=l_{r+2}$ and  we have  $v_{r+2}=\val_p(\hat{\lambda}_{r+2}+1)$ so that   ${{\hat{\lambda}_{r+2}+p^{v_{r+2}}}\choose{p^{v_{r+2}}}}\neq 0$ and therefore  \\
$y(r,r+2)_{p^{v_r}}=0$. So we are done.

\end{proof}

\q The  Lemmas 12.9 to 12.13  which follow correspond to the cases in which $(\lambda_r,\lambda_{r+1},\lambda_{r+2})$ has the form described in Proposition 11.16 (i)-(v). We therefore adopt as  standing assumption that $(\lambda_r,\lambda_{r+1})$ is pointed.

\q We use  the same strategy as in the above cases, namely we show that by subtracting a standard multi-sequence from an arbitrary coherent multi-sequence for $\lambda$ we can obtains a coherent multi-sequence $(y(t,u)_i)$ is which the extension sequences $(y(r+1,r+2)_i$ and $(y(r+2,r+3)_j)$ are zero and hence, by Lemma 12.4, $(y(t,u)_i)$ is identically zero and hence any coherent multi-sequence is standard.

\q The first case is that in which  $(\lambda_{r+1},\lambda_{r+2},\lambda_{r+3})$ has the form described in Proposition 11.16(i).

\begin{lemma}

Suppose that $l_{r+1}>l_{r+2}\geq v_r=v_{r+1}$ with $\val_p(\lambda_r+p^{v_r}+1)\geq l_{r+1}$  and  $\lambda_{r+1}=p^{v_r}-1+p^{l_{r+1}}$.  Then  $\lambda$ is split.

\end{lemma}

\begin{proof}

The proof for this case is completely analogous to the proof of  Lemma 12.5.

\end{proof}

\q In the next case   $(\lambda_{r+1},\lambda_{r+2},\lambda_{r+3})$ has the form described in Proposition 11.16(ii).

\begin{lemma}

Suppose that  $l_{r+1}=l_{r+2}> v_r=v_{r+1}$ with $\len_p(\lambda_{r+1}+p^{l_{r+2}})<\val_p(\lambda_r+p^{v_r}+1)$ and  $\lambda_{r+1}=p^{v_r}-1+p^{l_{r+1}}$. Then  $\lambda$ is split.

\end{lemma}

\begin{proof}

The proof is completely analogous to the proof of  Lemma 12.6.

\end{proof}

\q The next case, in which    $(\lambda_{r+1},\lambda_{r+2}.\lambda_{r+3})$ has the form described in Proposition 11.16(iii), requires some additional argument.

\begin{lemma} Suppose that  $l_{r+1}=l_{r+2}> v_r=v_{r+1}$ with $\len_p(\lambda_{r+1}+p^{l_{r+2}})<\val_p(\lambda_r+p^{v_r}+p^{l_{r+1}}+1)$ and $\lambda_{r+1}=p^{v_r}-1+p^{l_{r+1}}$. Then $\lambda$ is split.

\end{lemma}

\begin{proof} By Proposition 11.16(iii),  by subtracting a standard multi-sequence from an arbitrary multi-sequence for $\lambda$ we can obtain a multi-sequence $(y(t,u)_i)$ in which $y(r,r+1)_i=0$ for $i\neq p^{v_r}$,  and $y(r+1,r+2)_j=0$ for $ j\neq p^{l_{r+1}}$, and $y(r,r+2)_k=0$ for $k\neq p^{l_{r+1}}$ and that 

$$y(r,r+1)_{p^{v_r}}=-y(r+1,r+2)_{p^{l_{r+1}}}=y(r,r+2)_{p^{l_{r+1}}}.$$

As usual we have to prove that $(y(r+1,r+2)_i)=0$ and  $(y(r+2,r+3)_j)=0$.

\q We show first that $(y(r+2,r+3)_j)=0$. By the commuting relation (C) for the pairs $(r,r+1)$ and $(r+2,r+3)$, we get  that 

$${{\lambda_{r}+i}\choose{i}}y(r+2,r+3)_k={{\lambda_{r+2}+k}\choose{k}}y(r,r+1)_i.$$

Since $(\lambda_r)_{l_{r+1}}=p-2$ and $y(r,r+1)_{p^{l_{r+1}}}=0$, setting $i=p^{l_{r+1}}$ in the relation above we get that $y(r+2,r+3)_k=0$.

\q Now  by the same relation and the fact that  $(\lambda_{r+2},\lambda_{r+3})$ is not James we also get that $(y(r,r+1)_i)=0$. Since now  $y(r,r+1)_{p^{v_r}}=-y(r+1,r+2)_{p^{l_{r+1}}}$ and $y(r+1,r+2)_j=0$ for $j\neq p^{l_{r+1}}$ we get directly that  the extension sequence $(y(r+1,r+2)_j)$ is $0$ and we are done.

\end{proof}

\q We next consider the case in which     $(\lambda_{r+1},\lambda_{r+2},\lambda_{r+3})$ has the form described in Proposition 11.16(iv).

\begin{lemma}

Suppose that $l_{r+2}<v_{r+1}\leq v_r<l_{r+1}$ and $\lambda_{r+1}=\hat{\lambda}_{r+1}+p^{l_{r+1}}$ with $0\leq \hat{\lambda}_{r+1}<p^{v_r}$. Then  $\lambda$ is split.

\end{lemma}

\begin{proof}  By Proposition 11.16(iv), by subtracting a standard multi-sequence from an arbitrary multi-sequence for $\lambda$ we can obtain multi-sequence $(y(t,u)_i)$ in which $(y(r+1,r+2)_i)=0$ and $(y(r,r+2)_k)=0$ and that $y(r,r+1)_i=0$ for $i\neq p^{l_{r+1}}$. Hence it remains  to prove that $(y(r+2,r+3)_j)=0$. This is straightforward. Considering the commuting relation (C) for $(r,r+1)$ and $(r+2,r+3)$ we have 

$${{\lambda_{r}+i}\choose{i}}y(r+2,r+3)_j={{\lambda_{r+2}+j}\choose{j}}y(r,r+1)_i.$$

Now setting $i=p^{v_r}$, and since $y(r,r+1)_{p^{v_r}}=0$, we deduce that  the extension sequence 
$(y(r+2,r+3)_j)$ is zero and we are done.

\end{proof}

\q In the next case $(\lambda_{r+1},\lambda_{r+2},\lambda_{r+3})$ has the form described in Proposition 11.16(v).

\begin{lemma}

Suppose $l_{r+1}>l_{r+2}=v_r>v_{r+1}$ with $l_{r+1}< \val_p(\lambda_r+p^{v_r}+1)$, $\lambda_{r+1}=\hat{\lambda}_{r+1}+p^{l_{r+1}}$ with $0\leq \hat{\lambda}_{r+1}<p^{v_r}$, and $\lambda_{r+2}=\hat{\lambda}_{r+2}+p^{v_r}$ with  $0\leq \hat{\lambda}_{r+2}<p^{v_{r+1}}$. Then  $\lambda$ is split.

\end{lemma}

\begin{proof}

The proof for this case is entirely analogous to the   proof of Lemma 12.8.

\end{proof}

\q  Thus we are left with the case in which  the non-split triple $(\lambda_r,\lambda_{r+1},\lambda_{r+2})$ corresponds to the case of Proposition 10.14(ii). We show here that under certain assumptions this triple provides a non-split partition $\lambda$ and this finishes  the proof of Proposition 12.3. More precisely, we prove the following.

\begin{lemma} Suppose that   $l_{r+2}=v_r=v_{r+1}$, with \\
$l_{r+1}<\val_p(\lambda_r+p^{v_r}+1)$,  $(\lambda_{r+1})_{v_r}\neq 0$ and $\lambda_{r+2}=\hat{\lambda}_{r+2}+p^{v_r}$ with \\
$0\leq \hat{\lambda}_{r+2}<p^{v_r}$. Then $\lambda$ is non-split if and only if: 

(i) $\lambda_{r+1}=(p^{v_r+1}-1)-p^{v_r}+p^{v_r+1}\lambda_{r+1}'$, $(\lambda_{r+1})_{\v_r}\neq 0$ and  $\lambda_{r+1}'\geq 0$; and 

(ii) $\lambda_{r+2}=(p^{v_r}-1)+p^{v_r}$; and 

(ii) $\lambda_{r+3}=\hat{\lambda}_{r+3}+p^{v_r}$ with $0\leq \hat{\lambda}_{r+3}<p^{v_{r}}$.

\q In that case $E(\lambda)$ is spanned by the standard multi-sequence and a non-zero  multi-sequence $(y(t,u)_i)$ in which 
$$y(r,r+2)_{p^{v_r}}=y(r+1,r+3)_{p^{v_r}}= - y(r+1,r+2)_{p^{v_r}}= - y(r,r+3)_{p^{v_r}}.$$
and $y(t,u)_i=0$ for all  other choices  of  $t,u$ and $i$.
\end{lemma}

\begin{proof}

We first emphasise that since $v_r=v_{r+1}$ and $\lambda_{r+1}=p^{v_r}-1+p^{v_r}\lambda_{r+1}'$ with $(\lambda_{r+1})_{v_r}\neq 0$, the characteristic of $K$ is odd.  We show now that if $\lambda_{r+1}$ or $\lambda_{r+3}$ does  not  have the form described in (i) and (iii) then $(\lambda_{r+1},\lambda_{r+2},\lambda_{r+3})$ is split  and so by Lemma 12.4, we get that $\lambda$ is split.

\q First notice that since $l_{r+2}=v_{r+1}$ we have that $(\lambda_{r+1},\lambda_{r+2})$ is  split.  Now by Remark 10.3 (i) and (ii), or directly by Proposition 10.14, we deduce that if $\lambda$ is non-split then we may write  $\lambda_{r+1}=(p^{v_r+1}-1)-p^{v_r}+p^{v_r+1}\lambda_{r+1}'$ and $l_{r+3}\geq v_r$. Since now $\lambda_{r+2}=\hat{\lambda}_{r+2}+p^{v_r}$ we get that  $\lambda_{r+3}=\hat{\lambda}_{r+3}+p^{v_r}$ with $\hat{\lambda}_{r+3}<p^{v_r}$. So we have that if $\lambda$ is non-split then $\lambda_{r+1}$ and $\lambda_{r+3}$ must have the form of (i) and (iii) described in the statement.

\q We show now that if $\lambda_{r+2}$ is not as in (ii) of the statement then $\lambda$ is split. 

By the description of its parts so far, we have that if $(\lambda_{r+1},\lambda_{r+2},\lambda_{r+3})$ is a non-split triple then it corresponds to one of the cases  (ii) or (v) of Proposition 10.14, where $v_{r+2}=v_{r+1}$ or $v_{r+2}<v_{r+1}$ respectively. We show that if $v_{r+2}<v_{r+1}$ then $\lambda$ is split.  Indeed, if $v_{r+2}<v_{r+1}$ then the triple $(\lambda_{r+1},\lambda_{r+2},\lambda_{r+3})$ corresponds to  case of Proposition 10.14(v). Thus we can  subtract a multiple of the standard multi-sequence to obtain one in which $(y(r+1,r+2)_i)=0$, and $y(r+1,r+3)_k=y(r+2,r+3)_j=0$ for $j,k\neq p^{v_r}$, and

 $$y(r+2,r+3)_{p^{v_r}}=-y(r+1,r+3)_{p^{v_r}}.$$

 Considering the commuting relation (C) for the pairs $(r,r+1)$ and \\
 $(r+2,r+3)$ we have
 
 $${{\lambda_{r+2}+j}\choose{j}}y(r,r+1)_i={{\lambda_{r}+i}\choose{i}}y(r+2,r+3)_j.$$
 
 Setting $j=p^{v_{r+2}}$ we have  ${{\lambda_{r+2}+j}\choose{j}}\neq 0$ and  $y(r+2,r+3)_{p^{v_{r+2}}}=0$, since $v_{r+2}<v_r$. Therefore, we have $(y(r,r+1)_i)=0$.   We  have deduced that  $(y(r,r+1)_i)=0$ and $(y(r+1,r+2)_j)=0$ and so by Proposition 9.4(i) we get that  $\lambda$ is split.
 
 \q Hence, from now on we can assume  that $(\lambda_{r+1},\lambda_{r+2},\lambda_{r+3})$ is a non-split triple, corresponding to Proposition 10.14(ii), where $v_{r+2}=v_{r+1}$ and since $v_r=v_{r+1}$ we have that  $\lambda_{r+2}=(p^{v_r}-1)+p^{v_r}$,  as  in the statement. 
  
\q We prove that in this case, $\lambda$ is a non-split partition and we have a non-zero coherent multi-sequence  of the given form. 

\q We consider first the triple $(\lambda_{r+1},\lambda_{r+2},\lambda_{r+3})$. By Proposition 10.14(ii) we can assume that after subtracting  a multiple of the  standard multi-sequence  we have  $(y(r+2,r+3)_j)=0$, $y(r+1,r+3)_k=y(r+1,r+2)_i=0$ for $i,k\neq p^{v_r}$, and

 $$y(r+1,r+3)_{p^{v_r}}=-y(r+1,r+2)_{p^{v_r}}.$$

We make the following simple observation. Since the pairs $(\lambda_{r+3},\lambda_t)$ with $t>r+3$ are James partitions we have  

$$l_t<v_{r+3}\leq v_r=v_{r+1}=v_{r+2}$$ 

and so all the pairs $(\lambda_s,\lambda_t)$ with $t>r+3$ are James. Moreover,  all the pairs $(\lambda_s,\lambda_t)$ with $s<r$ are James.  Therefore, the only pairs of rows of $\lambda$ which are not James are those in the quadruple $(\lambda_r,\lambda_{r+1},\lambda_{r+2},\lambda_{r+3})$. 

\q We deal first with the coherent conditions involving $y(r,r+2)_{p^{v_r}}$. For $s,t$ both different from  $r,r+2$ we have that the pairs $(r,r+2)$ and $(s,t)$ are related via the commuting relation (C). If we assume further that $(s,t)\neq(r+1,r+3)$ we have that for such a pair $(\lambda_s,\lambda_t)$ is James and $(y(s,t)_i)=0$. This gives directly that both sides in the commuting relation are zero and so this relation is satisfied. Hence, it remains to investigate the relations between the pairs $(\lambda_r,\lambda_{r+2})$ and $(\lambda_{r+1},\lambda_{r+3})$ and those relations coming from the triples $(\lambda_q,\lambda_r,\lambda_{r+2})$ with $q<r$ and  $(\lambda_r,\lambda_{r+2},\lambda_t)$ with $t>r+2$. 

\q The pairs $(r,r+2)$ and $(r+1,r+3)$ are related via the commuting relation (C) and we have,

$${{\lambda_{r+1}+i}\choose{i}}y(r,r+2)_{p^{v_r}}={{\lambda_{r}+p^{v_r}}\choose{p^{v_r}}}y(r+1,r+3)_i=-y(r+1,r+3)_i$$
where the last equality follows from  the fact that $(\lambda_r+p^{v_r},\lambda_{r+1})$ is James. Now notice that since  $\lambda_{r+1}=(p^{v_r+1}-1)-p^{v_r}+p^{v_r+1}\lambda_{r+1}'$ and $l_{r+3}=v_r$, we also have that for $i\leq l_{r+3}$,  ${{\lambda_{r+1}+i}\choose{i}}\neq 0$ if and only if $i=p^{v_r}$, in which case ${{\lambda_{r+1}+p^{v_r}}\choose{p^{v_r}}}=-1$. Hence, here we recover the relation  $y(r,r+2)_{p^{v_r}}=y(r+1,r+3)_{p^{v_r}}$ that we have in the multi-sequence.

\q We move on to the triples $(\lambda_q,\lambda_r,\lambda_{r+2})$ with $q<r$ and  $(\lambda_r,\lambda_{r+2},\lambda_t)$ with $t>r+2$. For the triple  $(\lambda_q,\lambda_r,\lambda_{r+2})$ we have that $y(r,r+2)_j$ is identified with $y_j$ in the relations (T1) -(T3b). Notice that for this triple we have that $\x=0$, $\z=0$ and $\y=0$ for $j\neq p^{v_r}$ with $y^{{\rm{st}}}_{p^{v_r}}=-1$. Therefore, the relations associated with  this triple are satisfied by our multi-sequence.

\q For the triple $(\lambda_r,\lambda_{r+2},\lambda_t)$ we have that  $y(r,r+2)_i$ is identified with $x_i$ in the relations (T1)-(T3b). If $t>r+3$ we have that since $(\lambda_r,\lambda_t)$ and $(\lambda_{r+2},\lambda_t)$ are James, the standard multi-sequence for this triple is $\y=0$, $\z=0$ and $\x=0$ for $i\neq p^{v_r}$ with $x^{{\rm{st}}}_{p^{v_r}}=-1$. So again, the relations associated with  this triple are satisfied by our multi-sequence.

\q If $t=r+3$ then we have the triple $(\lambda_r,\lambda_{r+2},\lambda_{r+3})$. This triple corresponds to  the case (ii) of Proposition 10.14. Here $y(r,r+2)_{p^{v_r}}$ and $y(r,r+3)_{p^{v_r}}$ are both involved in the relations of this triple. Also, $y(r,r+2)_i$ is identified with $x_i$ and $y(r,r+3)_k$ is identified with $z_k$ in the relations (T1)-(T3b). Now by Proposition 10.14(ii) we get  that $y(r,r+2)_{p^{v_r}}=-y(r,r+3)_{p^{v_r}}$ is a solution for the relations of this triple and this is exactly what we have in our multi-sequence.

\q We now deal  with $y(r+1,r+3)_{p^{v_r}}$. For $s,t$ both different from  $r+1,r+3$ we have that the pairs $(r+1,r+3)$ and $(s,t)$ are related via the commuting relation (C). If we assume further that $(s,t)\neq(r,r+2)$ then  for such a pair $(\lambda_s,\lambda_t)$ is James and $(y(s,t)_i)=0$. This gives directly that both sides in the commuting relation are zero and so this relation is satisfied. Hence, it remains to investigate the relations between the pairs $(\lambda_{r+1},\lambda_{r+3})$ and $(\lambda_r,\lambda_{r+2})$ and those relations coming from the triples $(\lambda_q,\lambda_{r+1},\lambda_{r+3})$ with $q\leq r$ and  $(\lambda_{r+1},\lambda_{r+3},\lambda_t)$ with $t>r+3$.

\q For the pairs $(\lambda_{r+1},\lambda_{r+3})$ and $(\lambda_r,\lambda_{r+2})$ we have already seen that our multi-sequence  satisfies the commuting relation associated with  these pairs in our previous consideration.  Hence, we can move on to the triples \\
$(\lambda_q,\lambda_{r+1},\lambda_{r+3})$ with $q\leq r$ and  $(\lambda_{r+1},\lambda_{r+3},\lambda_t)$ with $t>r+3$.

\q For the triple  $(\lambda_q,\lambda_{r+1},\lambda_{r+3})$ we have that $y(r+1,r+3)_j$ is identified with $y_j$ in relations (T1)-(T3b). Notice that if $q<r$, then    $\x=0$, $\z=0$ and $\y=0$ for $j\neq p^{v_r}$ with $y^{{\rm{st}}}_{p^{v_r}}=-1$. Therefore, the relations associated with  this triple are satisfied by our multi-sequence. If $q=r$ then  $(\lambda_r,\lambda_{r+1},\lambda_{r+3})$ corresponds to  the case of Proposition 10.14(ii). By Proposition 10.14(ii) we have that any solution for this triple has the form $x_i=g\x+ h x'_i$, for some scalars $g,h$, with $x'_i=0$ for $i\neq p^{v_r}$ and $x'_{p^{v_r}}=1$,  $y_j=g \y$ and $z_k=g\z+ h z'_k$, with $z'_k=0$ for $k\neq p^{v_r}$  and $z'_{p^{v_r}}=2$. Also, notice that $\x=\y=\z=0$ for $i,j,k\neq p^{v_r}$ and $x^{\rm{st}}_{p^{v_r}}=y^\st_{p^{v_r}}=z^{\rm{st}}_{p^{v_r}}=-1$. Therefore setting $h=g$ we obtain a solution of the form $x_i=0$, $y_j=z_k=0$ for $j,k\neq p^{v_r}$ and $y_{p^{v_r}}=-z_{p^{v_r}}$. But this is exactly the solution in our multi-sequence, since here $y_{p^{v_r}}$ is $y(r+1,r+3)_{p^{v_r}}$ and $z_{p^{v_r}}$ is  $y(r,r+3)_{p^{v_r}}$. Therefore, we are done.

\q For the triple $(\lambda_{r+1},\lambda_{r+3},\lambda_t)$ now we have that $y(r+1,r+3)_i$ is identified with $x_i$ in the relations (T1)-(T3b). The standard coherent triple satisfies $\y=0$, $\z=0$ and $\x=0$ for $i\neq p^{v_r}$ with $x^{{\rm{st}}}_{p^{v_r}}=-1$. Therefore, the relations associated with  this triple are satisfied by our multi-sequence.

\q The considerations for  $y(r+1,r+2)_{p^{v_r}},$ and $ y(r,r+3)_{p^{v_r}}$ are identical with those of the  last two cases and  we leave them to the reader to check.

\end{proof}

\begin{example} One can easily verify that for $p=3$ the partition $\lambda=(1,1,1,1)$ satisfies the conditions of Lemma 12.14 and that for $n\geq 4$ we have  $\Ext^1_{B(n)}(S^4E,K_\lambda)=H^1(\Sigma_4,\Sp(1,1,1,1))=K$, where $\Sp(1,1,1,1)$ is the sign representation for $\Sigma_4$. 

\end{example}

\newpage

\subsection{The case in which $(\lambda_{r+2},\lambda_{r+3})$ is James.}

\q We consider now the case in which  the pair $(\lambda_{r+2},\lambda_{r+3})$ is a James partition. We have the following result.

\begin{proposition} Suppose  $\lambda$ is  a partition which is not James and $r<n-1$ minimal such that $(\lambda_r,\lambda_{r+1})$ is not James. Assume further that $(\lambda_{r+2},\dots,\lambda_n)$ is a James partition. Then $\lambda$ is non-split if and only if one of the following holds. 

\smallskip

(i) $(\lambda_{r+1},\lambda_{r+2})$ is not James and $(\lambda_r,\lambda_{r+1},\lambda_{r+2})$ is a non-split triple.

(ii) $(\lambda_{r+1},\lambda_{r+2})$ is James, $(\lambda_r,\lambda_{r+1})$ splits and $(\lambda_r,\lambda_{r+1},\lambda_{r+2})$ is a non-split triple, with the extra condition that $l_{r+3}<l_{r+2}$ if $p=2$ and $n\geq r+3$.

(iii) $(\lambda_{r+1},\lambda_{r+2})$ is James, $(\lambda_r,\lambda_{r+1})$ is pointed and there is no $q<r$ with $v_q=\len_p(\lambda_r+p^{l_{r+1}})$.

\q In all these cases a non-standard multi-sequence $(y(t,u)_i)$ is obtained from a non-standard coherent triple for  $(\lambda_r,\lambda_{r+1},\lambda_{r+2})$
 with $(y(t,u)_i)=0$ for  $(t,u)\neq (r,r+1), (r+1,r+2)$ and $(r,r+2)$.  

\end{proposition}

\q We note that for $p\geq 3$ the condition $l_{r+3}<l_{r+2}$ in case (ii) of Proposition 12.16 is automatic,  coming from the fact that $(\lambda_{r+2},\lambda_{r+3})$ is James and the form of the non split triple $(\lambda_r,\lambda_{r+1},\lambda_{r+2})$ with $(\lambda_r,\lambda_{r+1})$ split and $(\lambda_{r+1},\lambda_{r+2})$ James, which is the case (iv) in Proposition 10.14.

\medskip

\q We give a proof of Proposition 12.16 in a series of Lemmas corresponding to the forms  of the non-split triples of Propositions 10.14 and 11.16. Before we embark on the proof of   Lemmas 12.18 to 12.28 we make the following remark that will be used repeatedly.

\begin{remark}
{\rm{Let $\lambda$ be a partition which is not James and assume that  $r$ is minimal such that $(\lambda_r,\lambda_{r+1})$  not  James.  Assume further that $(\lambda_{r+2},\dots,\lambda_n)$ is a James partition. Then for every pair $(\lambda_s,\lambda_t)$ with $s\neq r,r+1$ the partition  $(\lambda_s,\lambda_t)$ is  James.

\q Indeed, since $(\lambda_1,\dots,\lambda_r)$ and $(\lambda_{r+2},\dots,\lambda_n)$ are James partitions, we have   $l_{q+1}<v_q$ for $q\leq r-1$ and $q\geq r+2$. Now, since $s\neq r,r+1$ we have that  $s\leq r-1$ or $s\geq r+2$. Also, for $t>s$ we have that $l_t\leq l_{s+1}<v_s$ and so $(\lambda_s,\lambda_t)$  is James.}}

\end{remark}

\q We state and prove now the following ten Lemmas. Recall that $\lambda$ is a partition which is not James, $1\leq r<n-1$ is minimal such that $(\lambda_r,\lambda_{r+1})$  is not  James and $(\lambda_{r+2},\dots,\lambda_n)$ is  James. 

\q The Lemmas 12.18 to 12.23 which follow correspond to the cases in which $(\lambda_r,\lambda_{r+1},\lambda_{r+2})$ is as described in Proposition 10.14, (i)-(v). We therefore adopt as  standing assumptions that $(\lambda_r,\lambda_{r+1})$ is split and \\
$l_{r+1}< \val_p(\lambda_r+p^{v_r}+1)$.

\q We first  treat the case in which $(\lambda_{r+1},\lambda_{r+2},\lambda_{r+3})$ has the form described in Proposition 10.14(i).

\begin{lemma} 
Suppose $l_{r+2}\geq v_r=v_{r+1}$, with $l_{r+1}<\val_p(\lambda_r+p^{v_r}+1)$ and $\val_p(\lambda_{r+1}-p^{v_r}+1)>l_{r+2}$.  Then  $\lambda$ is non-split and $E(\lambda)$ is spanned by the standard multi-sequence and a multi-sequence $(y(t,u)_i)$ with $y(r,r+1)_{p^{v_r}}\neq 0$ and $y(t,u)_i=0$ for all other choices of $t,u$ and $i$.
\end{lemma}

\begin{proof} Let $(y(t,u)_i)$ be a multi-sequence as in the statement. It suffices to prove that this is coherent. By  Proposition 10.14(i) the sequence satisfies the coherence  conditions for the triple $(\lambda_r,\lambda_{r+1},\lambda_{r+2})$. 

\q Consider $(s,t)$ with $s<t$, and   $s$ and  $t$   different from  $r$ and $r+1$. Then, $(r,r+1)$ and $(s,t)$ are related  via the commuting relation (C). Since $(y(s,t)_j)=0$  the condition is 

 $${{\lambda_s+j}\choose{j}}y(r,r+1)_i=0.$$
 
 However, since $s\neq r,r+1$ we have that $(\lambda_s,\lambda_t)$ is a James partition by Remark 12.17 and so ${{\lambda_s+j}\choose{j}}=0$ for  $1\leq j\leq \lambda_t$. Hence, this relation is satisfied.  
 
 \q Therefore we can assume that $s$ or $t$ are equal to $r$ or $r+1$ and consider the triples $(\lambda_q,\lambda_r,\lambda_{r+1})$ and  $(\lambda_r,\lambda_{r+1},\lambda_t)$.
 
 \q  We consider the triple $(\lambda_q,\lambda_r,\lambda_{r+1})$ first. Here, $y(r,r+1)_j$ is identified with $y_j$ in the relations (T1)-(T3b). Since $(\lambda_q,\lambda_r)$ and $(\lambda_r+p^{\v_r},\lambda_{r+1})$ are James partitions we have that the standard solution for this triple is  $\x=\z=0$ and $\y=0$ for $j\neq p^{v_r}$ with $y^\st_{p^{v_r}}=-1$. Therefore the coherence  conditions are satisfied for this triple.

\q We consider now  the triple $(\lambda_r,\lambda_{r+1},\lambda_t)$. Here $y(r,r+1)_i$ is identified with $x_i$ in the relations (T1)-(T3b). But as $l_t\leq l_{r+2}$ and any value of  $y(r,r+1)_{p^{v_r}}$  (and all other values  $0$) is coherent for  $(\lambda_r,\lambda_{r+1},\lambda_{r+2})$ so we  get directly that the coherence conditions are satisfied also for  $(\lambda_r,\lambda_{r+1},\lambda_t)$.
\end{proof}

\q We next consider the case in which  $(\lambda_{r+1},\lambda_{r+2},\lambda_{r+3})$ has the form described in Proposition 10.14(ii).

\begin{lemma}

Suppose that $l_{r+2}=v_r=v_{r+1}$, with \\
$l_{r+1}<\val_p(\lambda_r+p^{v_r}+1)$,  $(\lambda_{r+1})_{v_r}\neq 0$ and $\lambda_{r+2}=\hat{\lambda}_{r+2}+p^{v_r}$ with \\
$0\leq \hat{\lambda}_{r+2}<p^{v_r}$. Then $\lambda$ is non-split.

\q Moreover  $E(\lambda)$ is spanned by the standard multi-sequence and a non-zero multi-sequence $(y(t,u)_i)$ such that 
 $y(r,r+1)_i=y(r,r+2)_k=0$ for $i,k\neq p^{v_r}$, with
  $$y(r,r+2)_{p^{v_r}}+(\lambda_{r+1})_{v_r} y(r,r+1)_{p^{v_r}}=0$$ 
  and $(y(s,t)_j)=0$ for   $(s,t)\neq (r,r+1), (r,r+2)$.
 
 \end{lemma}
 
 \begin{proof}
 
 By Remark 12.17, $(\lambda_s,\lambda_t)$ is a James partition for $s\neq r,r+1$. We show that in this case we also have that for $t\geq r+3$ the pairs $(\lambda_r,\lambda_t)$ and $(\lambda_{r+1},\lambda_t)$ are James.

\q First notice that since $(\lambda_{r+1})_{v_r}\neq 0$ and $v_r=v_{r+1}$, the characteristic of $K$ is odd.  Now since $(\lambda_{r+2},\lambda_{r+3})$ is James and $\lambda_{r+2}= \hat{\lambda}_{r+2}+p^{v_r}$ with $\hat{\lambda}_{r+2}<p^{v_r}$, (and since $\char(K)\neq 2$), we have that $\lambda_{r+3}\leq \hat{\lambda}_{r+2}<p^{v_r}$. Therefore, $l_{r+3}<v_r=v_{r+1}$ and so $(\lambda_r,\lambda_{r+3})$ and $(\lambda_{r+1},\lambda_{r+3})$ are James partitions. Moreover for $t>r+3$ we have   $l_t\leq l_{r+3}$ and so $(\lambda_r,\lambda_t)$ and $(\lambda_{r+1},\lambda_t)$ are also James partitions. Therefore, we deduce that the only pairs which are not James come from $(\lambda_r,\lambda_{r+1}), (\lambda_{r+1},\lambda_{r+2})$ and $(\lambda_r,\lambda_{r+2})$.

\q Let $(y(t,u)_i)$ be a multi-sequence as in the statement of the Lemma.  By Proposition 10.14(ii) the coherence condition is satisfied for $(\lambda_r,\lambda_{r+1},\lambda_{r+2})$.

\q We deal first with coherence conditions involving $y(r,r+1)_{p^{v_r}}$. For a pair  $(s,t)$, with $s<t$, and both $s$ and $t$  different from  $r$ and $r+1$ we have that  $(r,r+1)$ and $(s,t)$ are  related via the commuting relation (C). Since $(y(s,t)_j)=0$ and $(\lambda_s,\lambda_t)$ is a James partition we have though that this relation becomes a tautology.  

\q Therefore, it remains to consider the triples of the form $(\lambda_q,\lambda_r,\lambda_{r+1})$ with $q<r$ and  $(\lambda_r,\lambda_{r+1},\lambda_t)$ with $t>r+2$.

\q For the triple  $(\lambda_q,\lambda_r,\lambda_{r+1})$ with $q<r$ we have that $y(r,r+1)_j$ is identified with $y_j$ in the relations (T1)-(T3b). Here the standard solution  for this triple has the form $\x=\z=0$ and $\y=0$ for $j\neq p^{v_r}$ with $y^\st_{p^{v_r}}=-1$. Therefore, the relations associated with  this triple are satisfied by our multi-sequence.

\q For the triple $(\lambda_r,\lambda_{r+1},\lambda_t)$ with $t>r+2$ we have that $y(r,r+1)_i$ is identified with $x_i$ in the relations (T1)-(T3b). Here the standard solution  for this triple has the form $\x=0$ for $i\neq p^{v_r}$, $x^\st_{p^{v_r}}=-1$ and $\y=\z=0$. Therefore, the relations associated with  this triple are satisfied by our multi-sequence.

\q We now deal with coherence conditions involving  $y(r,r+2)_{p^{v_r}}$.  For a pair $(s,t)$ with $s<t$ and both $s$ and $t$ different from  $r$ and $r+2$ we have that $(r,r+2)$ and $(s,t)$ are related via the commuting relation (C). As $(\lambda_s,\lambda_t)$ is James and $(y(s,t)_j)=0$ for these values of $(s,t)$ we have directly that both sides in the commuting relation are $0$.

\q Therefore it remains to consider the triples $(\lambda_q,\lambda_r,\lambda_{r+2})$ with $q<r$ and $(\lambda_r,\lambda_{r+2},\lambda_t)$ with $t>r+2$. 

\q For the triple  $(\lambda_q,\lambda_r,\lambda_{r+2})$ with $q<r$ we have that $y(r,r+2)_j$ is identified with  $y_j$ in the relations (T1)-(T3b). Here, the standard solution  for this triple is $\x=\z=0$ and $\y=0$ for $j\neq p^{v_r}$ with $y^\st_{p^{v_r}}=-1$. Therefore, the relations associated with  this triple are satisfied by our multi-sequence.

\q For the triple $(\lambda_r,\lambda_{r+2},\lambda_t)$ we have that $y(r,r+2)_i$ is identified with $x_i$ in the relations (T1)-(T3b). Notice that since $l_t<v_r$  we have that the standard solution for this triple is $\x=0$ for $i\neq p^{v_r}$ and $x^{\rm{st}}_{p^{v_r}}=-1$ and $\y=\z=0$. Hence, again the relations associated with  this triple are satisfied by our multi-sequence.

\end{proof}

\q We next consider the case in which  $(\lambda_{r+1},\lambda_{r+2},\lambda_{r+3})$ has the form described in Proposition 10.14(iii).

\begin{lemma}

Suppose $l_{r+2}>v_r=v_{r+1}$, with $\len_p(\lambda_{r+1}+p^{l_{r+2}})<\val_p(\lambda_r+p^{v_r}+1)$,  $\val_p(\lambda_{r+1}-p^{v_r}+1)=l_{r+2}$ and $\lambda_{r+2}=\hat{\lambda}_{r+2}+p^{l_{r+2}}$, with $\hat{\lambda}_{r+2}<p^{v_r}$. Then  $\lambda$ is non-split.

\q The space $E(\lambda)$ is spanned by the standard multi-sequence and a non-zero multi-sequence $(y(t,u)_i)$ such that 
 $y(r,r+1)_i=0$ for $i\neq p^{v_r}$, and \\
 $y(r+1,r+2)_j=0$ for $j\neq p^{l_{r+2}}$ with\\
  $$y(r+1,r+2)_{p^{l_{r+2}}}+(\lambda_{r+1})_{l_{r+2}} y(r,r+1)_{p^{v_r}}=0$$
   and $(y(s,t)_j)=0$ for $(s,t)\neq (r,r+1), (r+1,r+2)$.

\end{lemma}

\begin{proof}

By Remark 12.17 we have that $(\lambda_s,\lambda_t)$ is a James partition for $s\neq r,r+1$. We show that in this case we also have that for $t\geq r+3$, the pairs $(\lambda_r,\lambda_t)$ and $(\lambda_{r+1},\lambda_t)$ are James.

\q First notice that since $l_{r+2}>v_r$ and $\lambda_{r+2}= \hat{\lambda}_{r+2}+p^{l_{r+2}}$ with $\hat{\lambda}_{r+2}<p^{v_r}$ our assumption that $(\lambda_{r+2},\lambda_{r+3})$ is James gives directly  that $l_{r+3}<v_r=v_r+1$ and so we have that $(\lambda_r,\lambda_{r+3})$ and $(\lambda_{r+1},\lambda_{r+3})$ are James partitions. Now, for $t\geq r+3$ we have $l_t\leq l_{r+3}$ and so $(\lambda_r,\lambda_t)$ and $(\lambda_{r+1},\lambda_t)$ are also James partitions. Therefore, we deduce that the only pairs which are not James come from $(\lambda_r,\lambda_{r+1}), (\lambda_{r+1},\lambda_{r+2})$ and $(\lambda_r,\lambda_{r+2})$.

\q Let $(y(t,u)_i)$ be a multi-sequence as in the statement of the Lemma. By Proposition 10.14(iii) the coherence condition is satisfied for $(\lambda_r,\lambda_{r+1},\lambda_{r+2})$. 

\q We deal first with coherence conditions involving $y(r,r+1)_{p^{v_r}}$.  For a pair  $(s,t)$, with $s<t$ and both $s$ and $t$  different from  $r$ and $r+1$ we have $(r,r+1)$ and $(s,t)$ are related via the commuting relation (C). Since $(y(s,t)_j)=0$ and $(\lambda_s,\lambda_t)$ is a James partition we get directly that this relation becomes a tautology.  

\q Therefore, it remains to consider the triples of the form $(\lambda_q,\lambda_r,\lambda_{r+1})$ with $q<r$ and  $(\lambda_r,\lambda_{r+1},\lambda_t)$ with $t>r+2$.

\q For the triple  $(\lambda_q,\lambda_r,\lambda_{r+1})$ with $q<r$ we have that $y(r,r+1)_j$ is identified with  $y_j$ in the relations (T1)-(T3b). Here the standard solution  for this triple has the form $\x=\z=0$ and $\y=0$ for $j\neq p^{v_r}$ with $y^\st_{p^{v_r}}=-1$. Therefore the coherence  conditions are satisfied for this triple.

\q For the triple $(\lambda_r,\lambda_{r+1},\lambda_t)$ with $t>r+2$ we have that  $y(r,r+1)_i$ is identified with $x_i$ in the relations (T1)-(T3b). The standard solution for this triple is $\x=0$ for $i\neq p^{v_r}$ and $x^{\rm{st}}_{p^{v_r}}=-1$ and $\y=\z=0$. So, again the coherence  conditions are satisfied for this triple.

\q We now  deal with coherence conditions involving $y(r+1,r+2)_{p^{l_{r+2}}}$.   For a pair $(s,t)$ with $s<t$ and $s$ and $t$ both different from  $r+1$ and $r+2$ we have that $(r+1,r+2)$ and $(s,t)$ are  related  via the commuting relation (C). As $(\lambda_s,\lambda_t)$ is a James partition and $(y(s,t)_j)=0$ for such $(s,t)$, this relation becomes a tautology.  

\q Therefore, it remains to consider the triples $(\lambda_q,\lambda_{r+1},\lambda_{r+2})$ with $q<r$ and $(\lambda_{r+1},\lambda_{r+2},\lambda_t)$ with $t>r+2$.

\q For the triple  $(\lambda_q,\lambda_{r+1},\lambda_{r+2})$ with $q<r$ we have that $y(r+1,r+2)_j$ is identified with  $y_j$ in the relations (T1)-(T3b). Notice now that the partitions $(\lambda_q,\lambda_{r+1},\lambda_{r+2})$ satisfies the properties of the case of Lemma 8.1 and so we get that the coherence  conditions are satisfied for this triple.

\q For the triple $(\lambda_{r+1},\lambda_{r+2},\lambda_t)$ we have that $y(r+1,r+2)_i$ is identified with $x_i$ in the relations (T1)-(T3b). This triple now satisfies the properties of the case of Proposition 11.16(iv) and we have that $y(r+1,r+2)_{p^{l_{r+2}}}\neq0$, with all the other values  $0$, is coherent  for this triple.

\end{proof}

\q We next consider the case in which  $(\lambda_{r+1},\lambda_{r+2},\lambda_{r+3})$ has the form described in Proposition 10.14 (iv).

\begin{lemma} Suppose $l_{r+2}=v_r<v_{r+1}$, with $\len_p(\lambda_{r+1}+p^{l_{r+2}})<\val_p(\lambda_r+p^{v_r}+1)$ and $\lambda_{r+2}=\hat{\lambda}_{r+2}+p^{l_{r+2}}$ with $\hat{\lambda}_{r+2}<p^{v_r}$. Then

(i) For $p\geq 3$ the partition $\lambda$ is non-split.

(ii) For $p=2$ the partition $\lambda$ is non-split if and only if $l_{r+3}<l_{r+2}$.

 \q In the  non split case $E(\lambda)$ is spanned by the standard multi-sequence and a non-zero multi-sequence $(y(t,u)_i)$ satisfying 
  $y(r+1,r+2)_j=y(r,r+2)_k=0$ for $j,k\neq p^{v_r}$ with 
  
  $$y(r+1,r+2)_{p^{v_r}}+ y(r,r+2)_{p^{v_r}}=0$$
  
 and $(y(s,t)_j)=0$ for  $(s,t)\neq (r+1,r+2), (r,r+2)$.

\end{lemma}

\begin{proof}

Assume first that $l_{r+3}<l_{r+2}$. Notice that this is always the case if $\char(K)\geq 3$, since $(\lambda_{r+2},\lambda_{r+3})$ is James and $\lambda_{r+2}=\hat{\lambda}_{r+2}+p^{{v_r}}$, with $\hat{\lambda}_{r+2}<p^{v_r}$.  

\q By Remark 12.17, $(\lambda_s,\lambda_t)$ is a James partition for $s\neq r,r+1$. Now, for $t\geq r+3$ we have  $l_t\leq l_{r+3}<v_r<v_{r+1}$, and so we get that $(\lambda_r,\lambda_t)$ and $(\lambda_{r+1},\lambda_t)$ are also James pairs for $t\geq r+3$. Therefore, the only pairs which are not James come from $(\lambda_r,\lambda_{r+1})$ and $(\lambda_r,\lambda_{r+2})$.

 \q Let $(y(t,u)_i)$ be a multi-sequence as in the statement of the Lemma. By Proposition 10.14 (iv) the coherence condition is satisfied for $(\lambda_r,\lambda_{r+1},\lambda_{r+2})$. 
 
 \q We deal first with coherence conditions involving $y(r+1,r+2)_{p^{v_r}}$. For a pair  $(s,t)$, with $s<t$, and both $s$ and $t$  different from  $r+1$ and $r+2$ we have $(r+1,r+2)$ and $(s,t)$ are related via the commuting relation (C). Since $(y(s,t)_j)=0$ and $(\lambda_s,\lambda_t)$ is a James partition for such $(s,t)$ we get directly that this relation becomes a tautology. 
 
 \q Therefore, it remains to consider the triples of the form  $(\lambda_{r+1},\lambda_{r+2},\lambda_t)$ with $t>r+2$ and $(\lambda_q,\lambda_{r+1},\lambda_{r+2})$ with $q<r$.

\q For the triple $(\lambda_{r+1},\lambda_{r+2},\lambda_t)$ with $t>r+2$ we have that $y(r+1,r+2)_i$ is identified with $x_i$ in the relations (T1)-(T3b). Now, since $l_{r+2}=v_r<v_{r+1}$  and $(\lambda_{r+2},\lambda_t)$ is a James partitions we get that the partition  $(\lambda_{r+1},\lambda_{r+2},\lambda_t)$ is James. Moreover, as $l_t\leq l_{r+3}<l_{r+2}$ we have by Lemma 7.7 that the coherence  conditions are satisfied for this triple.

\q For the triple  $(\lambda_q,\lambda_{r+1},\lambda_{r+2})$ with $q<r$ we have that $y(r+1,r+2)_j$ is identified with  $y_j$ in the relations (T1)-(T3b). Again we have that this partition is James. Also, since $\len_p(\lambda_{r+1}+p^{l_{r+2}})<\val_p(\lambda_r+p^{v_r}+1)$ we have that $\len_p(\lambda_{r+1}+p^{l_{r+2}})<v_q$. Therefore, if $l_{r+2}=l_{r+3}$ we get that the coherence  conditions are satisfied by Lemma 7.7(i) and if $l_{r+2}>l_{r+3}$ the conditions are satisfied by Lemma 7.7(iii).

 \q We deal now with the  coherence conditions involving $y(r,r+2)_{p^{v_r}}$. For a pair $(s,t)$ with $s<t$, and both $s$ and $t$  different from  $r$ and $r+2$ we have that $(r,r+2)$ and $(s,t)$ are  related only via the commuting relation (C).  Since $(y(s,t)_j)=0$ and $(\lambda_s,\lambda_t)$ is a James partition for such $(s,t)$ we get directly that this relation becomes a tautology.

 \q Therefore it remains to consider the triples $(\lambda_{r},\lambda_{r+2},\lambda_t)$ with $t>r+2$ and $(\lambda_q,\lambda_{r},\lambda_{r+2})$ with $q<r$. 
 
 \q For the triple $(\lambda_{r},\lambda_{r+2},\lambda_t)$ we have that $y(r,r+2)_i$ is identified with $x_i$ in the relations (T1)-(T3b). The standard solution for this triple is $\x=0$ for $i\neq p^{v_r}$ and $x^{\rm{st}}_{p^{v_r}}=-1$ and $\y=\z=0$. So the coherence  conditions are satisfied.

\q For the triple  $(\lambda_q,\lambda_{r},\lambda_{r+2})$ with $q<r$ we have that $y(r,r+2)_j$ is identified with  $y_j$ in the relations (T1)-(T3b).
Now here we have that the standard solution for this triple has the form $\x=\z=0$ and $\y=0$ for $j\neq p^{v_r}$ with $y^\st_{p^{v_r}}=-1$ and so again we are fine.

\q We show now that if $l_{r+3}=l_{r+2}$ then $\lambda$ is split.  Notice that since $(\lambda_{r+2},\lambda_{r+3})$ is James and $\lambda_{r+2}=\hat{\lambda}_{r+2}+p^{l_{r+2}}$, this is valid only in characteristic $2$.

\q We consider an arbitrary coherent multi-sequence for $\lambda$.  By Proposition 10.14(iv) we have that after subtracting a multiple of the  standard solution we can assume that $(y(r,r+1)_i)=0$ , $y(r+1,r+2)_j=0$ for $j\neq 2^{v_r}$ and $y(r,r+2)_k=0$ for $k\neq 2^{v_r}$, with $y(r+1,r+2)_{2^{v_r}}+y(r,r+2)_{2^{v_r}}=0$.  By Proposition 9.4(i) it is enough to show that   $(y(r,r+1)_i)$ and $(y(r+1,r+2)_j)$ are zero. Hence it remains  to prove that $y(r+1,r+2)_{2^{v_r}}=0$. Now the pairs $(r,r+3)$ and $(r+1,r+2)$ are related via the commuting relation (C) and since $(\lambda_{r+1},\lambda_{r+2})$ is James we have that

$${{\lambda_r+i}\choose{i}}y(r+1,r+2)_j=0, {\rm{with}}\ i\leq \lambda_{r+3}.$$

As $l_{r+3}=l_{r+2}=v_r$ we can set $i=2^{v_r}$ and we have ${{\lambda_r+2^{v_r}}\choose{2^{v_r}}}\neq 0$. Therefore, $y(r+1,r+2)_j=0$ and we are done.

\end{proof}

\begin{example}

\rm We give a concrete example to describe the failure of the above result for the situation when $l_{r+3}=l_{r+2}=v_r<v_{r-1}$ in characteristic $2$. It is easy to see that the partition $(2,1,1,1)$ satisfies all the assumptions of Lemma 12.20 with $v_1=l_3=l_4=0$ and $v_2=1$. By Proposition 10.14(iv) we have that  $\Ext^1_{B(n)}(S^4E, K_{(2,1,1)})=K$. However  $\Ext^1_{B(n)}(S^5E, K_{(2,1,1,1)})=0$, for $n\geq 5$, since the partition $(2,1,1,1)$ has 2-core $(2,1)$ and so $\Ext^1_{B(n)}(S^5E, K_{(2,1,1,1)})=\Ext^1_{G(n)}(S^5E, \nabla(2,1,1,1))=0$.  

\end{example}

\q We next consider the case in which  $(\lambda_{r+1},\lambda_{r+2},\lambda_{r+3})$ has the form described in Proposition 10.14(v).

\begin{lemma}

 Suppose $l_{r+2}=v_r>v_{r+1}$, with $\len_p(\lambda_{r+1}+p^{l_{r+2}})<\val_p(\lambda_r+p^{v_r}+1)$ and $\lambda_{r+2}=\hat{\lambda}_{r+2}+p^{l_{r+2}}$, with $\hat{\lambda}_{r+2}<p^{v_{r+1}}$. Then  $\lambda$ is non-split. 

 \q Moreover $E(\lambda)$ is spanned by the standard multi-sequence and a non-zero multi-sequence $(y(t,u)_i)$ satisfying $y(r+1,r+2)_j=y(r,r+2)_k=0$ for $j,k\neq p^{v_r}$, with 
 
 $$y(r+1,r+2)_{p^{v_r}}+ y(r,r+2)_{p^{v_r}}=0$$
 
  and $(y(s,t)_j)=0$ for $(s,t)\neq (r+1,r+2), (r,r+2)$.

\end{lemma}

\begin{proof}

 By Remark 12.17, $(\lambda_s,\lambda_t)$ is a James partition for $s\neq r,r+1$. We show that in this case we also have that $(\lambda_r,\lambda_t)$ and $(\lambda_{r+1},\lambda_t)$, with $t\geq r+3$ are James.

\q Since $(\lambda_{r+2},\lambda_{r+3})$ is James, $\lambda_{r+2}= \hat{\lambda}_{r+2}+p^{l_{r+2}}$, with $\hat{\lambda}_{r+2}<p^{v_{r+1}}$, and $v_r>v_{r+1}$  we must have that $\lambda_{r+3}<p^{v_{r+1}}$. Hence $l_t\leq l_{r+3}<v_{r+1}<v_r$ and so $(\lambda_r,\lambda_t)$ and $(\lambda_{r+1},\lambda_t)$ are James partitions. Therefore, the only pairs which are not James come from $(\lambda_r,\lambda_{r+1}), (\lambda_{r+1},\lambda_{r+2})$ and $(\lambda_r,\lambda_{r+2})$.

 \q Let $(y(t,u)_i)$ be a multi-sequence as in the statement of the Lemma. By Proposition 10.14(v) the coherence condition is satisfied for $(\lambda_r,\lambda_{r+1},\lambda_{r+2})$.     
 
 \q We first  deal  with the  coherence conditions involving $y(r+1,r+2)_{p^{v_r}}$.   For a pair $(s,t)$, with $s<t$, and both $s$ and $t$  different from  $r+1$ and $r+2$ we have that $(r+1,r+2)$ and $(s,t)$ are  related via the commuting relation (C). Since $(y(s,t)_j)=0$ and $(\lambda_s,\lambda_t)$ is a James partition  for such $(s,t)$ we have that this relation becomes a tautology.
 
 \q   Therefore, it remains to consider the triples of the form  $(\lambda_{r+1},\lambda_{r+2},\lambda_t)$ with $t>r+2$ and $(\lambda_q,\lambda_{r+1},\lambda_{r+2})$ with $q<r$.

 \q For a  triple $(\lambda_{r+1},\lambda_{r+2},\lambda_t)$ with $t>r+2$ we have that  $y(r+1,r+2)_i$ is identified with $x_i$ in the relations (T1)-(T3b). Notice that $(\lambda_{r+1},\lambda_{r+2})$ is a pointed partition and that $l_t<v_{r+1}$. Hence this triple has the form described in Proposition 11.16(iv) and the coherence condition holds for our multi-sequence.

 \q For the triple  $(\lambda_q,\lambda_{r+1},\lambda_{r+2})$ with $q<r$ we have that $y(r+1,r+2)_j$ is identified with  $y_j$ in the relations (T1)-(T3b). Now, since $(\lambda_q,\lambda_{r+1})$ is James and $(\lambda_{r+1},\lambda_{r+2})$ is a pointed partition with $\len_p(\lambda_{r+1}+p^{v_r})<\val_p(\lambda_r+p^{v_r}+1)\leq v_s$ we have that this is the case appearing in Lemma 8.1 and this gives directly that the coherence conditions hold.

  \q We now deal  with the  coherence conditions involving  $y(r,r+2)_{p^{v_r}}$.
  
  \q For a pair $(s,t)$, with $s<t$, and both $s$ and $t$  different from $r$ and $r+2$  we have that $(r,r+2)$ and $(s,t)$ are  related via the commuting relation (C). As $(\lambda_s,\lambda_t)$ is a James partition and $(y(s,t)_j)=0$ for these values of $(s,t)$ we have directly that both sides in the commuting relation are $0$.
  
  \q  Therefore it remains to consider the triples $(\lambda_r,\lambda_{r+2},\lambda_t)$ with $t>r+2$ and $(\lambda_q,\lambda_r,\lambda_{r+2})$ with $q<r$.
  
  \q For the triple $(\lambda_r,\lambda_{r+2},\lambda_t)$ we have that $y(r,r+2)_i$ is identified with $x_i$ in the relations (T1)-(T3b). Notice that since $l_t<v_r$  and $(\lambda_r+p^{v_r},\lambda_{r+1})$ is James, the standard solution for this triple is $\x=0$ for $i\neq p^{v_r}$ and $x^{\rm{st}}_{p^{v_r}}=-1$ and $\y=\z=0$. So, our multi-sequence satisfies the conditions for this triple.
  
\q For the triple  $(\lambda_q,\lambda_r,\lambda_{r+2})$ with $q<r$ we have that $y(r,r+2)_j$ is identified with  $y_j$ in the relations (T1)-(T3b). Here, the standard solution for this triple is $\x=\z=0$ and $\y=0$ for $j\neq p^{v_r}$ with $y^\st_{p^{v_r}}=-1$ and so again the coherence conditions  for our multi-sequence are satisfied.

\end{proof}

\q The remaining five Lemmas correspond  to the cases in which $(\lambda_r,\lambda_{r+1},\lambda_{r+2})$ has the form described in Proposition 11.16 (i)-(v). We therefore adopt as  the standing assumption  that $(\lambda_r,\lambda_{r+1})$ is pointed.

\q We begin with the case in which  $(\lambda_r,\lambda_{r+1},\lambda_{r+2})$ has the form described in Proposition 11.16(i).

\begin{lemma}

Suppose $l_{r+1}>l_{r+2}\geq v_r=v_{r+1}$, with $\val_p(\lambda_r+p^{v_r}+1)\geq l_{r+1}$  and  $\lambda_{r+2}=p^{v_r}-1+p^{l_{r+1}}$. Then  $\lambda$ is non-split. 

\q The space $E(\lambda)$ is spanned by the standard multi-sequence and a non-zero multi-sequence  $(y(t,u)_i)$ satisfying   $y(r,r+1)_i=0$ for $i\neq p^{v_r}$ and the extension sequence  $(y(t,u)_j)=0$  for $(t,u) \neq (r,r+1)$.

\end{lemma}

\begin{proof} Let $(y(t,u)_i)$ be a multi-sequence as in the statement of the Lemma. By Proposition 11.16(i),  the coherence conditions are  satisfied for \\
 $(\lambda_r,\lambda_{r+1},\lambda_{r+2})$.     
 
 \q For $(s,t)$ with $s<t$, and both $s$ and $t$ different from  $r$ and $r+1$, we have that $(r,r+1)$ and $(s,t)$ are related via the commuting relation (C). Now, since $(\lambda_s,\lambda_t)$ is James by Remark 12.17 and $(y(s,t)_j)=0$ for such $(s,t)$ we have  that both sides  are $0$.

\q  Hence it remains  to consider the triples $(\lambda_r,\lambda_{r+1},\lambda_t)$ with $t>r+2$ and $(\lambda_q,\lambda_r,\lambda_{r+1})$ with $s<r$.  

\q For the triple $(\lambda_r,\lambda_{r+1},\lambda_t)$, we have that  $y(r,r+1)_i$ is identified with $x_i$ in the relations (T1)-(T3b). 
But as $l_t\leq l_{r+2}$ and any value of  $y(r,r+1)_{p^{v_r}}$  (and all other values  $0$) is coherent for  $(\lambda_r,\lambda_{r+1},\lambda_{r+2})$ then we get directly that the coherence conditions are satisfied also for  $(\lambda_r,\lambda_{r+1},\lambda_t)$.

\q For the triple $(\lambda_q,\lambda_r,\lambda_{r+1})$, we have that  $y(r,r+1)_i$ is identified with  $y_j$ in the relations (T1)-(T3b). Since $(\lambda_q,\lambda_r)$ is James and $\len_p(\lambda_r+p^{v_r})=l_r$ It is a trivial check to see that our solution here satisfies the relations (T1)-(T3b) for this triple. 

\end{proof}

\q We next consider  the case in which  $(\lambda_r,\lambda_{r+1},\lambda_{r+2})$ has the form described in Proposition 11.16(ii).

\begin{lemma}

Suppose $l_{r+1}=l_{r+2}> v_r=v_{r+1}$, with $\len_p(\lambda_{r+1}+p^{l_{r+2}})<\val_p(\lambda_r+p^{v_r}+1)$ and  $\lambda_{r+1}=p^{v_r}-1+p^{l_{r+1}}$. Then $\lambda$ is non-split. 

\q The space $E(\lambda)$ is spanned by the standard multi-sequence and a non-zero multi-sequence $(y(t,u)_i)$ satisfying  $y(r,r+1)_i=0$ for $i\neq p^{v_r}$ and $y(r+1,r+2)_j=0$ for $j\neq p^{l_{r+1}}$ with $$y(r,r+1)_{p^{v_r}}+ y(r+1,r+2)_{p^{l_{r+1}}}=0$$
 and $(y(t,u)_j)=0$ for $(t,u)\neq (r+1,r+2), (r,r+1)$.
\end{lemma}

\begin{proof} By Remark 12.17 we have that $(\lambda_s,\lambda_t)$ is a James partition for $s\neq r,r+1$. We show that in this case we also have that $(\lambda_r,\lambda_t)$ and $(\lambda_{r+1},\lambda_t)$, with $t\geq r+3$, are also James. Notice that since $\lambda_{r+2}=\hat{\lambda}_{r+2}+p^{l_{r+1}}$ with $\hat{\lambda}_{r+2}<p^{v_r}$, $l_{r+1}> v_r$ and $(\lambda_{r+2},\lambda_{r+3})$ is James we get that $\lambda_{r+3}\leq \hat{\lambda}_{r+2}<p^{v_r}$. Therefore we have that the pairs $(\lambda_r,\lambda_t)$ and $(\lambda_{r+1},\lambda_t)$ with $t\geq r+3$ are also James partitions. Hence, the only pairs of rows of $\lambda$ which are not James  come from $(\lambda_r,\lambda_{r+1})$,   $(\lambda_{r+1},\lambda_{r+2})$,   and $(\lambda_r,\lambda_{r+2})$. 

\q Let $(y(t,u)_i)$ be a multi-sequence as in the statement of the Lemma. By Proposition 11.16(ii),  the coherence conditions are  satisfied for \\
 $(\lambda_r,\lambda_{r+1},\lambda_{r+2})$.   
 
  \q We first  deal  with the  coherence conditions involving $y(r,r+1)_{p^{v_r}}$.  For $(s,t)$ with $s<t$, and both $s$ and $t$ different from  $r$ and $r+1$, we have that $(r,r+1)$ and $(s,t)$ are related via the commuting relation (C). Now, since $(\lambda_s,\lambda_t)$ is James and $(y(s,t)_j)=0$ for such $(s,t)$ we have both sides of the relation are zero. 
  
  \q   Hence it remains  to consider the triples $(\lambda_r,\lambda_{r+1},\lambda_t)$ with $t>r+2$ and $(\lambda_q,\lambda_r,\lambda_{r+1})$ with $q<r$.  
  
  \q For the triple $(\lambda_r,\lambda_{r+1},\lambda_t)$, we have that  $y(r,r+1)_i$ is identified with $x_i$ in the relations (T1)-(T3b). Since $\len_p(\lambda_{r+1}+p^{l_{r+2}})<\val_p(\lambda_r+p^{v_r}+1)$   and $l_t<v_r$ we have that the standard solution for this triple is $\y=\z=0$ and $\x=0$ for $i\neq p^{v_r}$ with $x^{\rm{st}}_{p^{v_r}}=-1$. Hence, the coherence conditions  for our multi-sequence are satisfied.
  
  \q For the triple $(\lambda_q,\lambda_r,\lambda_{r+1})$, we have that  $y(r,r+1)_i$ is identified with  $y_j$ in the relations (T1)-(T3b). Here the standard solution for this triple is $\x=\z=0$ and $\y=0$ for $i\neq p^{v_r}$ with $y^\st_{p^{v_r}}=-1$. So, the relations associated with  this triple are satisfied by our multi-sequence.

    \q We now   deal  with the  coherence conditions involving $y(r+1,r+2)_{p^{l_{r+1}}}$.   Again for For $(s,t)$ with $s<t$, and both $s$ and $t$ different from  $r+1$ and $r+2$, we have that $(r,r+1)$ and $(s,t)$ are related only via the commuting relation (C). Now, since $(\lambda_s,\lambda_t)$ is James and $(y(s,t)_j)=0$ for such $(s,t)$, both sides in this relation are $0$.
    
   \q Hence it remains  to consider the triples $(\lambda_r,\lambda_{r+1},\lambda_t)$ with $t>r+2$ and $(\lambda_q,\lambda_r,\lambda_{r+1})$ with $q<r$.
   
   \q For the triple $(\lambda_{r+1},\lambda_{r+2},\lambda_t)$, we have that  $y(r+1,r+2)_i$ is identified with $x_i$ in the relations (T1)-(T3b). Now $l_t<v_r=v_{r+1}$ and so this triple is of the form described in Proposition 11.16(iv). Thus we get directly that the coherence conditions  for our multi-sequence are satisfied.   
   
\q For the triple $(\lambda_q,\lambda_{r+1},\lambda_{r+2})$, we have that  $y(r+1,r+2)_j$ is identified with  $y_j$ in the relations (T1)-(T3b). Here this triple has the form appearing in Lemma 8.1 and so our multi-sequence satisfies the conditions for triple.  
 
\end{proof}

\q We next consider  the case in which  $(\lambda_r,\lambda_{r+1},\lambda_{r+2})$ has the form described in Proposition 11.16(iii).

\begin{lemma}

 Suppose $l_{r+1}=l_{r+2}> v_r=v_{r+1}$, with $\len_p(\lambda_{r+1}+p^{l_{r+2}})<\val_p(\lambda_r+p^{v_r}+p^{l_{r+1}}+1)$ and $\lambda_{r+1}=p^{v_r}-1+p^{l_{r+1}}$. Then $\lambda$ is non-split. 

\q The space $E(\lambda)$ is spanned by the standard multi-sequence and a non-zero  multi-sequence $(y(t,u)_i)$ satisfying $y(r,r+1)_i=0$ for $i\neq p^{v_r}$,  \\ 
$y(r+1,r+2)_j=y(r,r+2)_k=0$ for $j,k\neq p^{l_{r+1}}$  with  
$$y(r,r+1)_{p^{v_r}}=-y(r+1,r+2)_{p^{l_{r+1}}}=y(r,r+2)_{p^{l_{r+1}}}$$
 and $y(t,u)_j=0$ for all  other choices of $t,u$ and $j$.

\end{lemma}

\begin{proof}  By Remark 12.17 we have that $(\lambda_s,\lambda_t)$ is a James partition for $s\neq r,r+1$. We show that in this case we also have that $(\lambda_r,\lambda_t)$ and $(\lambda_{r+1},\lambda_t)$, with $t\geq r+3$, are James. We have that since $\lambda_{r+2}=\hat{\lambda}_{r+2}+p^{l_{r+1}}$ with $\hat{\lambda}_{r+2}<p^{v_r}$, $l_{r+1}> v_r$ and $(\lambda_{r+2},\lambda_{r+3})$ is James then $\lambda_{r+3}\leq \hat{\lambda}_{r+2}<p^{v_r}$. Therefore all the pairs $(\lambda_r,\lambda_t)$ and $(\lambda_{r+1},\lambda_t)$ with $t\geq r+3$ are James. Hence the only pairs of rows of $\lambda$ which are not James partitions come from $(\lambda_r,\lambda_{r+1})$,   $(\lambda_{r+1},\lambda_{r+2})$,   and $(\lambda_r,\lambda_{r+2})$.

\q Let $(y(t,u)_i)$ be a multi-sequence as in the statement of the Lemma. By By Proposition 11.16(iii),  the coherence conditions are  satisfied for \\
 $(\lambda_r,\lambda_{r+1},\lambda_{r+2})$.   

  \q We first  deal  with the  coherence conditions involving  $y(r,r+1)_{p^{v_r}}$.  For $(s,t)$ with $s<t$, and both $s$ and $t$  different from  $r$ and $r+1$, we have that $(r,r+1)$ and $(\lambda_s,\lambda_t)$ are related via the commuting relation (C). Now, since $(\lambda_s,\lambda_t)$ is James and $(y(s,t)_j)=0$ for such $(s,t)$, both sides of the relation are $0$.
  
  \q  Hence it remains  to consider the triples $(\lambda_r,\lambda_{r+1},\lambda_t)$ with $t>r+2$ and $(\lambda_q,\lambda_r,\lambda_{r+1})$ with $q<r$.  
  
 \q For the triple $(\lambda_r,\lambda_{r+1},\lambda_t)$, we have that  $y(r,r+1)_i$ is identified with $x_i$ in the relations (T1)-(T3b). The standard solution for this triple is $\y=\z=0$ and $\x=0$ for $i\neq p^{v_r}, p^{l_{r+1}}, $ with $x^{\rm{st}}_{p^{v_r}}=x^{\rm{st}}_{p^{l_{r+1}}}=-1$. \q Also $(\lambda_r,\lambda_{r+1},\lambda_t)$ satisfies the properties of the triple of Proposition 11.16(iv)  and so we have that in a coherent triple, $x_i=h\x+g \px, \ {\rm{with}}\ h,g \in K$.  Setting $h=g$ we have that any value of $x_{p^{v^r}}=y(r,r+1)_{p^{v_r}}$ (with all the other values  being  $0$) is a solution for this triple and we are done.
  
  \q For the triple $(\lambda_q,\lambda_r,\lambda_{r+1})$, we have that  $y(r,r+1)_i$ is identified with  $y_j$ in the relations (T1)-(T3b). Since $(\lambda_q,\lambda_r)$ is James and $\len_p(\lambda_r+p^{v_r})=l_r$  it is a routine to check that our multi-sequence satisfies the coherence conditions for this triple.

    \q We now deal with the   coherence conditions involving  $y(r+1,r+2)_{p^{l_{r+1}}}$.  Again for $(s,t)$ with $s<t$, and both $s$ and $t$ different from  $r+1$ and $r+2$, we have that $(r,r+1)$ and $(s,t)$ are related only via the commuting relation (C). Now, since $(\lambda_s,\lambda_t)$ is James and $(y(s,t)_j)=0$ for such $(s,t)$, both sides of the relation are zero.
    
    \q Hence it remains  to consider the triples $(\lambda_{r+1},\lambda_{r+2},\lambda_t)$ with $t>r+2$ and $(\lambda_q,\lambda_{r+1},\lambda_{r+2})$ with $q<r$.    
    
    \q For the triple $(\lambda_{r+1},\lambda_{r+2},\lambda_t)$, we have that  $y(r+1,r+2)_i$ is identified with $x_i$ in the relations (T1)-(T3b). Notice that $l_t<v_r=v_{r+1}$ (since $(\lambda_r,\lambda_t)$ is James) and so we are in the case described in Proposition 11.16(iv) and we get directly that any value of $y(r+1,r+2)_{p^{l_{r+1}}}$ (with all the other values  being  $0$) satisfies the coherence relations.
    
     \q Also the triple  $(\lambda_q,\lambda_r,\lambda_{r+2})$ corresponds to the  case appearing in Lemma 8.1 and so again we get directly that the coherence conditions are satisfied for this triple.
     
     \q We leave to the reader to check that the same considerations take care the case of $y(r,r+2)_{p^{l_{r+1}}}$.

\end{proof}

\q We now consider  the case in which  $(\lambda_r,\lambda_{r+1},\lambda_{r+2})$ has the form described in Proposition 11.16(iv).

\begin{lemma}

Suppose $l_{r+2}<v_{r+1}\leq v_r$, and $\lambda_{r+1}=\hat{\lambda}_{r+1}+p^{l_{r+1}}$ with $0\leq \hat{\lambda}_{r+1}<p^{v_r}$. Then $\lambda$ is non-split if and only if there is no $q<r$ with $v_q=\len_p(\lambda_r+p^{l_{r+1}})$. 

\q In that case $E(\lambda)$ is spanned by the standard multi-sequence and a non-zero multi-sequence $(y(t,u)_i)$ satisfying 
 $y(r,r+1)_i=0$ for $i\neq p^{l_{r+1}}$ and $(y(t,u)_j)=0$ for $(t,u)\neq (r,r+1)$.

\end{lemma}

\begin{proof}

We show first that if there exists a $q<r$ with $v_q= \len_p(\lambda_r+p^{l_{r+1}})$ then $\lambda$ splits. Assume that such a $q$ exists.

\q Let $(y(t,u))_i)$ be a coherent multi-sequence. Considering $(\lambda_r,\lambda_{r+1},\lambda_{r+2})$ we obtain by Proposition 11.16(iv) that $(y(r+1,r+2)_i)=0$. Also for such $q$ we have by Lemma 8.1 that  $(\lambda_q,\lambda_r,\lambda_{r+1})$ splits. Hence we can subtract a standard multi-sequence to get a coherent multi-sequence (also denoted $(y(t,u)_i)$) in which  $(y(r,r+1)_i)=0$. Now by Proposition 9.4(i) we get directly that $\lambda$ is split.

\q We now assume that  $v_q>\len_p(\lambda_r+p^{l_{r+1}})$ for every $1\leq q<r$.   Let $(y(t,u)_i)$ be a multi-sequence as in the statement of the Lemma. By Proposition 11.16(iv), the multi-sequence satisfies the coherence conditions for the triple $(\lambda_r,\lambda_{r+1},\lambda_{r+2})$.  By our usual arguments for the pairs related via the commuting relation (C) we obtain an equation in which both sides are zero and the coherence conditions are satisfied. 

 \q Hence it remains to consider  the triples  $(\lambda_q,\lambda_r,\lambda_{r+1})$ with $q<r$ and $(\lambda_{r+1},\lambda_{r+2},\lambda_t)$ with $t>r+2$.

\q The triple $(\lambda_q,\lambda_r,\lambda_{r+1})$ corresponds to the case appearing in Lemma 8.1 and since $v_q>\len_p(\lambda_r+p^{l_{r+1}})$ for every $q<r$ we have by Lemma 8.1 that the coherence conditions are satisfied. 

\q Also for the triple $(\lambda_r,\lambda_{r+1},\lambda_t)$, $y(r,r+1)_i$ is identified with $x_i$ in the relations (T1)-(T3b). But as $l_t\leq l_{r+2}$ and any value of  $y(r,r+1)_{p^{v_r}}$  (and all other values  $0$) is coherent for  $(\lambda_r,\lambda_{r+1},\lambda_{r+2})$ then we get directly that the coherence conditions are satisfied also for  $(\lambda_r,\lambda_{r+1},\lambda_t)$.

 \end{proof}
 
 \q The final case corresponds to the situation in which   $(\lambda_r,\lambda_{r+1},\lambda_{r+2})$ has the form described in Proposition 11.16(v). 

\begin{lemma}

Suppose $l_{r+2}=v_r>v_{r+1}$ with $l_{r+1}< \val_p(\lambda_r+p^{v_r}+1)$, $\lambda_{r+1}=\hat{\lambda}_{r+1}+p^{l_{r+1}}$ with $0\leq \hat{\lambda}_{r+1}<p^{v_r}$, and $\lambda_{r+2}=\hat{\lambda}_{r+2}+p^{v_r}$ with  $0\leq \hat{\lambda}_{r+2}<p^{v_{r+1}}$. Then $\lambda$ is non-split. 

\q The space $E(\lambda)$ is spanned by the standard multi-sequence and a non-zero multi-sequence $(y(t,u)_i)$  satisfying 
 $y(r+1,r+2)_j=y(r,r+2)_k=0$ for $j,k\neq p^{v_r}$ with 
 
 $$y(r+1,r+2)_{p^{v_r}}+y(r,r+2)_{p^{v_r}}=0$$
 
and $(y(t,u)_j)=0$ for $(t,u)\neq (r+1,r+2),(r,r+2)$.

\end{lemma}

\begin{proof}

The considerations here are identical with those of Lemma 12.23 and the proof follows in the same way.

\end{proof}

\subsection{Summary}

\q We gather together now our results from Corollary 7.13 and Propositions 9.5, 12.3 and 12.16 in the following theorems to describe all the non-split partitions. We remind to the reader that for a partition $\lambda=(\lambda_1,\dots,\lambda_n)$ we have $v_i=\val_p(\lambda_i+1)$ and $l_i=\len(\lambda_i)$ for $1\leq i\leq n$. Also, a partition $\lambda$ is James if and only if $v_i>l_{i+1}$ for $1\leq i\leq n-1$.

\begin{theorem} A James partition is non-split if and only if it has length $n\geq 2$. The dimension of the extension group is bounded by $n-1$, and is described explicitly in Corollary 7.13.
\end{theorem}

\begin{theorem}  Suppose that $p\geq 3$ and $\lambda=(\lambda_1,\ldots,\lambda_n)$ is a non-James partition. Let $1\leq r<n$ be minimal such that $(\lambda_r,\lambda_{r+1})$ is not James. Then $\lambda$ is non-split if and only if one of the following holds.

(i) We have $r<n-1$, $(\lambda_{r+1},\lambda_{r+2})$ is not James, $(\lambda_r,\lambda_{r+1})$  and $(\lambda_{r+1},\lambda_{r+2})$ are the only non-James consecutive pairs of rows of $\lambda$ and $(\lambda_r,\lambda_{r+1},\lambda_{r+2})$ is non-split.

(iii) We have $r<n-1$, $(\lambda_r,\lambda_{r+1})$ is the only non-James consecutive pair of rows of $\lambda$,  $(\lambda_r,\lambda_{r+1})$ is split and $(\lambda_r,\lambda_{r+1},\lambda_{r+2})$ is non-split.

(iii) We have $r<n$, $(\lambda_r,\lambda_{r+1})$ is the only non-James consecutive pair of rows of $\lambda$,  $(\lambda_r,\lambda_{r+1})$ is pointed and there is no $1\leq q<r$ with 
$v_q=\len_p(\lambda_r+p^{l_{r+1}})$.

(iv) We have $r<n-2$, $(\lambda_{r+3},\ldots,\lambda_n)$ is James and the quadruple \\
$(\lambda_r,\lambda_{r+1},\lambda_{r+2},\lambda_{r+3})$ satisfies the following conditions:

 $\lambda_r=(p^{l_{r+1}+1}-1)-p^{v_r}+p^{l_{r+1}+1}\lambda_r'$, for some $\lambda_r'\geq 0$; and 

 $\lambda_{r+1}=(p^{v_r+1}-1)-p^{v_r}+p^{v_r+1}\lambda_{r+1}'$, $(\lambda_{r+1})_{\v_r}\neq 0$ and  $\lambda_{r+1}'\geq 0$; and 

$\lambda_{r+2}=p^{v_r}-1+p^{v_r}$; and 

  $\lambda_{r+3}=\hat{\lambda}_{r+3}+p^{v_r}$, with $0\leq \hat{\lambda}_{r+3}<p^{v_r}$.
  
  \smallskip

 \q  In all cases the  dimension of the space of extensions  is $1$.

\end{theorem}

\begin{theorem}  Suppose that $p=2$ and $\lambda=(\lambda_1,\ldots,\lambda_n)$ is a non-James partition. Let $1\leq r<n$ be minimal such that $(\lambda_r,\lambda_{r+1})$ is not James. Then $\lambda$ is non-split if and only if one of the following holds.

(i) We have $r<n-1$, $(\lambda_{r+1},\lambda_{r+2})$ is not James, $(\lambda_r,\lambda_{r+1})$  and $(\lambda_{r+1},\lambda_{r+2})$ are the only non-James consecutive pairs of rows of $\lambda$ and $(\lambda_r,\lambda_{r+1},\lambda_{r+2})$ is non-split.

(ii)  We have  $r<n-1$,  $(\lambda_r,\lambda_{r+1})$ is the only non-James consecutive pair of rows of $\lambda$,    $(\lambda_r,\lambda_{r+1})$ is a split pair, $(\lambda_r,\lambda_{r+1}, \lambda_{r+2})$ is non-split and if $r<n-2$ then $l_{r+3}<l_{r+2}$.

(iii) We have  $r<n$  $(\lambda_r,\lambda_{r+1})$ is the only non-James consecutive pair of rows of $\lambda$,  $(\lambda_r,\lambda_{r+1})$ is pointed and there is no $q<r$ with $v_q=\len_p(\lambda_r+p^{l_{r+1}})$.

  \smallskip

 \q  In all cases the  dimension of the space of extensions  is $1$.
\end{theorem}

\bs\bs\bs\bs


\section{Further Remarks}

\q In this section we highlight some consequences of our main result and give counterexamples to a conjecture of D. Hemmer in each characteristic and answer a problem that he states in the last section of  \cite{Hem}.

\begin{corollary}

Let $p\geq 3$ and $\lambda=(\lambda_1,\dots,\lambda_n)$ a partition of degree $d$ with at least three parts. Then $H^1(\Sigma_{pd}, \Sp(p\lambda))=0$.

\end{corollary}

\begin{proof}

This follows immediately by our main results, Theorem 12.29-31. However, we give here an alternative proof using some of our results from the earlier sections. First notice that since $(p\lambda_1,p\lambda_2)$ is not James we get by Proposition 9.4(i) that it is enough to show that $\mu=(p\lambda_1,p\lambda_2,p\lambda_3)$ splits.

If $(p\lambda_1,p\lambda_2)$ is split then by Lemma 10.2,  $\mu$ splits, since $(p\lambda_1+1,p\lambda_2)$ is not James.

If $(p\lambda_1,p\lambda_2)$ is pointed then by Lemma 11.3, $\mu$ splits, since $\val_p((p\lambda_1+1)+1)=0$ and $\lambda_2>0$. So we are done.

\end{proof}

\begin{remark} \rm 

The above result is not true for $n=2$. More precisely, by Lemma 5.12 we have that for a partition $\lambda=(\lambda_1,\lambda_2)$ of degree $d$ with $\lambda_2>0$, the first cohomology group $H^1(\Sigma_{pd}, \Sp(p\lambda))\neq0$ if and only if $p\lambda$ is pointed, i.e. $\lambda_2=p^t$ for some $t> 0$. In this case we have that the first cohomology group is 1 dimensional. Of course this also follows from the work of Erdmann \cite{Erdmann}.

\end{remark}

\begin{remark} \rm 

In  \cite{Hem}, Hemmer  states  that for a partition $\lambda=(\lambda_1,\dots,\lambda_n)$ of degree $d$ there is an isomorphism of vector spaces $H^1(\Sigma_{p^cd}, \Sp(p^c\lambda))\cong H^1(\Sigma_{p^{c+1}d}, \Sp(p^{c+1}\lambda))$ for $c\geq1$, \cite{Hem} Theorem 6.5.8. 

\q If $\lambda$ has length $2$  this follows by our remark above, observing first that for $c\geq 1$ the partition $p^c\lambda$ is pointed if and only if $p\lambda$ is pointed.  For $n\geq 3$, Corollary 13.1 gives a much stronger result, namely that  \\
$H^1(\Sigma_{pd}, \Sp(p\lambda))=0$ for all the partitions $\lambda$ with at at least three parts.

\q We note  that though Hemmer's statement is true his argument  is  problematic. More precisely, \cite{Hem} Corollary 6.4 is not correct and it is heavily used in what follows.  This applies also to other parts of the paper  (see for example Lemma 6.5.2). The source of these difficulties is his statement that for two partitions $\lambda$ and $\mu$  with length at most $n$ and  $\lambda>\mu$ one has that $\lambda_i-\lambda_{i+1} \geq \mu_i-\mu_{i+1} $ for every $1\leq i\leq n-1$. This is not true as one may see in simple examples.  This  statement  is used in the argument given for  Corollary 6.4. The reader can also check that for $p=3$ and $G=\GL_5(K)$ the partitions $\lambda=(23,11,2)$ and $\mu=(18,9,9)$ satisfy the conditions of Corollary 6.4 and that $\lambda-\mu-2\rho$ is a weight of the Steinberg module $\St_1$ contrary to the statement of Corollary 6.4.

\end{remark}

{\bf{13.4 Counterexamples to Hemmer's Conjecture}}. In the same paper, the author conjectures, \cite{Hem}, Conjecture 8.2.1, that if  $p\geq 3$ and $\lambda$ is a partition of degree $d$ with $\lambda\neq (d)$ such that the simple module $L(\lambda)$ for $\GL_N(K)$ (with $N\geq d$)  appears as a composition factor of  $S^d(E)$, then  
 $$H^1(\Sigma_{d}, \Sp(\lambda))\neq0.$$

Here we give a counterexample   to this conjecture in each characteristic  for   $n\geq 2$. By Krop's description of the composition factors of the symmetric powers in \cite{K} (see also \cite{DG3}, Remark 4.13) we have that  for \\
$p\lambda=p(p-1,\dots,p-1)$, the partition with all  $n$ parts equal to $p(p-1)$,  the simple module $L(\lambda)$ is  a composition factor of $S^d(E)$, where \\$d=(p^2-p)(n-1)$. 

\q Now by Corollary 13.1 we have that $H^1(\Sigma_{pd}, \Sp(p\lambda))=0$ for $n\geq3$.  Moreover for $n=2$,  since $p\geq 3$, we have  $\lambda_2\neq p^t$ and so by Remark 13.2 we have again that $H^1(\Sigma_{pd}, \Sp(p\lambda))=0$. Thus  $p\lambda=p(p-1,\dots,p-1)$ is a counterexample to this conjecture for   $n\geq 2$ and $p\geq3$.

\bigskip

{\bf {13.5 Another Problem}}. We conclude  by answering a question  raised in  the last section of Hemmer's paper.  In \cite{Hem}, Lemma 8.2.3, the author makes the following observation, via  James's description of the fixed points  of the Specht modules. Given a partition $\lambda=(\lambda_1,\dots,\lambda_n)$ of degree $d$ and $a\geq\lambda_1$ with $\len_p(\lambda_1)<\val_p(a+1)$ (i.e. $(a,\lambda_1)$ is James) one has that $H^0(\Sigma_{d}, \Sp(\lambda))\cong H^0(\Sigma_{a+d}, \Sp(\mu))$, where $\mu=(a,\lambda_1,\dots,\lambda_n)$. It is then asked,  \cite{Hem}, Problem 8.2.4, whether  one has such an isomorphism for cohomology in higher degree.

\q Note that in the case of $i=1$, using Lemma 7.7 and Proposition 7.11 (or directly by Corollary 7.13), one can verify that  
$$\dim H^1(\Sigma_d, \Sp(p^n-1,p^{n-1}-1,\dots,p^2-1,1))=n-1$$
and
 $$\dim H^1(\Sigma_{d'}, \Sp(p^{n+1}-1,p^n-1,\dots,p^2-1,1))=n$$
(where $d$ is the degree of $(p^n-1,p^{n-1}-1,\dots,p^2-1,1)$  and $d'$ is the degree of $(p^{n+1}-1,p^n-1,\dots,p^2-1,1)$).

\newpage

\centerline{\bf Appendix I: Comparing Cohomology}

\bs

\q We fix integers $n>0$, $d \geq 0$, with $n\geq d$. We shall use the notation of \cite{qSchur}, taking $q=1$ where appropriate.

\q We have the Schur algebra $S(n,d)$ and an idempotent $e\in S=S(n,d)$ with $eSe$ identified with the group algebra $K \Sigma_d$, as in \cite{EGS}, and \cite{qSchur}. For a partition $\lambda$ of $d$ we write $M(\lambda)$ for the corresponding permutation module and $Y(\lambda)$ for the corresponding Young module. The argument comparing extensions for  $G(n)$-modules and $K \Sigma_d$-modules is based on Section 10 of \cite{HTT}.

\q For a finite dimensional algebra $A$, we write $\mod(A)$ for the category of finite dimensional left $A$-modules.  

\q We have the Schur functor $f:\mod(S)\to \mod(eSe)$, given on objects by $fV=eV$, $V\in \mod(S)$,  as in \cite{EGS}, Chapter 6.  Then $f$ is exact and we have the (right exact) functor 
 $g:\mod(eSe)\to \mod(S)$, given on objects by  $gX=Se\otimes_{eSe} X$, as in \cite{HTT}, Chapter 6.  We have $\Hom_S(gX,Y)=\Hom_{eSe}(X,fY)$, for $X\in \mod(eSe)$, $Y\in \mod(S)$, by Frobenius reciprocity.  We write $L^i g:\mod(eSe)\to \mod(S)$ for the left derived functors and simply $Lg$ for the first derived functor of $g$   Now $g$ takes projective modules to projective modules so the factorisation $\Hom_{eSe}(-,  fY)=\Hom_S(-,Y)\circ g$ gives rise to a Grothendieck spectral sequence, with second page $\Ext^i_S(L^jg X,Y)$,  converging to  $\Ext^*_{eSe}(X,fY)$, for $X\in \mod(eSe)$, $Y\in \mod(S)$.  In particular we have the  $5$-term exact sequence
 \begin{align*}
\Ext^2_{eSe}(X,fY)&\to \Ext^2_S(gX,Y)\to \Hom_S(Lg X,Y)\cr
&\to \Ext^1_{eSe}(X,fY)\to \Ext^1_S(gX,Y)\to 0.
\end{align*}

\q For $\alpha=(\alpha_1,\ldots,\alpha_n)\in \Lambda(n,d)$ we write $S^\alpha E$ for $S^{\alpha_1} E\otimes \cdots \otimes S^{\alpha_n} E$, the tensor product of symmetric powers of the natural module $E$. For $\lambda\in \Lambda^+(n,d)$ we write  $I(\lambda)$ for the injective hull of the simple module $L(\lambda)$ with highest weight $\lambda$,  in the polynomial category.

\q We shall need  the following.

\bs

\bf Lemma\q\sl    For $\lambda\in \Lambda^+(n,r)$ we have  $g Y(\lambda) = I(\lambda)$ and $g M(\lambda)= S^\lambda E$.

\rm
\bs

\begin{proof}      By unitriangularity of decompositions of the modules $S^\lambda E$ as a direct sum of $I(\mu)$s (see e.g. \cite{qSchur}, 2.1 (8))  and the Krull-Schimdt Theorem it is enough to show that $g M(\lambda) = S^\lambda E$.   For this we argue as follows.   For $\alpha\in \Lambda(n,d)$   we have $fS^\alpha E = M(\alpha)$  (see e.g., Section 9, (1)(i) of \cite{HTT}) and hence  
$$\Hom_S(g M(\lambda), S^\alpha  E)= \Hom_{eSe}( M(\lambda), fS^\alpha E)=\Hom_{eSe}(M(\lambda), M(\alpha))$$
 by Frobenius reciprocity. But this is independent of characteristic (since $M(\lambda)$ and $M(\alpha)$ are permutation modules). Hence, by \cite{qSchur}, 2.1 (8) the weight space dimension  $\dim (g M(\lambda))^\alpha$ is independent of characteristic. 
But in characteristic $0$ the functors $f$ and $g$ are inverse equivalences of categories  and $f S^\lambda E = M(\lambda)$ so we have $\dim g M(\lambda) = \dim S^\lambda E$ (in all characteristics). 

\q Now for $\lambda= (1^d)$ we have $M(1^d)=eSe$ so that 
$$g M(1^d) = Se \otimes_{eSe} eSe=Se=E^{\otimes d}.$$
   Now for arbitary $\lambda$ we have the natural epimorphism $E^{\otimes d} \to S^\lambda E$ and so the epimorphism $eSe\to M(\lambda)$ and so the epimorphism $E^{\otimes d} \to g M(\lambda)$. Now we have a commutative diagram 

$$
\begin{matrix}
Se\otimes_{eSe} M(1^d)&\to & E^{\otimes d}\cr
\downarrow& & \downarrow\cr
Se\otimes_{eSe} M(\lambda) & \to & S^\lambda E

\end{matrix}
$$
where the horizontal maps are induced from the inclusions $M(1^d) \to E^{\otimes d}$ and $M(\lambda) \to S^\lambda E$.  The top map is an isomorphism and the right map is surjective, hence the bottom map is surjective. Since $\dim Se\otimes_{eSe} M(\lambda) = \dim S^\lambda E$ we have an isomorphism $g M(\lambda) \to S^\lambda E$.

\end{proof}

\q For  $X=M(\mu)$ and  $Y=\nabla(\lambda)$  the  $5$-term exact sequence  takes the form
 \begin{align*}
\Ext^2_{\Sigma_d}(M(\mu),\Sp(\lambda))&\to \Ext^2_{G(n)}(S^\mu E,\nabla(\lambda))\to \Hom_{G(n)}(Lg M(\mu),\nabla(\lambda))\cr
&\to \Ext^1_{\Sigma_d}(M(\mu),\Sp(\lambda))\to \Ext^1_{G(n)}(S^\mu E,\nabla(\lambda))\to 0.
\end{align*}

\q On the other hand if we take $X\in \mod(eSe)$ and $Y$ an injective $S$-module then  the spectral sequence degenerates and we have \\
$\Ext^i_{eSe}( X,fY)=\Hom_S(L^i gX,Y)$. Hence we have $LgX=0$ provided that $\Ext^1_{eSe}( X,fY)=0$, for all injective modules $Y$. Now an injective $S$-module is a direct summand of a direct sum of copies of the symmetric powers $S^\alpha E$, $\alpha\in \Lambda^+(n,d)$ and so we have $LgX=0$ provided that $\Ext^1_{\Sigma_d}(X,M(\tau))=0$ for all $\tau\in \Lambda^+(n,d)$. For $X=M(d)$, the trivial module, we have  $\Ext^1_{\Sigma_d}(X,M(\tau))=H^1(\Sigma_d,M(\tau))$, which by Shapiro's Lemma, is $H^1(\Sigma_\tau,K)$, where $\Sigma_\tau=\Sigma_{\tau_1}\times \Sigma_{\tau_2}\times\cdots$ (and $\tau=(\tau_1,\tau_2,\ldots)$).  Now if $p>2$ we have $H^1(\Sigma_\tau,K)=0$ and hence $Lg M(d)=0$. Hence, from the above $5$-term exact sequence we have $\Ext^1_{\Sigma_d}(M(d),\Sp(\lambda))=\Ext^1_{G(n)}(S^d E,\nabla(\lambda))$.  Of course in all characteristics we have a surjection $\Ext^1_{\Sigma_d}(M(d),\Sp(\lambda))\to \Ext^1_{G(n)}(S^d E,\nabla(\lambda))$, from the $5$-term exact sequence.

\bs

\bf Proposition \q\sl For $\lambda\in \Lambda^+(n,d)$ we have:

(i) $\Hom_{G(n)}(S^d E,\nabla(\lambda))=H^0(\Sigma_d,\Sp(\lambda))$; and 

(ii) $\dim H^1(\Sigma_d,\Sp(\lambda))\geq \dim \Ext^1_{G(n)}(S^d E,\nabla(\lambda))$, with equality if $p\neq 2$.

\rm\bs

\q The first assertion comes directly from the Lemma  and Frobenius reciprocity and the second from the discussion above.

\newpage

\centerline{\bf Appendix II: The Relations}

\bs

\q For the convenience of the reader, we collect together the relations that are used a great deal  throughout the paper.

\bs

\bf Definition \sl \,  Let $(a,b)$ be a two part partition. By an {\em extension sequence} for $(a,b)$ we mean a sequence $(x_i)=(x_1,\ldots,x_b)$ of elements of $K$ such that 
$${a+i+j\choose j}x_i={i+j\choose j}x_{i+j}.$$

\rm

\bs

\bf Definition \sl \,   We fix a three part partition $\lambda=(a,b,c)$. By a   {\em coherent triple}  of extension sequences for $\lambda$ we mean a triple  \\
 $((x_i)_{1\leq i\leq b}, (y_j)_{1\leq j\leq c}, (z_k)_{1\leq i\leq c})$ of extension sequences satisfying the following relations:

\begin{align*}
&{\rm (T1)}\     {{a+i+k}\choose{k}}x_i={{a+i+k}\choose{i}}z_k,  \hskip 85pt 1\leq i\leq b, 1\leq k\leq c; \\ \cr
&{\rm (T2)}\    {{a+k}\choose{k}}y_j={{b+j}\choose{j}}z_k, \hskip 110pt 1\leq j,k\leq c,   j+k\leq c; \\ \cr
&{\rm (T3a)}\    {{a+i}\choose{i}}y_j=\sum_{s=0}^{i-1}{{b+j-i}\choose{j-s}}{{a+i}\choose{s}}x_{i-s}+{{b+j-i}\choose{j-i}}z_i,  \hskip 0pt 1\leq  i\leq j\leq c;\\ \cr
&{\rm (T3b)}\   {{a+i}\choose{i}}y_j=\sum_{s=0}^{j}{{b+j-i}\choose{j-s}}{{a+i}\choose{s}}x_{i-s},  \hskip 30pt 1\leq j\leq c,  j<i\leq b+j. \\ \cr
\end{align*}

\rm

In any coherent multi-sequence $(y(t,u)_i)$ for a partition $\lambda=(\lambda_1,\ldots,\lambda_n)$ we have the following commuting relations:

\begin{align*} {\rm (C)} \hskip 10pt {\lambda_s+i \choose i} y(q,r)_j&={\lambda_q+j\choose j} y(s,t)_i, \hskip 170pt \cr
&\hbox{ for } 1\leq q,r,s,t\leq n \hbox{ distinct with } q<r, s<t.
\end{align*}

\q In particular for  $1\leq  r< s< n$ with $s>r+1$ we have

$$
(C')  \   {\lambda_s+j\choose j} y(r,r+1)_i={{\lambda_r+i}\choose{i}}y(s,s+1)_j,\  1\leq  i\leq \lambda_{r+1}, 1\leq j\leq \lambda_{s+1}.
$$

\bs\bs\bs\bs



\begin{thebibliography}{99}


\bibitem{ARS}{Maurice Auslander, Idun Reiten and Sverre O. Smal\o, \em{Representation Theory of Artin Algebras}, Cambridge Tracts in Mathematics {\bf 36}, CUP 1995}



















\bibitem{CPS}{Edward Cline, Brian Parshall and Leonard Scott, \emph{Cohomology, hyperalgebras and representations}, Journal of Algebra, {\bf63},  98-123, (1980)}



\bibitem{DD} {R. Dipper and S. Donkin, \emph{Quantum ${\rm SL}_n$},  Proc. Lond. Math. Soc. {\bf 63}, 165-211, (1991).}





\bibitem{D4}{S. Donkin, \emph{Rational Representations of Algebraic Groups: Tensor Products and Filtrations}, Lecture Notes in Math.  1140, Springer  1985.}







\bibitem{HTT} {S. Donkin, \emph{Tilting modules for algebraic groups and finite dimensional algebras}, in: Handbook of Tilting Theory, London Mathematical Society Lecture Note Series {\bf 332}, pp. 215-257. Edited by L. A. H\"ugel, D. Happel and H. Krause, Cambridge University Press 2007.}


\bibitem{DStd}{S. Donkin, \emph{Standard Homological Properties  for Quantum  ${\rm GL}_n$}, J. Algebra, {\bf 181}, 235-266, 1996.}


\bibitem{qSchur} {S. Donkin, \emph{The $q$-Schur algebra},  LMS Lecture Notes 253, Cambridge University Press 1998.}








\bibitem{DG3} {S. Donkin and H. Geranios, \emph{Composition factors of a tensor product of truncated  symmetric powers}, Advances in Mathematics {\bf 285}, 394-433, 2015.}







\bibitem{Erdmann} {K. Erdmann, \emph{$Ext^1$ for Weyl modules of $SL_2(K)$}, Math. Z  {\bf218}, 447-459 (1995).}

\bibitem{EGS}{K. Erdmann, J. A. Green and M. Schocker , \emph{Polynomial Representations of ${\rm GL}_n$,  Second Edition with an Appendix on Schenstead Correspondence and Littelmann Paths}, Lecture Notes in Mathematics 830, Springer 2007.}




\bibitem{Haboush}{W. Haboush, \emph{Central differential operators on split semi-simple
groups over fields of positive characteristic}
In: S\'eminaire
d'Alg${\grave {\rm e}}$bre P. Dubreil et M.-P. Malliavain, Proc. 1979," 
{Lecture Notes in Mathematics \bf795}, pp.35-83, Springer 1980, Berlin, Heidelberg, New York.}








\bibitem{Hem}{D. J Hemmer. \emph{Cohomology and generic cohomology of Specht modules for the symmetric group}, Journal of Algebra, {\bf 322} (2009), 1498-1515.}


\bibitem{HN}{D. J. Hemmer and D. K. Nakano, \emph{Specht filtrations for Hecke algebras of type $A$}, J. London Math. Soc. {\bf 69}, (2004), 623-638.}

\bibitem{HN2}{D. J. Hemmer and D. K. Nakano, \emph{On the cohomology of Specht modules}, Journal of Algebra, {\bf 306} (2006), 191-200.}



\bibitem{James}  {G. D. James,  \emph{The Representation Theory of the Symmetric Groups}, Lecture Notes in Mathematics 682, Springer 1970.}






\bibitem{RAG}{Jens Carsten Jantzen, \emph{Representations of Algebraic Groups}, second ed., Math. Surveys Monogr., vol 107, Amer. Math. Soc., 2003.}














\bibitem{KN}{Alexander S. Kleshchev and Daniel K. Nakano, \emph{On comparing the cohomology of the general linear and symmetric groups}, Pacific Journal of Mathematics, {\bf 201}, 339-355,  2001.}



\bibitem{K}{L. Krop, \emph{On the representations of the full matrix semigroup on homogeneous polynomials, I}. J.   Algebra, {\bf 99}, 284-300, 1986.}




\bibitem{Parker}{Alison E. Parker, \emph{Higher extensions between modules for ${\rm SL}_2$},  Advances in Mathematics, {\bf 209}, 381-405,  (2007).}

 
 \bibitem{rotman}{Joseph J. Rotman, \emph{An Introduction to Homological Algebra}, Second Edition, Springer 2008.}
 
 
 
 
 \bibitem{sullivansimply} {J. B. Sullivan, \emph{Simply connected groups, the hyperalgebra, and Verma's conjecture}, Amer. J. Math. {\bf100}, 1015-1019, (1978).}
 
 
 
 
 \bibitem{tata} {W. van der Kallen, \emph{Lectures on Frobenius Splittings and $B$-modules}, Tata Institute Lecture Notes on Mathematics and Physiscs, Springer 1993 }
 





\bibitem{Web} {Christian Weber, \emph{First degree cohomology of Specht modules over fields of odd prime characteristic}. J. Algebra {\bf 392}, 23-41, 2013.}


\end{thebibliography}
\end{document}